\documentclass[10pt]{amsart}
	 \usepackage[all]{xy}
	 \usepackage{tikz}
	\usetikzlibrary{shapes.geometric,intersections,3d,quotes,angles}
	\usepackage[pagebackref]{hyperref}
	\usepackage{sagetex}
	\usepackage{mathrsfs}
	\usepackage{graphicx}
        \usepackage{float}
\newtheorem*{thm*}{Theorem}
 \newtheorem*{prop*}{Proposition}
\newtheorem*{rmq}{\textit{Remark}}

\newtheorem{theorem}{Theorem}[subsection]
 \newtheorem{lemma}[theorem]{Lemma}
  \newtheorem*{lemma*}{Lemma}
\newtheorem{proposition}[theorem]{Proposition}
\newtheorem{corollary}[theorem]{Corollary}
\newtheorem{definition}[theorem]{Definition}
 \newtheorem{example}[theorem]{Example}

 \newtheorem{remark}[theorem]{\textit{Remark}}
\numberwithin{table}{subsection}
\def\lset{\{}  
	\def\rset{ \}}  
	\def\set#1{\lset#1\rset} 
	\def\sett#1#2{\lset #1 \mid  #2 \rset}  
 \def\half{\frac 12}
 \renewcommand\hom{\mathop{\rm Hom}\nolimits}

 \def\gr{\text{\rm Gr}}
\def\la{\langle}
\def\ra{\rangle}
\def\espinnm{\operatorname{{\mathsf  Nm}_{\rm spin}^\varepsilon}}
\DeclareMathOperator{\Spec}{Spec}
\DeclareMathOperator{\Proj}{Proj}

\DeclareMathOperator{\Pic}{Pic}
\DeclareMathOperator{\NS}{NS}

\DeclareMathOperator{\Gal}{Gal}
\DeclareMathOperator{\Br}{Br}
\DeclareMathOperator{\ch}{char}
\DeclareMathOperator{\disc}{disc}
\DeclareMathOperator{\Frob}{Frob}
\DeclareMathOperator{\Tr}{Tr}
\DeclareMathOperator{\tH}{H}
\newcommand{\bC}{{\mathbb{C}}}
\newcommand{\bbF}{{\mathbb{F}}}

\newcommand{\bP}{{\mathbb{P}}}
\newcommand{\bQ}{{\mathbb{Q}}}
\newcommand{\bR}{{\mathbb{R}}}

\newcommand{\bZ}{{\mathbb{Z}}}
 \usepackage[cal=boondoxo]{mathalfa}

	\newcommand\cF{{\mathcal F}}

	\newcommand\cO{{\mathcal O}}

	\newcommand\cM{{\mathcal M}}
	
	\newcommand\cS{{\mathcal S}}

  \DeclareFontEncoding{LS1}{}{}
    \DeclareFontSubstitution{LS1}{stix2}{m}{n}
              \DeclareSymbolFont{symbols2}{LS1}{stixfrak} {m} {n}
		\DeclareMathSymbol{\operp}{\mathbin}{symbols2}{"A8}
  		
\def\mapright#1{\mathop{\vbox{\ialign{
                                ##\crcr
    ${\scriptstyle\hfil\;\;#1\;\;\hfil}$\crcr
 \noalign{\kern2pt\nointerlineskip}
    \rightarrowfill\crcr}}\;}}
\def\into{\hookrightarrow}

\def\del#1#2{\partial #1/\partial #2}
\def\comp{\raise1pt\hbox{{$\scriptscriptstyle\circ$}}}
\def\id{\text{\rm id}} 

\begin{document}

\title[A Class of Elliptic Surfaces  in Weighted Projective Space]
{A Remarkable Class of Elliptic Surfaces of Amplitude $1$  in Weighted Projective Space 
\\  
} 
\author{Gregory Pearlstein \& Chris Peters\\
With an appendix by Wim Nijgh}
\date{\today}
\begin{abstract}
Surfaces of amplitude 1 in ordinary projective space are of general type, but this need not be the case in weighted projective spaces. 
Indeed, there are 4 classes of quasi-smooth weighted hypersurfaces in $\bP(1,2,a,b)$ of amplitude 1 with an elliptic pencil cut out by hyperplanes. 
Their moduli spaces are constructed, the monodromy of their universal families is determined as well as their period maps which turn out
to be generally immersive. For those that are not, a mixed Torelli theorem holds. 
 We added an application  to 
 certain compactifications of moduli  spaces of surfaces of general type with $K^2=1$, $p_g=2$ and $q=0$
 as a follow up of \cite{GPSZ},   as well as detailed  \textsc{SageMath}-calculations. The appendix    written by Wim Nijgh shows that the general 
 member of the type (a) and type  (b) elliptic family has "trivial" Picard lattice, i.e. is spanned by fiber components and a multisection. 
\par

\noindent MSC Class: 14Jxx; 32G20

\end{abstract}


\maketitle

\section*{Introduction}

Hypersurfaces and, more generally, complete intersections in weighted
projective spaces are basic entries in the geography of algebraic varieties. 
In particular,  M.~Reid \cite{ReidCanSIng} gave a list of $95$ families of
weighted projective K3 hypersurfaces with Gorenstein singularities.

There are several instances where a moduli space  of  a class of surfaces can
be described in terms of weighted complete intersections. We mention the
Kunev surfaces \cite{Kynev,Todorov2} which are  bidegree $(6,6)$ complete
intersections  in $\bP(1,2,2,3,3)$, and certain Horikawa surfaces studied in
\cite{PearlZhang} which are hypersurfaces of degree $10$ in $\bP(1,1,2,5)$.

\subsection*{1} We recall some properties of hypersurfaces in weighted
projective spaces and refer to \cite{DolWPV,IFWings} for details.
If $X$ is a degree $d$ hypersurface in weighted projective space
$\bP(a_0,\dots,a_n) $ its \textbf{\emph{type}} is the symbol
$(d,[a_0,\dots,a_n])$.
  Following \cite{IFWings}, the integer $\alpha(X)=d -(a_0+\dots+a_n)$ is called the
\textbf{\emph{amplitude}} of $X$. If $n=3$ and $X$ would be smooth,
$\alpha(X)<0$,  $\alpha(X)=0$, respectively $\alpha(X)>0$ corresponds to $X$
being a rational or ruled surface, a K3-surface, or a surface of general type
respectively.  Since weighted projective spaces and their hypersurfaces
therein in general are singular, the amplitude no longer measures their place
in the classification. B.~Hunt and R. Schimmrigk~\cite{hunt} found a striking
example of this phenomenon: the degree $66$ 
Fermat-type surface $x_0^{66}+ x_1^{11}+ x_2^{3 }+x_3^{2}=0$ in $\bP(1,6,22,33)$
of amplitude $66-(1+6+22+33)=4$ turns out to be an elliptic K3-surface. In
fact it is isomorphic to the unique K3 surface with  the cyclic group of order
$66$ as its automorphism group described by H. Inose~\cite{inose}. 
J. Koll\'ar~\cite[Sect. 5]{k-bmy} found  several families of hypersurfaces in
weighted projective space with positive amplitude which have the same rational
cohomology as projective space.  In the surface case one then obtains rational
surfaces with positive amplitude, for example the surface in
$\bP(a_0,a_1,a_2,a_3)$ given by
 $x_0^{d_0}x_1+ x_1^{d_1}x_2+x_2^{d_2}x_3+x_3^{d_3}x_0$  where
$(d_0,d_1,d_2,d_3)= (4, 5, 6, 7)$ and $(a_0, a_1, a_2,a_3)=(174,143,124,95)$
which has degree $839$,  amplitude  $303$  and Hodge numbers
$p_g=0,h^{1,1}=2$.

In this note we  restrict our discussion to families of surfaces in
weighted projective $3$-space whose amplitude is $1$ and which are not of
general type. The classification of surfaces suggests looking  for conditions
that give K3 surfaces or elliptic surfaces. 

Except in Section~\ref{sec:KSBA} where we apply the results of previous sections, 
we  only consider quasi-smooth surfaces, i.e., surfaces whose
only singularities occur where the weighted projective space has singularities.
This  ensures that a  surface of amplitude $\alpha$ has canonical sheaf
$\cO(\alpha)$  (see Remark~\ref{rmk:OnAmpl}) which simplifies many calculations.
Assuming that $\alpha=1$ and that $p_g=1$ then leads to degree $d=a+b+4$
surfaces in $\bP(1,2,a,b)$.\footnote{Note that a quasi-smooth surface in
 $\bP(1,1,a,b)$  with amplitude $1$ has $p_g\ge 2$ since $H^0(\cO(1))$
corresponds to the polynomials of degree $1$.} This  restricts the
possibilities to just four cases. Two give properly elliptic surfaces and two
give K3 surfaces. 

\begin{prop*}[=Proposition~\ref{prop:OurExamples} and
Proposition~\ref{prop:Main1}] The only quasi-smooth hypersurfaces $X$ of type
$(d,[1,2,a,b])$ with $a,b$ co-prime odd integers and such that $d= a+b+4$ are: 
 \begin{enumerate}
\item  $(14,[1,2,3,7])$, for  example
  $ x_0^{14}+ x_1^7 + x_2^4x_1+   x_3^2$;
\item  $(12,[1,2,3,5])$, for example $x_0^{12}+x_1^6+x_2^4+x_1x_3^2$;
\item $(16,[1,2,5,7])$, for example $x_0^{16}+x_1^8+x_0x_2^3+x_1x_3^2$;
\item $(22,[1,2,7,11])$, for example $x_0^{22}+x_1^{11} + x_0 x_2^3+x_3^2$.
 \end{enumerate}
The   types mentioned in 1 and 2 give  properly elliptic surfaces and  those of type   3  and 4 give
K3 surfaces.  
\end{prop*}
 The  examples figuring  in the above proposition are preferred members of the given type which  will be referred to
   as the \textbf{\emph{basic examples}}.  The monomials present for each of the basic examples are required for quasi-smoothness.
  The minimal surfaces they define  can be shown to be Delsarte surfaces  in Shioda's terminology. See \cite[\S~13.2.1]{mwlatts} for details.

\begin{rmq}   
We shall show (see \S\ref{ssec:Invs}) that
all  type 3 surfaces are birational to surfaces of type $(9,  [1,1,3,4])$ and
all type 4 surfaces are  birational to  surfaces of type $(12,[1,1,4,6])$.
 The first is number  $8$ in  M. Reid's list of $95$ families, and the second
is number $14$. 
\end{rmq}
 
\subsection*{2}  In the weighted case  the group of projective automorphisms is
in general not reductive which causes problems when we want to construct moduli
spaces of weighted hypersurfaces. In our situation we circumvent this problem
by giving  certain normal forms which give projectively isomorphic surfaces if
and only if they are in the orbit of some fixed  algebraic torus of projective
transformations. In each of the four cases this gives a quasi-projective
moduli space of the expected dimension. See \S~\ref{ssec:NormForms}.
 
\begin{rmq} A general approach to the construction of geometric quotients under
non-reductive group actions has been proposed  in 
\cite{GITuni1,GITuni2,GITuni3}. Based on this, D. Bunnett showed~\cite{bunnett}
that certain classes of weighted hypersurfaces admit GIT-moduli spaces. In his
work, it is crucial that the weights divide {the} degree (in order to have
Cartier divisors instead  of $\bQ$-divisors for the linearization). Another
crucial assumption concerns the unipotent radical of the group of projective
automorphisms of the weighted projective space. Neither one of these holds
for our examples.
\end{rmq}

The collection of   degree $d=a+b+4$ weighted 
hypersurfaces in $\bP(1,2,a,b)$ in a natural way forms an ordinary projective
space $\bP^N$ by considering the $N+1$ coefficients in front of all possible
monomials. The quasi-smooth  hypersurfaces  belong to  a Zariski-open subset
$U_{1,2,a,b}$ of this projective space. The tautological family $\cF_{a,b}$ of
degree $d$ quasi-smooth hypersurfaces over $U_{1,2,a,b}$ is called the 
corresponding \textbf{\emph{universal family}}.   Using degenerations having
an isolated exceptional unimodal Arnol'd-type singularity we show that the
global monodromy group of the universal families in each case 1--4   is
as big  as possible:

\begin{prop*}[=Proposition~\ref{prop:GlobMon}]  Let $L $ be the middle
cohomology group of the minimal resolution of singularities of a quasi-smooth
member of $\cF_{a,b}$,  let $S \subset L$ be the Picard lattice of a general
member and $T=S^\perp$ the transcendental lattice. Then the monodromy group of
the universal family of such quasi-smooth hypersurfaces is the subgroup of
$O^{\#-}(L)$ preserving $T$ and  inducing the identity  on $S$.\footnote{See below for the notation.} 
\end{prop*}

\subsection*{3}   Our  examples  all  are simply connected and  the Hodge structure on the middle
cohomology group   looks like that of a K3 surface
(see Proposition~\ref{prop:InsAndSIngs}). In particular, the  period domain
is of similar type  (see  formula~\eqref{eqn:PerDom}).

It is well known that for a Kuranishi family of  K3 surfaces  
the period map is always an immersion and so infinitesimal Torelli holds.
In the setting of elliptic surfaces having  multiple fibers this is no longer the case
according to an observation  of K.\ Chakiris:

\begin{thm*}[\cite{ChakCounter}] Simply connected  elliptic surfaces with
$p_g>0$ and having one or at most two multiple fibers (with co-prime
multiplicities)  are counterexamples to the Torelli theorem: the fiber of the
period map for its Kuranishi  family  is positive dimensional.
\end{thm*}
The proof in loc.\ cit.\ is only sketched.  We therefore decided to
give a (simple) proof in the case of one multiple fiber (the situation
occurring in our examples), see Proposition~\ref{prop:ChaksThm}.

In our setting these results need to be used with care since our deformations are restricted
to the ones that keep the surface in a fixed weighted projective space.
As we show in ~\ref{ssec:Manual}, the period map for  the
Kuranishi family of  the basic examples   in the sense of Lemma~\ref{lem:kuranishi}), as  given above,  
has a $1$-dimensional kernel.    However, 
as shown in ~\ref{ssec:KernelCode}, this is not generically the case:
\begin{prop*}[=Proposition~\ref{prop:GenPerMap}]  
The period map for  the
Kuranishi family  (in the sense of   Lemma~\ref{lem:kuranishi})   
   for a general   type 1--4
surface  is an immersion.
\end{prop*}

An  interesting arithmetic consequence of this result has been signaled  to us by Ben Moonen, namely  the validity of the Tate 
 conjecture as well as the Mumford--Tate conjecture for each of the present surfaces. See Corollary~\ref{cor:moonen}.

\subsection*{4}  
 The results in the paper involving the structure of the period domain as well as  the behavior of the period map 
use precise information about the fibers of the genus 1  fibrations on a general surface from each of the four classes.
The determination of the fiber types is relatively standard and  has been facilitated by  calculations in \textsc{SageMath}. Together with the obvious 
bisection which comes from the resolution of singularities this gives a sublattice of the Picard lattice but in general it is hard 
to determine  whether this is the entire Picard lattice.
\par
Given suitable models in our families which are defined  over the integers,  counting points on reductions  at one or two "good"  primes 
gives a by now standard method to determine  the Picard number of a general member of a family.
Originally we applied another trick to show this for type 1 surfaces, but this trick cannot be applied to
type 2 surfaces. Thanks to the competence of W.\ Nijgh 
the type 2  surfaces could be handled by applying
the first mentioned method. 
From the way this proof was set up, we only recently found out  that a general type 2  surface 
is birational to  a type 1 surface  of the sort for which we had shown that the Picard number is generally equal to $2$.
Since this birational transformation lowers the Picard number by $1$,   the original surface has generally Picard 
rank $3$. See Remark~\ref{rmk:OnNormForms}.1 and Remark~\ref{rmk:bToa}.

Since  there are many possible birational transformations, this would not easily have been discovered before  
having gone through the details of Nijgh's approach. Therefore it was clear to us that  his proof  forms a natural companion to our
paper and so we placed it  in ~\ref{app:nijgh}. We want to mention that
  he furthermore proves  (see Remark~\ref{rmk:OnTypeA}.(2)) that a similar but much simpler 
  approach also implies that a general type 1 surface has Picard number $2$.

Using the determination  of the general Picard lattice for the four types 1--4,   in Proposition~\ref{prop:translatt},
 we calculate   the transcendental lattice of the generic surfaces.

\subsection*{5}  As in the case of the Kunev example, in the properly
elliptic case there is a unique canonical divisor $K$ on the surface $X$
and one may associate to the pair $(X,\text{supp}(K))$ the mixed Hodge
structure on $H^2(X\setminus \text{supp}(K))$.  We arrive in this way
at two further results: first of all Theorem~\ref{thm:Main}, stating
that for the associated \emph{mixed period map}  in cases where ordinary infinitesimal Torelli fails,
the infinitesimal
Torelli theorem does hold for the mixed Hodge structure, and, secondly Corollary~\ref{cor:Main}
which states that a wide class of related variations is rigid in the
sense of \cite{HolBisect}, that is, all deformations of the mixed
period map keeping source and target fixed are trivial.
 
\subsection*{6}  
In Section~\ref{sec:KSBA} we give an application to 
  certain compactifications of moduli  spaces of surfaces of general type with $K^2=1$, $p_g=2$ and $q=0$
 as a follow up of \cite{GPSZ}.

\subsection*{Acknowledgements} {\small  We thank
Patricio Gallardo and Luca Schaffler for their comments on an earlier
version of the manuscript, Miles Reid and Matthias Sch\"utt for their
help in understanding the geometry of the elliptic pencils,
Wolfgang Ebeling for answering questions about singularities and their
monodromy groups, and J\'anos Koll\'ar for pointing out the reference
\cite{k-bmy}. 
}
   \subsection*{Conventions and Notation}
   {\small 
    \begin{itemize}
\item A lattice is a free $\bZ$-module of finite rank equipped with a
  non-degenerate symmetric bilinear integral form which is denoted
  with a dot.
\item A rank one lattice $\bZ e$ with $e.e= a$ is denoted $\langle
  a\rangle$, orthogonal direct sums by $\operp$. Other standard
  lattices are the hyperbolic plane $U$, and the root-lattices $A_n,
  B_n$ ($n\ge 1$), $D_n$ ($n\ge 4$) and $ E_n$, $n=6,7,8$.
\item If one replaces the form on the lattice $L$ by $m$-times the
  form, $m\in \bZ$, this scaled lattice is denoted $L(m)$.
\item $A(L)=L^*/L$ is the discriminant group of a lattice $L$, $b_L$
  the discriminant bilinear form. In case $L$ is even, $q_L$ denotes
  the discriminant quadratic form. See \S~\ref{ssec:OnLats}.
\item The orthogonal group of a lattice $L$ is denoted $O(L)$,
  $O^{\# }(L)$ is the subgroup of isometries inducing the identity on
  $A(L)$, $O^{ \#\pm }(L)$ is the subgroup of $O^{ \# }(L)$ consisting of
  isometries with signed spinor norm $ 1$.  See \S~\ref{sec:GlobMono}.
 \item We denote  weighted projective spaces in the usual fashion as
  $\bP(a_0,\dots,a_n)$ with weighted homogeneous coordinates, say
  $x_0,\dots, x_n $. Let $I\subset \set{0,\dots,n}$.
  The weighted subspace obtained by setting the coordinates in
  $\set{0,\dots,n}\setminus I$ equal to zero is denoted $\mathsf P_I$ so that
  the coordinate points are $\mathsf P_0,\dots, \mathsf P_n$.

  A degree $d$ polynomial with such weights has \textbf{\emph{symbol}}
  $(d,[a_0,\dots,a_n])$.  Let $F$ be a polynomial with this symbol. We set
  \begin{align*}
    \Omega_n &= \sum_{j=0}^n x_j dx_0\wedge dx_1\cdots\wedge \widehat{dx_j}\wedge
         \cdots \wedge dx_n,\\
     J_F &= (\del F {x_0},\dots ,\del F {x_n})\subset \bC [x_0,\dots,x_n],
         \text{ the Jacobian ideal of }F, \\
     R_F &= \bC [x_0,\dots,x_n]/J_F,   \text{ the Jacobian ring of $F$} ,
            \quad R^k_F \text{ degree $k$ part of } R_F.
\end{align*}
\item We often do not write coefficients in front of monomials and so we use
  the shorthand  $\sum_{k_0,\dots,k_n}  x^{k_0}_{0}\cdots x^{k_n}_{n}$ instead of
  $\sum_{k_0,\dots,k_n}  a_{k_0,\dots,k_n} x^{k_0}_{0}\cdots x^{k_n}_{n}$.
\end{itemize}
}

 \section{Weighted projective hypersurfaces}
 \label{sec:genwphs}

 \subsection{Generalities} In this subsection we recall some results from the
 literature on hypersurfaces in weighted projective spaces, e.g.
 \cite{DolWPV,IFWings,steen}. Recall that $\bP:=\bP(a_0,\dots,a_n)$ is
 the quotient of $\bC^{n+1}\setminus\set{0}$ under the $\bC^*$-action
 given by $\lambda (x_0,\dots,x_n)=
 (\lambda^{a_0}x_0,\dots,\lambda^{a_n}x_n)$.  We may always assume
 that $a_0\le a_1\le \cdots\le a_n$.  The affine piece $x_k\not=0$ is
 the quotient of $\bC^n$ with coordinates
 $(z_0,\dots,\widehat{z_k},\dots,z_n)$ by the action of $\bZ/a_k\bZ$
 given on the coordinate  $z_i =x_i/x_k^{(a_i/a_k)}$  by $\rho^{a_i}z_i$,
 where $\rho$ is a primitive $a_k$-th root of unity. Observe that in
 case $a_0=1$, the coordinates $z_j=x_j/x_0, j=1,\dots,n$ are actual
 coordinates on the affine set $x_0\not=0$; there is no need to divide
 by a finite group action.

 In general $\bP$ has cyclic quotient singularities of transversal
 type $\frac 1 h (b_1,\dots,b_k)$, i.e., these are the image of
 $0\times \bC^\ell \subset \bC^{k}\times \bC^\ell $, where $\bZ/h\bZ$
 acts on $ \bC^{k}$ by $\zeta (x_1,\dots,x_k)=
 (\zeta^{b_1}x_1,\dots,\zeta^{b_k}x_k)$, $\zeta$ a primitive $h$-th
 root of unity.  More precisely, the simplex
 $x_{j_1}=\cdots=x_{j_k}=0$ is singular if and only if the set of
 weights that result after discarding $a_{j_1},\dots, a_{j_k}$ are not
 co-prime, say with gcd equal to $h_{j_1,\dots,j_k}$, and then
 transversal to the simplex one has a singularity of type
 \[\frac 1{h_{j_1,\dots,j_k}} (a_{0},\dots, \widehat{a_{j_1}}, \dots,
   \widehat{a_{j_k}},\dots,  a_ {n}).\] 
So in case any $n$-tuple from the collection $\set{a_0,\dots,a_n}$ of
weights is co-prime, the only possible singularities occur in
codimension $\ge 2$. We call such weights \textbf{\emph{well formed}}
and in what follows we shall assume that this is the case.

 A hypersurface $X=\set{F=0}$ in $\bP$ is quasi-smooth if the
 corresponding variety $F=0$ in $\bC^{n+1}$ is only singular at the
 origin. This implies that the possible singularities of quasi-smooth
 hypersurfaces come from the singularities of $\bP$. Such a
 hypersurface has at most cyclic quotient singularities, i.e. it is a
 $V$-variety.  A hypersurface of degree $d$ in $\bP$ is called
 \textbf{\emph{well formed}} if its weights are well formed and if
 moreover $h_{ij}=\gcd(a_i,a_j)$ divides $d$ for $0\le i<j\le n$. All
 our examples are well formed hypersurfaces.
 To test if $F=0$ is quasi-smooth one uses the Jacobian criterion: the 
 only solution to $\nabla F (\underline x)=0$ is $\underline x=(x_0,\dots,x_n)=0$.

We quote a result implied by Fletcher's statement \cite[Theorem 8.1]{IFWings}. We use it to
exclude types that do not give a quasi-smooth weighted hypersurface:

\begin{lemma}  \label{lem:Lat}  Given a weighted projective space $\bP=\bP(a_0,\dots,a_n)$ and an integer $d$ with $d>\max{a_j}$,
 let $A$ be the set of weights dividing $d$ and $B$ the remaining set of weights.   
  Suppose that  a quasi-smooth degree $d$ hypersurface  in $\bP$ exists.
 Then the set of weights $\set{a_0,\dots,a_n}$ satisfies the conditions
 \begin{enumerate}
\item for each $\beta \in B$ there is   a weight $\gamma$ and  some positive integer $r$ 
such that   $d= r \beta+ \gamma$.
\item  no weight  appears more than once as such a remainder  $\gamma$.
\end{enumerate}
   \end{lemma}
     
   \begin{example}
   We give two examples of surfaces having  type  $(d;(1,a_1,a_2,a_3))$  and which will be called 
   \textbf{\emph{basic degree $d$ quasi-smooth hypersurfaces}} in $\bP(1,a_1,a_2,a_3)$.
 \\
 1. Assume that $A=\set {1,a_1,a_2}$ and $d=q_3 a_3+a_j$, $a_j\in A$.
 Then  $P=x_0^d+ x_1^{d  /a_1}+ x_2^{d/a_2}+ x_{a_j}x_3^{q_3 }$  is quasi-smooth
 as follows from the Jacobian criterion.
 \\
 2. Assume that $A=\set{1,a_1}$ and $d=q_2 a_2 +a_j$,    $a_ j\in A$,   $d=q_3 a_3 +a_k$,  $ a_k\in A$ but $k\not=j$.
  Then  $P=x_0^d+ x_1^{d /a_1}+ x_{ j}  x_2^{q_2} + x_{ k} x_3^{q_3} $  is quasi-smooth.
   \end{example}

\medskip

In what follows, especially in \S~\ref{sect:Elliptic}, the following
remark will be used tacitly.
\begin{remark}\label{rmk:OnAmpl} By \cite[Thm. 3.3.4]{DolWPV}, if $X$ is
quasi-smooth of amplitude $\alpha(X)$, then the sheaf $\cO_X(\alpha(X))$ is
the canonical sheaf of $X$.  However,  contrary to what happens in ordinary projective space, if $\alpha\ge 1$ the minimal model of  $X$
need not be of general type,  as we shall see in  Proposition~\ref{prop:Main1}.
\end{remark}

\subsection{On the Hodge decomposition of weighted hypersurfaces}  
\label{ssect:OnSteen}
A result by J.~Steenbrink~\cite{steen}  states that the Hodge decomposition for
quasi-smooth hypersurfaces $X$ of degree $d$ in weighted projective space
$\bP(a_0,\dots,a_n)$ can be stated in terms of the Jacobian ring $R_F$
using  Griffiths' residue calculus,  as in the non-weighted case. The Hodge
number $h^{n,0}(X)$ equals $\dim H^0(X,\omega_X)$, where $\omega_X$ is the
canonical sheaf.  This Hodge number can be calculated from the amplitude
$\alpha(X)$ since by \cite[Thm. 3.3.4]{DolWPV},
$h^{n,0}(X)= \dim H^0(X,\cO(\alpha(X)))$ in case $X$ is quasi-smooth.
\par
Since quasi-smooth hypersurfaces  in weighted projective space are $V$-manifolds, 
as  in the case of ordinary projective space we have:
  
\begin{lemma}[\protect{\cite[\S1]{Tu}}] The subspace $\mathsf{Def}_{ {\rm  proj}}$ of the Kuranishi
space of deformations of $X$  within $\bP(a_0,\dots,a_n)$ is smooth with tangent space
canonically isomorphic to $R_F^{d}$. \label{lem:kuranishi}
The Kuranishi family restricted to $\mathsf{Def}_{ {\rm  proj}}$ 
is called the \textbf{\emph{Kuranishi family of type $([d],(a_0,\dots,a_n))$}}.
\end{lemma}

\subsection{The four types of surfaces}
We now give the classification of the surfaces we are interested in: 

\begin{proposition}\label{prop:OurExamples} The only quasi-smooth hypersurfaces $X$ of
the form $(d,[1,2,a,b])$ with $a,b$   co-prime odd integers and such that
$d= a+b+4$ have the following characteristic:
\begin{enumerate}
\item  $(14,[1,2,3,7])$, basic quasi-smooth example
  $ x_0^{14}+ x_1^7 + x_2^4x_1+   x_3^2$;
\item  $(12,[1,2,3,5])$,  basic quasi-smooth example
  $x_0^{12}+x_1^6+x_2^4+x_1x_3^2$;
\item $(16,[1,2,5,7])$,  basic quasi-smooth example
  $x_0^{16}+x_1^8+x_0x_2^3+x_1x_3^2$;
\item $(22,[1,2,7,11])$,  basic quasi-smooth example
  $x_0^{22}+x_1^{11} + x_0 x_2^3+x_3^2$.
 \end{enumerate}
      \label{prop:GoodElliptic}
 \end{proposition}
\begin{proof}  We divide the possible cases according to the partition
$\set{1,2,a,b}=A\sqcup B$ of Lemma~\ref{lem:Lat}. We do not assume
that $a<b$ since their roles are symmetric. Indeed, the constraints
are $d=a+b+4$, $a$ and $b$ odd, and $\gcd(a,b)=1$. If $a\in A$,
according to whether $b\not\in A$ (case (A) and (B)) or $b\in A$
(case (C)) we have
\[
  d  =ka= \begin{cases}
   \text{either } & r b+   2 , \quad (A) \\
   \text{or } & r b+    1, \quad  (B) \\
    \text{or } & r b  \hspace{3em} (C).
  \end{cases} 
  \]  
Interchanging the roles of $a$ and $b$ this also covers the case $b\in A$ and so
there remains the case $2\in A$, that is $2|d$, and then  
 \[
 \begin{array}{ll}
   d  =2k & =ra +2 = sb+ 1 \text{ or }  \\
    d=2k   &=sa+1= rb+2  
    \end{array}\,   (D).  
 \]
 
  We first assume that (A) holds.  Since $d=ka=a+b+4$ we have
  \begin{equation}
   b=(k-1)a-4. \label{eqn:First}
  \end{equation}   
Since $a$ and $b$ are odd, \eqref{eqn:First} implies that $k $ and
hence $d$ is even and (A) implies that also $r$ is even. Put
$k=2\kappa$, $r=2\rho$. We rewrite (A) as $2\rho \kappa a
-(a+4)\rho-a\kappa +1=0$ and hence
\[
(2\rho-1)[a(2\kappa- 1)- 4]= a+2
\]
and we can test low values of $a$. For $a=3$ this reads $(2\rho-1)(6\kappa -7)=5$
with only solution $(\rho,\kappa)= (1,2)$ which yields $b=5$, $d=12$.  For $a=5$
one gets $(2\rho-1)(10\kappa -9)=7$ with solution $(4,1)$ which yields $b=1 $
which can be discarded.  For $a=7$ one gets $(2\rho-1)(14\kappa -11)= 9$ with
solution $(2,1)$ which yields $b=3$ and $ d=14$.  For $a=9$ one gets
$(2\rho-1)(18\kappa -12)= 11$ which has no solution.  There are no
other solutions. To see this, write
\begin{align}
\label{eqn:Second}
(\rho( 2\kappa-1)-\kappa) a= 4\rho -1.
\end{align} 
For $\rho=1$ the equation \eqref{eqn:Second} gives $(2\kappa-2) a=6$
which gives back the solution $a=3, b=5$ and so we may assume $\rho \ge 2$.
We may also use that $a\ge 11$. By \eqref{eqn:Second} this gives
$(2\rho\kappa -\rho-\kappa) \cdot 11 \le 4\rho -1$, or, multiplying by $2$,
\[
 (2\rho-1)(22\kappa-15) \le 13
 \]
and so $13\ge (2\rho-1)(22\kappa-15) \ge 66 \kappa -45$ which has no positive
integer solution.

Case (B) has solution $(22,[1,2,7,11])$ with $a=11$, $k=2$, $b=7$, $r=3$. This
follows as in case (A). Here we set $k=2\kappa$, $r=2\rho+1$ and obtain
\[
 2\rho [a(2\kappa- 1)- 4] = a+3
\]
As before, the smallest value of $a$ with a solution is $a=11$. There 
are no solutions with $a\ge 13$.  To see this
we use the analog of \eqref{eqn:Second} which reads
\[
(4\rho\kappa -2\rho -1)a= 8\rho-3 
\]
and from $a\ge 13$ we derive
\[
\rho(52\kappa-34)\le 10
\]
which is  not possible for positive integers $(\rho,\kappa)$.

Case (C) implies $ka= rb= r[ k a -4]$, which leads to  $(r-1) k  a=4r$ and since
$k $ and  $r $ must be even (recall that $d=a+b+4=ra=kb$  with $a,b$ odd), which
leads  to  a contradiction.

In case (D), eliminating $a$ and $b$ and substituting in $2k=d=a+b+4$, we find
$ (2k -4)rs-2k(r+s)+(2s+r)=0$ with $r$ even  and $s$ odd which can be rewritten
as 
\[
[(k-2)(r-1)-1][2(k-2)(s-1)-3]= (k-1)(2k-1).
\]
Since $r\ge 2$ and $s\ge 3$, $k=8$ is the smallest value of $k$ with
solution $(r,s)=(2,3)$ which leads to $(16,[1,2,5,7])$. Since $r\ge 3$ and
$s\ge 2$ we get the inequality $(4k-11)(k-3)\le (k-1)(2k-1)$ which gives
\[
(k-2) (k-8)\le 0
\]
which has no positive integer  solutions $>8$.
\end{proof}

\section{The universal family:  normal forms, moduli, global monodromy}

The collection of degree $d$ weighted hypersurfaces in
$\bP=\bP(a_0,\dots,a_n)$ form an ordinary projective space $\bP^N$ in
a natural way by considering the $N+1$ coefficients in front of all
possible monomials. The quasi-smooth hypersurfaces form a Zariski-open
subset $U_{a_0,\dots,a_n}$ of this projective space.  The tautological
family $\cF_{a_0,\dots,a_n}$ of degree $d$ quasi-smooth hypersurfaces
over $U_{a_0,\dots,a_n}$ is called the corresponding
\textbf{\emph{universal family}}. The group $\widetilde G$ of
substitutions $x_j\mapsto p_j(x_0,\dots,x_j)$, $j= 0,\dots,n$, where
$p_j$ is weighted homogeneous of degree $a_j$, acts on
$\bP(a_0,\dots,a_n)$.  Since $\lambda\in \bC^*$ sending
$(x_0,\dots,x_n)$ to $(\lambda^{a_0} x_0, \lambda^{a_1}x_1,\dots,
\lambda^{a_n}x_n)$ multiplies each weighted homogeneous polynomial $F$
of degree $d$ with $\lambda^d$, the group $G=\widetilde G/\bC^*$
acts effectively on hypersurfaces. The embedding of the subgroup
$\bC^*\subset \widetilde G$ is due to the weights and so will be
referred to as the \textbf{\emph{$1$-subtorus for the weights}}.

\subsection{Normal forms and moduli}  \label{ssec:NormForms} We show how to
obtain a quasi-projective moduli space as a certain GIT-quotient of
$U_{1,2,a,b}$.  The draw-back is that the group $G$ of weighted
projective substitutions is not in general reductive.  We can
circumvent this in our case by giving normal forms for the equation of
quasi-smooth hypersurfaces in the universal family.  On hypersurfaces
with their equations in normal form a reductive subgroup $T$ of $G$
(in fact a $3$-dimensional algebraic torus) acts effectively in such a
way that hypersurfaces in normal form are in the same $T$-orbit if and
only they are in the same $G$-orbit.
   
\begin{proposition} \label{prop:GIT} With $(a,b)\in \bZ^2$ as   in
Proposition~\ref{prop:GoodElliptic}, let $\bC[x_0,x_1,x_2,x_3]$  be the
homogeneous coordinate ring of $\bP(1,2,a,b)$. Assume that $F\in
\bC[x_0,x_1,x_2,x_3]$ defines   a quasi-smooth surface $(F=0)$ of
degree $d$ in $\bP(1,2,a,b)$ which does not pass through $(0:1:0:0)$
(i.e., the coefficient of $x_1^{d/2}$ is non-zero).   For type 4, that
is, for $(a,b)=(7,11)$, assume in addition that the coefficient of
$x_1^4x_2^2$ is non-zero.
  
Then, there exists  polynomials  $G_j$, homogeneous  of degree $j$ in  the degree $2$ monomials $x_0^2,x_1$   
 such that via the automorphism group of $\bP(1,2,a,b)$ the form $F$ can be put in the following normal
form:
\begin{enumerate}
\item  in case  $(a,b,d)=(3,7,14)$ we  have
\begin{align*}
     F &= x_1\,x_2^4 + G_0 x_0^5 \,x_2^3 + G_4(x_0^2,x_1)\,x_2^2 + \\
       &  \hspace{2em} x_0\,G_5 (x_0^2,x_1)\,x_2 + G_7(x_0^2,x_1) - x_3^2.
     \end{align*} 
\item  in case $(a,b,d)=(3,5,12)$, 
\begin{align*}
  F  &= x_1\,x_3^2 +x_0\,x_3\,G_2(x_0^3, x_2) + G_0\,x_2^4
        + G_3(x_0^2,x_1)\,x_2^2  + \\
      & \hspace{4em}  x_0\,G_4(x_0^2,x_1)\,x_2 + G_6(x_0^2,x_1),\quad G_0\not=0 . 
 	\end{align*} 
\item  in case $(a,b,d)=(5,7,16)$,   
\begin{align*}
  F  &=  x_1\,x_3^2 + x_0^4 G_1(x_0^5, x_2) x_3
         + r_0 x_0 x_2^3 + G_0\,x_1^3\,x_2^2 \\
     &  \hspace{2em} + x_0\,G_5(x_0^2,x_1)\,x_2 + G_8(x_0^2,x_1),
\end{align*} 
where $r_0$ is a non-zero constant.
\\
\item   in case  $(a,b,d)=(7,11,22)$,  
 \begin{align*}
   F   &=   \,x_0 \,x_2^3 + G_0\, x_1^4  \,x_2^2 + x_0 \,x_2 \,G_7(x_0^2,x_1)
            + G_{11}(x_0^2,x_1) - x_3^2,\quad G_0\not=0,
\end{align*}
where the coefficient of $x_0^{22}$ in $G_{11}$ is zero. 
 \end{enumerate}
In each case, the subgroup of the automorphism group $T$ of
$\bP(1,2,a,b)$ which preserves a normal form of the given type
consists of transformations of the form $x_j\mapsto c_jx_j$ with
$c_j\in \bC^*$, $j=0,1,2,3$ modulo the $1$-subtorus for the weights.
More concretely,  for types 1  and 4   this can be identified with the
subgroup of $(\bC^\ast)^3 $ consisting of triples $ (c_0,c_1,c_2 ) $
for which $c_1c_2^4=1$, $c_0c_2^3=1$ respectively, while  for types 
2  and 3   this is the  subgroup of $(\bC^\ast)^4 $ consisting of
quadruples $ (c_0,c_1,c_2,c_3) $ for which $c_1c_3^2=1$.

The stabilizer under $T$ of  a general such $F$  in all cases is the identity.

\end{proposition}

The proof of this result is relegated to ~\ref{sec:NormFormsBis}.  Note
that the supplementary condition on $x_1^{d/2}$ which is not required in
loc.\ cit. but will be used in Proposition~\ref{prop:Main1} is stable under the
group action.

\begin{remark}\label{rmk:OnNormForms}
1. A   type 2   surface is birational to a    type 1  surface: multiply the normal form 
with $x_1$ and perform the change of variables $y_0=x_0, y_1=x_1, y_2=x_2, y_3= x_1x_2$. This does not 
yield the normal form for   type 1, but it does after changing $y_3$ in $y_3+\half y_0G_2(y_0^3,y_2)$.
In the resulting normal form  the term with $y_0^{14}$ is missing, showing that after the birational transformation   the type 2  family
gives the subfamily of  type  1  where the coefficient of $x_0^{14}$ in the normal form vanishes.
\\
2.   For type  3  the coefficient of $x_1^3x_2^2$ in the normal form of
Proposition~\ref{prop:GIT}  is non-zero. Since this condition is stable under
the action of $T$, the moduli point of the basic example is on the boundary of
$\cM_{5,7}$.
Likewise, the condition on the coefficient of $x_1^4x_2^2$ for type
4  implies that the basic example can not be transformed in normal
form and so the moduli point of the basic example is on the boundary
of $\cM_{7,11}$.
\end{remark}
  
\begin{corollary} \label{cor:Moduli} (1) In each of the above cases the points in the
  Zariski-open subset $ U_{1,2,a,b} $  of degree $d=a+b+4$ quasi-smooth
  hypersurfaces in $\bP(1,2,a,b)$ are $G$-stable and
  $\cM_{a,b}=  U_{1,2,a,b }/\!/ G$ is a geometric quotient.
\\
(2)  $\cM_{a,b} $ has dimension  $18,17,16,18$  for types 1,2,3,4, 
respectively.
\end{corollary}

\begin{proof} (1) Since   for types 3 and 4   the basic examples are on the
boundary of $\cM_{a,b}$, we need to  check quasi-smoothness for at least one
surface whose moduli-point lies in the interior. This is done in
~\ref{sec:NormFormsBis}. The group $T$ acts effectively on hypersurfaces
defined by homogeneous forms in the coefficients of a weighted homogeneous
polynomial of degree $d$.  As in ordinary projective space
(cf. \cite[Prop. 4.2]{mum}), the locus of hypersurfaces that are not
quasi-smooth define in this way a "discriminant form", a $T$-invariant
homogeneous polynomial in the coefficients.  By construction this
polynomial is non-zero on $U_{1,2,a,b}$.  By definition, all points in
$U_{1,2,a,b}$ are then semi-stable. Since $T$-orbits are closed in
$U_{1,2,a,b}$ and (as in the projective setting) since a weighted
hypersurface   of types 1--4 by \cite[Theorems 2.1, 3.1]{esser}    has a finite automorphism group, the
points of $U_{1,2,a,b}$ are stable and the GIT-quotient
$U_{1,2,a,b}/\!/ G$ is a geometric quotient.
\\
(2) One counts the number coefficients of the monomials in the normal
form which are not fixed, and subtracts $2$   for types  1  and 4   and
$3$   for  the other two types. For instance,   for type   2 one finds $
3+1+4 +5+7-3=17$  and   for type 4 this becomes  $ 1+8+11 -2=18$.
\end{proof}

The universal family on $U_{1,2,a,b}$ does not descend to the
geometric quotient $\cM_{a,b}$. However, it does so over the open subset
$U^0_{1,2,a,b}\subset U_{1,2,a,b}$ corresponding to surfaces having no
automorphisms except the identity. This is a non-empty set since the
stabilizer of $T$ on the general $F$ is the identity as asserted
above.    Introduce the following notion: 
\begin{definition} A  \textbf{\emph{modular family}}  is an algebraic  family over
a smooth, quasi-projective base which is locally (in the analytic topology) isomorphic to the 
Kuranishi  family of type  $([d],(1,2,a,b))$  (see Lemma~\ref{lem:kuranishi}). \label{dfn:ModFam}
\end{definition}
 The above discussion can thus be rephrased as follows: 
\begin{corollary} The  family over $ U^0_{1,2,a,b}/\!/ G$  obtained from the universal
family of degree $d=a+b+4$ weighted hypersurfaces in $\bP(1,2,a,b)$  is   a \ modular family.    \label{cor:ModFam}
\end{corollary}

 \begin{remark}1. The standard holomorphic 2-form $\omega_F$ on $F$ given by the
residue of $x_0 \Omega_3/F$ is not fixed under the double plane involution
$x_3\mapsto -x_3$ for types  1 and 4.  One can view $\cM_{a,b}$ as the moduli
space for the pair $(F,\omega_F)$. Alternatively, we could put a non-zero
coefficient in front of $x_3^2$ and replace $T$ by a larger group
acting also on this coefficient. However then the isotropy group at a
generic double cover would always contain the double cover involution
preventing the existence of a universal family over a Zariski-open
subset of $U_{1,2,a,b}$. \\ 2. \label{rmk:OnUnivFam} The above normal
forms only generically give quasi-smooth hypersurfaces.
\\
3. As for ordinary projective spaces, the universal family of
quasi-smooth hypersurfaces of given degree is flat over the base. This
is because resolving the singularities of weighted projective space
also resolves the singularities of the hypersurfaces. The resolved
universal family being flat, also the universal family itself is flat.
\end{remark}

\subsection{Global Monodromy of the Universal Families} \label{sec:GlobMono}

\subsection*{Brief survey of singularity theory}

\par The purpose of this subsection is to investigate certain $1$-parameter
degenerations $X_t$ of quasi-smooth hypersurfaces in $U_{1,2, a,b}$ in
relation to the global monodromy of the universal family.
So we want to find a disc $D=\sett{t\in \bC}{|t|< r}$ embedded in
$\bP^N$ such that (i) $D^*=D\setminus \set{0}$ belongs to $U_{1,2,a,b}$ with
$X_t,t\in D^*$ a quasi-smooth hypersurface and (ii) $X_0$ (corresponding to
$0\in D$) has an isolated singularity at $x_0$ of some given type
suited for the calculation of global monodromy groups.  The main
object associated to $(X_0,x_0)$ is the \textbf{\emph{Milnor fiber}}
which is the intersection of $X_t$ ($|t|$ small enough) with a small
enough ball with center at $x_0$.

\par In what follows we freely quote results from W.\ Ebeling's book
\cite{monebe5}.  To understand these results we need to recall some
more lattice theory. Recall that a lattice is a free group of finite
rank equipped with a symmetric bilinear form which we denote by a dot.
A root $r$ in a lattice $L$ is a vector with $r\cdot r=-2$ and it
determines a reflection $\sigma_r$ sending $x\in L$ to $x+ (x\cdot r)
r$. The group of isometries generated by a set of roots $\Delta$ is
called its \textbf{\emph{Weyl group}} $W(\Delta)$.  Associated to
$\Delta$ is its \textbf{\emph{Dynkin diagram}}. The vertices
correspond to the roots and an edge is drawn between two edges
corresponding to roots $r,s$ if $r\cdot s=1$. In the lattices we
consider, only one other type of edge appears, namely if $r\cdot s=
-2$ one draws two dashed edges between the corresponding vertices.

\par The middle homology group of the Milnor fiber equipped with the
intersection pairing is the \textbf{\emph{Milnor lattice}} .  Its rank is the 
\textbf{\emph{Milnor number}} $\mu(X_0,x_0)$. 
Turning once around $0\in D$ induces the
monodromy-operator $T$ on $H^2(X_t,\bZ)$ as well as on the Milnor
lattice.  By \cite[\S~1.6]{monebe5} the Milnor lattice contains a
sublattice, its \textbf{\emph{vanishing lattice}} $ L=L (X_0,x_0)$.

\par In order to determine the global monodromy group, the graph of a
basic vanishing lattice, $\Delta_{\text{\rm min}} $ depicted in
Figure~\ref{fig:specvanlat} is of crucial importance. It intervenes in
the notion of a complete vanishing lattice since the notion of
vanishing lattice has a complete algebraic description:

\begin{definition}
 \begin{enumerate} 
\item 
A \textbf{\emph{vanishing lattice}} consists of a pair $(L,\Delta)$ of
a (possibly degenerate) lattice $L$ and a set of roots
$\Delta$ spanning $L$ and forming a single orbit under $W(\Delta)$.
  
\item  A vanishing lattice $(L,\Delta)$
\textbf{\emph{contains the vanishing lattice}} $(L',\Delta')$ if $L'$ is a
primitive sublattice of $L$ and $\Delta'\subset \Delta$.

\item A vanishing lattice $(L,\Delta)$ is \textbf{\emph{complete}} if
it contains $ (L_{\text{\rm min}}, \Delta_{\text{\rm min}})$.
\end{enumerate}
\end{definition}

\begin{figure}[hbt]
 \begin{center}
  \begin{tikzpicture}[scale= 0.65]
  \begin{scope}[ shift={( 0, -2)}]
   \draw[fill=black]  (-4,4) circle [radius=0.15cm]; 
  \draw[semithick]  (-4,4)--(-2,4); 
   \node at (-4.4,4)  {$r_1$}; 
   \node at (-2,4.5)  {$r_2$}; 
    \draw[fill=black]  (-2,4) circle [radius=0.15cm]; 
   \node at (2, 4.5)  {$r_5$}; 
     \node at (4.4, 4)  {$r_6$}; 
     \node at (0,6.9)  {$r_3$}; 
  \draw[semithick,dashed, double, double distance=1pt]  (-1.9,4)--(2,4);
     \draw[semithick]  (0, 5)--(0, 6.5); 
     \draw[semithick,dashed, double, double distance=1pt]  (-1.85,4.15)--(0,6.5);  
       \draw[semithick,dashed, double, double distance=1pt]  (2.15,4)--(0,6.5);  
        \draw[fill=black]  (0,6.5) circle [radius=0.15cm];
     \node at (0,4.6)  {$r_4$}; 
    \draw[semithick]  (-2,4)--(0,5); 
     \draw[semithick]  (0,5)--(2,4);
     \draw[fill=black]  (0,5) circle [radius=0.15cm]; 
    \draw[semithick]  (2,4)--(4,4);   
   \draw[fill=black]  (2,4) circle [radius=0.15cm]; 
    \draw[fill=black] (4,4) circle [radius=0.15cm];
\end{scope}
       \end{tikzpicture}
     \caption{$(L_{\text{\rm min}}, \Delta_{\text{\rm min}}) $}
     \label{fig:specvanlat}
   \end{center}     
    \end{figure}

\par The main interest in complete vanishing lattices is that their
isometry group is almost equal to its Weyl group.  Here two notions
intervene related to a lattice $L$: the \textbf{\emph{spinor norm}} of
an isometry of $L$ and the \textbf{\emph{discriminant group of $L$}}.

\par To define the former, recall that the Cartan--Dieudonn\'e theorem
states that all isometries of a $\bQ$-vector space $V$ with a
non-degenerate product are products of reflections, say $\sigma_x:V\to
V$, $\sigma_x(v)= v- [2(x.v) /(x.x)] v$. One defines the $\pm$-spinor
norm of such a product of reflections as follows.
  \begin{equation*} 
 \espinnm (\sigma_{x_1}\comp \cdots\comp\sigma_{x_r}  )= \begin{cases}
1 & \text{ if }\, \#\sett { j\in \set{1,\dots,r}}{\varepsilon q(x_j)<0}      \text{ is even }\\
 -1 & \text{ otherwise.}
 \end{cases} 
 \end{equation*} 
The group generated by isometries $\gamma$ with $\espinnm (\gamma)=1$
is denoted $O^{\#\varepsilon}(L)$. Here $\varepsilon=-1$ plays a
central role.  \par The discriminant group makes only sense for
non-degenerate lattices $L$, those for which the map $L\to
L^*=\hom_\bZ(L,\bZ)$ given by $x\mapsto (y \mapsto y\cdot x)$ is
injective. Then the discriminant group by definition is the group
$A(L)=L^*/L$.

We can now formulate the main technical result we are going to invoke:

\begin{theorem}{\protect{\cite[Thm. 5.3.5]{monebe5}}}  \label{thm:weylisbig}
Let $(L,\Delta)$ be a complete vanishing lattice. Then $W(\Delta)$
is the subgroup $O^{ \#-}(L)$ of the orthogonal group $O(L)$ of $ L$
consisting of isometries with (-)-spinor norm $+1$ and inducing the
identity on the discriminant group.
\end{theorem}
  \begin{figure}[h]
\begin{center}
  \begin{tikzpicture}[scale=0.7]
   \draw[fill=black]  (0,0) circle [radius=0.07cm]; 
  \draw[thin] (-1,0)-- (1,0);
  \draw[thin] (3,0)--(4,0);
  \draw[thin,dashed] (-1,0)--(-3,0);
\draw[thin, dashed, double, double distance=1pt] (0,0.06)--(0,1);
  \draw[thin] (-3,0)--(-4,0);
   \draw[thin] (0,0)--(0.5,-1); 
    \draw[thin,dashed] (1,0)--(3,0);
    \draw[thin,dashed] (0.5,-1)--(1.5,-3);
    \draw[thin] (1.5,-3) --(2,-4);
 \draw[thin] (0,1)--(0,2);
\draw[thin] (0,1)--(-1,0);
\draw[thin] (0,1)--(1,0);
\draw[thin] (0,1)--(0.5,-1);
 \draw[fill=black]  (1,0) circle [radius=0.07cm]; 

\draw[fill=black]  (3,0) circle [radius=0.07cm]; 
\draw[fill=black]  (4,0) circle [radius=0.07cm]; 
\draw[fill=black]  (0,1) circle [radius=0.07cm];  
\draw[fill=black]  (0,2) circle [radius=0.07cm]; 
 \draw[fill=black]  (0.5,-1) circle [radius=0.07cm]; 
 \draw[fill=black]  (1.5,-3) circle [radius=0.07cm];
  \draw[fill=black]  (2,-4) circle [radius=0.07cm]; 
  \draw[fill=black]  (-1,0) circle [radius=0.07cm]; 
  
\draw[fill=black]  (-3,0) circle [radius=0.07cm]; 
\draw[fill=black]  (-4,0) circle [radius=0.07cm]; 
 \node at (-4, -0.5)  {$u_{p-1}$}; \node at (-3, -0.5)  {$u_{p-2}$};
  \node at (-1, -0.5)  {$u_1$}; \node at (0, -0. 25)  {$r$} ;
  						 
   \node at (0.3, 1.2)  {$r_2$}; \node at (0.3, 2.2)  {$r_1$};
   \node at (1, -0.5)  {$s_1$};  
   \node at (1, -0.5)  {$s_1$}; \node at (3, -0.5)  {$s_{q-2}$};
   \node at (4, -0.5)  {$s_{q-1}$};  
    \node at (1, -1)  {$t_{1}$};\node at (2, -3)  {$t_{r-2}$};\node at (2.5, -4)  {$t_{r-1}$};
     \end{tikzpicture}
     \caption{Vanishing lattice given by  $T^1_{p,q,r} , p\ge 2, q\ge 3, r\ge 7$ is complete.} 
     \label{fig:T1pqr}
   \end{center}      
 \end{figure}

\begin{example} \label{exm:Arnold}
From \cite[Prop. 5.3.5]{monebe5}, one deduces that a root lattice with
a Dynkin diagram of type $T^1_{p,q,r}$ depicted in Figure~\ref{fig:T1pqr} has a different
root basis making it a complete vanishing lattice. Such Dynkin diagrams
come up as vanishing lattices for the $14$ exceptional unimodal
families of Arnold \cite{arn}. In Table~\ref{table:Sings} we describe
three of those which play a role in \S~\ref{ssec:Monodromy} below.
The vanishing lattice $T^1_{p,q,r} $ is given in
Figure~\ref{fig:T1pqr} by means of a Dynkin diagram with $\mu$
vertices. The "modulus" $a$ is any complex number.  This lattice is
non-degenerate.
 \begin{table}[ht]
\caption{Three exceptional unimodal singularities}\label{table:Sings}
\begin{center}
\begin{tabular}{|c|c|c|c|c|}
\hline
Notation & Normal Form   & Milnor number $\mu$  &Dynkin Diagram\\
\hline
$K_{12}$ & $x^3+y^7+z^2+a xy^5$ & $12$ & $T^1_{2,3,7}$    \\
$K_{13}$ & $x^3+xy^5+z^2+ay^8$& $13$ &  $T^1_{2,3,8}$  \\
$K_{14}$& $x^3+y^8+z^2+axy^6$&  $14$ &  $T^1_{2,3,9}$ \\
 \hline
\end{tabular}
\end{center}
\end{table}

  \end{example}

  \subsection{Applications to Monodromy} \label{ssec:Monodromy}
  
\par The embedding of the Milnor fiber into $X_t$ induces a lattice
morphism
$j_*: \Lambda (X_0,x_0) \to H_2(X_t,\bZ)\simeq H^2(X_t,\bZ)$ \footnote{The
  latter isomorphism is Poincar\'e duality.} which is in general
injective nor surjective.  We consider the global monodromy of the
universal families.  Note that monodromy not only preserves the
hyperplane class but also the singularities of the weighted projective
space. In our case the quasi-smooth members of the universal families
$\cF_{a,b}$ have singularities only at some of the isolated singular
points of $\bP(1,2,a,b)$ and we take the minimal resolution of their
singularities. We shall see (cf. Proposition~\ref{prop:InsAndSIngs})
that in two cases additional exceptional curve configurations are
present which are also preserved by the monodromy.  For a general
member $\widetilde X$ of each of the resulting families of smooth (but
not always minimal) surfaces we show that all of the curves just
mentioned generate the Picard lattice (cf. Proposition~\ref{prop:Main1}),
and so the transcendental lattice of $\widetilde X$ is left invariant.

\par The main result here is as follows:

\begin{proposition}\label{prop:GlobMon}  Let $L=H^2(\widetilde X,\bZ)$ be the
middle cohomology group of the minimal resolution of singularities of
a quasi-smooth member of $\cF_{a,b}$, let $S \subset L$ be the Picard
lattice of a general member and $T=S^\perp$ the transcendental
lattice.  Then the monodromy group of the universal family of such
quasi-smooth hypersurfaces is the subgroup of $O^{\#-}(L)$ preserving
$T$ and inducing the identity on $S$.
\end{proposition}
\begin{proof} The proof is similar to the proof of A.\ Beauville  given in
  \cite[Section 2]{monbeau}. The main ingredients are:
\begin{itemize}
\item The monodromy representation of $\pi(U_{1,2,a,b})$ on $L$ is the
same as the representation induced by a Lefschetz pencil.
\item The discriminant locus is connected implying that all vanishing
cycles are conjugate under monodromy and so these give a vanishing
lattice.
\item The vanishing cycles generate the orthogonal complement of $S$
in $L$ which is precisely $T $.
\item There is a weighted degree $d=a+b+4$ hypersurface in
$\bP(1,2,a,b)$ with an exceptional unimodal isolated singularity
from Arnol'd's list. Its vanishing lattice is non-degenerate and so
embeds in $T$. Hence $T$ is a complete vanishing lattice.
\end{itemize}
Since the first three assertions can be proven as in the ordinary
hypersurface case, it suffices to exhibit suitable singularities in
each of the four cases. We exhibit such a singularity at $(0,0,0)$ in
the affine chart $x_0\not=0$. We refer to Table~\ref{table:Sings} for
the notation of these singularities.

\par For type 1 the form $x_1^7+x_3^2+x_0^5x_2^3+x_0x_1^5x_2$ gives a
$K_{12}$-singularity and so does $x_0x_1^3+x_0^8x_2^7+x_3^2+x_0x_1^5x_3$ for type
4. For type 2  one has a $K_{13}$-singularity given by
$x_0x_1x_2^3+x_0^2x_1^5+x_0^2x_3^2+x_0x_1^4x_2$ and finally,
$x_0x_2^3+x_1^8+x_0^2x_3^2+x_0x_1^6x_2$ gives a $K_{14}$ for type 3.
 \end{proof}

\section{Invariants and  Elliptic Pencils On  the Four types of Surfaces}
\label{sect:Elliptic}

\subsection{Invariants} \label{ssec:Invs} 
 
We start this section by comparing types 3  and 4with two families
from M. Reid's list of surfaces,  
which shows directly that these are birational incarnations of K3 surfaces.

For  type 3 we start multiplying the normal form from
Proposition~\ref{prop:GIT} by $x_0^2$ yielding
\begin{align*}
  F & = x_1(x_0 x_3)^2 + x_0^5 G_1(x_0^5,x_2) x_0 x_3  + r_0 (x_0 x_2)^2
        + G_0 x_1^3 (x_0 x_2)^3 \\
     &\quad \quad  +  x_0 ^2 G_5 (x_0^2,x_1) x_0 x_2   +  x_0 ^2 G_8(x_0^2,x_1).
\end{align*}
Therefore, we can change variables to $y_0 = x_0^2, y_1 = x_1, y_2 = x_0 x_2$
and $y_3 = x_0 x_3$, except for possibly the term 
\[
   x_0^5 G_1(x_0^5,x_2) x_0 x_3
   = x_0^5(A x_0^5 + x_2) x_0 x_3     
   = (A x_0^{10} +  x_0^5 x_2 )   x_0 x_3 ,
\]
which can also be rewritten in the variables $ x_0^2, x_0 x_2$ and $x_0 x_3$. So
after these substitutions, one  obtains  a surface of  type $(18,[2,2,6,8])$,
or, equivalently a surface of   type $(9,[1,1,3,4])$.
 
For  type 4   first  multiply all monomials by $x_0^2$.
In the new variables $y_0 = x_0^2, y_1=x_1, y_2 = x_0x_2, y_3=x_0 x_3$  and as
for type 3,  one  sees that all  type $(22,[1,2,7,11])$ surfaces 
are  birational to  surfaces of type $(12,  [1,1,4,6])$.  
\smallskip

Next, we give a table of the Hodge numbers and the number of
projective moduli resulting from applying Steenbrink's approach
outlined in \S~\ref{ssect:OnSteen} for the four types of surfaces we
just found as well as for the two surfaces in the Reid incarnation
which we denote by $3^*$, respectively $ 4^*$.  We observe that the
last column of the table corroborates the dimensions of the moduli
spaces found in Corollary~\ref{cor:Moduli}.  For details see also
~\ref{app:SAGE}.

\begin{table}[ht]
\caption{Invariants for the classes of elliptic weighted surfaces}
\begin{center}
\begin{tabular}{|c|c|c|c| }
\hline
symbol  &  $h^{2,0}=h^{0,2}$  &  $h^{1,1}_{\rm prim} $  & no. of    \\
&&& projective moduli    \\
\hline
 1.  $(14,[1,2,3,7]) $  & $1$   & $18 $  &  $18$     \\
 2.    $(12,[1,2,3,5])$   &  $1 $ & $17 $  & $17 $ \\
3.  $(16,[1,2,5,7])$   & $1 $ & $ 17$  & $16 $ \\
 3$^\ast$.    $(9,  [1,1,3,4])$ & $1 $ & $ 16$  & $16 $  \\
\,  4.   $(22,[1,2,7,11])$  & $1 $ & $ 18$  & $ 18$ \\
\, 4$^\ast$.   $(12,  [1,1,4,6])$  & $1 $  & $ 18$  &  $ 18$ \\
\hline
\end{tabular}
\end{center}
\label{tab:Main}
\end{table}%

 In what follows we focus on the  incarnations  1--4, i.e.,  we consider, the surfaces $X\subset \bP(1,2 ,a ,b)$ having
singularities at most at $\mathsf P_2=(0:0:1:0)$ and $\mathsf
P_3=(0:0:0:1)$.  The minimal resolution $\widetilde X$ of the
singularities does not necessarily give a minimal surface as we shall
see  for types  3  and 4. We let $X'$ be its minimal
model.\footnote{$X$ being rational nor ruled, (see below) the minimal model
is unique.}

\par We calculate the invariants for the four types 1--4, making use of the Hodge numbers in Table~\ref{tab:Main}.

\begin{proposition}
\label{prop:InsAndSIngs} $X'$ is a simply connected  surface  with invariants
$e(X')=24$, $K^2_{X'}=0$, $b_2(X')=22$,  Hodge numbers
$(h^{2,0}, h^{1,1},h^{0,2})= (1,20,1)$ and signature $(3,19)$.  More precisely,

\begin{itemize}
\item
For type 1  the general surface $X$ has only one cyclic singularity
at $\mathsf P_2 $ of type $\frac{1}{3}(1,1)$ which is resolved by a
rational curve of self-intersection $(-3)$.
\item
For type 2  there is generically only one cyclic singularity at
$\mathsf P_3 $ of type $\frac{1}{5} (1,3)$ which is resolved by   a
chain of two transversally intersecting rational curves of
self-intersections $-2$ and $ -3 $ respectively.
\item
For type 3  the  surface  $X$ has generically two singularities: a
$\frac{1}{5}(1,1)$-singularity at $\mathsf P_2 $ resolved by a single rational
curve with self-intersection $-5$, and a $\frac{1}{7}(1,5)$-singularity
at $\mathsf P_3 $ resolved by a chain of three rational curves with
self-intersections $-2,-2,-3$, respectively. The surface $X'$ is
obtained by blowing down an exceptional configuration in $\widetilde X$
consisting of a chain of two smooth rational curves of self-intersections
$-1,-2$.
\item
For type 4  the  surface  $X$ has generically one singularity at $\mathsf P_2 $ of type
$\frac{1}{7}(1,2)$  resolved  by a  chain  of two  rational curves  with
self-intersections $-2,-4$ respectively. The surface $X'$ is obtained by
blowing down an exceptional curve in $\widetilde X$.
\end{itemize} 
 
\end{proposition}

\begin{proof} First of all observe that $X$ and hence $\widetilde X$ and $X'$
are all simply connected since all quasi-smooth hypersurfaces
(of dimension $>1$) in a weighted projective space are simply connected.
In particular, $X$ cannot be a rational or ruled surface, and $X'$ is uniquely
determined.

Secondly, all surfaces have canonical sheaf $\cO(1)$ with
$1$-dimensional space of sections generated by $x_0$.  Since $p_g$ is
a birational invariant, $p_g(X')=1$.  For the Hodge numbers it
therefore suffices to show that $e(X')=24$ in all cases, since then
$h^{1,1}(X')= 24-4=20$.  We now treat the four cases separately.

\par Type 1.  The normal form given in Proposition~\ref{prop:GIT} (a)
shows that the surface always passes through the singular point
$\mathsf P_2=(0:0:1:0)\in \bP(1,2,3,7)$. We use the techniques as
described in \S~\ref{sec:genwphs}.  The affine piece $x_2\not=0$
containing $\mathsf P_2$, is the $\bZ/3\bZ$-quotient of $\bC^3$ with
coordinates $\set {z_0,z_1,z_3}$ and the surface is the quotient of
the smooth surface $g= 0$ with $g=z_1+ $ higher order terms.  Since at
the origin $\nabla g=(0, 1,0)$, $z_0, z_3$ are local coordinates and
the $\bZ/3\bZ$-action is given by $(z_0, z_3) \mapsto (\rho z_0,
\rho^7 z_3)= (\rho z_0, \rho z_3)$, $\rho$ a primitive root of unity.
This gives a singularity of type $\frac 13 (1,1)$ which is resolved by
a rational curve of self-intersection $(-3)$.
 From Table~\ref{tab:Main} we see that $b_2(X)= 2 +1+h^{1,1}_{\rm prim} = 21$
and hence $e(X)=23$.  
Since the  singularity is resolved by one rational curve, 
$e(\widetilde X)= 23-1+2= 24$, and so $K^2_{\widetilde X}=0$  by Noether's theorem.

\par For type 2 this  is similar, but now the surface always passes through
$\mathsf P_3=(0:0:0:1)$. The affine piece $x_3\not=0$ is the $\bZ/
5\bZ$-quotient of $\bC^3$ with coordinates $\set {z_0, z_1, z_2}$ and
the surface is the quotient of the smooth surface $g= 0$ with $g=z_1+
$ higher order terms as we see from the normal form from
Proposition~\ref{prop:GIT} (b).  Thus the surface has a singularity of
type $\frac 1 5 (1, 3) $.  It is resolved by a chain of two
transversally intersecting rational curves of self-intersections $-2$
and $ -3 $ respectively.  Table~\ref{tab:Main} now gives
$e(X)=22$. The singularity is resolved by a chain of two rational
curves and so $e(\widetilde X)= 22-1+ 3= 24$ and then 
$K^2_{\widetilde X}=0$ by Noether's theorem.

\par Let us now investigate type 3. Here we have two singularities
at $\mathsf P_2$ and at $\mathsf P_3$.  As in the previous cases we
find that the former is of type   $\frac 15 (1,1)$, resolved by a
single $(-5)$-curve, while the latter is a $\frac 17
(1,5)$-singularity resolved by a chain of three rational curves of
self-intersections $-2,-2,-3$ respectively.  From Table~\ref{tab:Main}
we find that $e(X)= 22$ and so $e(\widetilde
X)= 22-2+2+4=26$ implying that $\widetilde X$ becomes minimal after
twice successively blowing down. The resulting surface $X'$ then has
$e=24$ and $K^2_{ X'}=0$ as it should.

\par Finally, let us pass to type 4. Here there is one singularity at
$\mathsf P_3$ of type $\frac 17 (1,2)$ resolved by a chain of two
rational curves with self-intersections $-4$ and $-2$.  Using
Table~\ref{tab:Main} we find $e(\widetilde X)=e(X)-1+3=23 -1+3=25$ and so
$\widetilde X$ contains one exceptional curve.  Blowing down gives
$X'$ with $e(X')=24$ and $K^2_{ X'}=0$.
\end{proof}

\begin{remark}
We could also calculate $K_{\widetilde X}^2$ using Reid's calculus of
discrepancies, i.e., using an expression of the form $
K_{\widetilde X}= \sigma^*(\omega_X)+ \Delta$,
where $\sigma:\widetilde X \to X$ is the minimal resolution of $X$ and
$\Delta$ is a $\bQ$-divisor with support on the exceptional divisors.
For instance in case (c) denote the exceptional chain at $\mathsf P_2$
by $E$ and at $\mathsf P_3$ by $F_1,F_2,F_3$.  Then the discrepancy
divisor is $\Delta= -\frac 35 E- \frac 17(F_1+2F_2+3F_3)$ and
$\Delta^2= \frac{78}{35}$. Then   
\[
K_{\widetilde X}^2=(\cO(1) +\Delta)^2= \frac{8}{35}-  \frac{78}{35} =-2.
\]
\end{remark}

\subsection{Generalities on elliptic surfaces}  \label{ssec:GenEllFib} An elliptic surface is a surface
$X$ admitting a holomorphic map $f:X\to C$,  where $C$ is a smooth curve and the
general fiber of $f$ is a smooth genus $1$ curve. Such a fibration $f$ is called
an \textbf{\emph{genus $1$ fibration}}.\footnote{Note that a scheme-theoretic
fiber $f^{-1}s$, $s\in S$ of $f$ may be multiple, say $f^{-1}s= m F_0$, $F_0$
reduced.}  We assume that $X$ does not contain $(-1)$-curves as a component of
a fiber of $f$.

\par The possible singular fibers of an elliptic fibration have been
enumerated by Kodaira. See e.g. \cite[Ch V, \S 7]{4authors}.  These
are the non-multiple fibers of types $I_b$, $b\ge 1$, $I\!I, I\!I\!I,
I\!V$, $I^*_b, b\ge0$, $ I\!I^*, I\!I\!I^*, I\!V^*$, where the
irreducible fibers are $I_1$ with one ordinary node and $ I\!I$ with
one cusp.    A type \emph{III}  fiber consists of two $-2$ curves touching each other in one point, 
explaining the Euler number $2\times 1 +1=3$.  
The multiple fibers are multiples of a smooth fiber or of a
singular fiber of type $ I_b$, $b\ge 1$.  For the purpose of this
article, in Table~\ref{tab:Kod} we also give the type of lattice that
occurs after omitting one irreducible component of multiplicity $1$.

 \begin{table}[ht]
\caption{Non-multiple  singular fibers of an elliptic fibration} \label{tab:Kod}
\begin{center}
\begin{tabular}{|c|c| c|}
\hline
Kodaira's notation & lattice component & Euler number   \\
\hline
$I_b, b\ge 2$ &   $A_{b-1}(-1) $ & $b$ \\
$I_1$           &   &  $1$ \\
$I\!I$            &    & $2$ \\
$I\!I\!I$      & $A_1(-1)$ &$3$ \\
$I\!V$      & $A_{2}(-1) $ &$4$ \\
$I^*_b$    & $D_{4+b}(-1)$ &$b+6$ \\
$I\!I^*$      & $E_8(-1)$  &$10$ \\
$I\!I\!I^*$ & $E_7(-1)$  & $9$ \\
$I\!V^*$   & $E_6(-1)$  & $8$ \\
\hline
\end{tabular}
\end{center} 
\end{table}

The Kodaira dimension $\kappa(X)$ of an elliptic surface can be equal
to $-\infty$, $0$ or $1$. This can be determined from the plurigenera
$P_m(X)=\dim H^0(X,K_X^{\otimes m})$, $m\ge 1$.  For instance,
$\kappa(X)=1$ if at least one plurigenus is $\ge 2$. More generally,
one applies the canonical bundle formula:

\begin{proposition}[\protect{\cite[Ch. V, \S12]{4authors}}]  \label{thm:CanBunForm} Let $f:X\to C$ 
be a genus $1$  fibration and let  $g=\text{genus}(C)$. Assume that $\sett{m_iF_i}{i\in I}$ is the set of multiple fibers
($I$ is finite but possibly  empty).
There is a divisor  $D$ on $C$ of degree $d:=\chi(\cO_X)+ 2g-2$ such that the canonical divisor  $K_X$ of $X$ is given by
\[
K_X= f^*D+ \sum _{i\in I}  (m_i-1)F_i.
\]
With  $\delta:= d+ \sum_{i\in I}   (1- m_i^{-1})$, one has
\[
    \begin{cases}\delta <0 &\iff   \kappa(X)=-\infty\\
\delta=0 &\iff   \kappa(X)=0\\
\delta >0 &\iff   \kappa(X)=1.
\end{cases} 
\]
\end{proposition}

\begin{corollary} \label{cor:OnElFibs}
A genus $1$ fibration $X\to \bP^1$ (on a minimal surface $X$) with
$p_g=1,q=0$ and at least one multiple fiber has Kodaira dimension $1$.
\end{corollary}

\subsection{The elliptic fibrations on the four classes of surfaces}  
First a preliminary observation. From the invariants of $X'$ given in
Proposition~\ref{prop:InsAndSIngs} coupled with the classification of
algebraic surfaces~\cite{4authors}, we infer that $X'$ either is a
(minimal) K3 surface or a (minimal) properly elliptic surface. In both
cases the canonical divisor $K_{X'}$ has self-intersection $0$.  We
show below that $X$ has a pencil of genus $1$ curves. On $X'$ the
resulting pencil $|F'|$ then necessarily is fixed point free since
$F'\cdot F'=0$ (by the genus formula) and hence gives a holomorphic
map $ X'\to \bP^1$.

\begin{proposition}\label{prop:Main1}
 The rational map on $X$ given by $(x_0:x_1:x_2:x_3)\mapsto (x_0^2:x_1)\in \bP_1$
induces an elliptic fibration $ \pi:X'\to \bP^1$. For type 1 and 2    the surface
$X'=\widetilde X$ has Kodaira dimension $1$ and  for type 3  and  4   the minimal
surface $X'$ is a K3 surface, a surface of Kodaira dimension $0$.
\par
Furthermore,  for types 1 and 2, resolving the  orbifold singularity  yields  a   rational curve of selfintersection $-3$ which is a
bisection for the elliptic fibration. 
   For types 3 and 4,   on the minimal model $X'$ there is a rational curve with   selfintersection $-2$ originating from
the orbifold singularities which is a section for the elliptic fibration.  
\par

The generic fiber type (i.e. number and type of singular fibers)  has been summarized in the table below.   

\begin{center}
\begin{tabular}{|c|c| c|c|}
\hline
 case        &  $\pi^{-1}(0:1)$&  $  \pi^{-1}(1:0)$ & remaining singular fibers   \\
& or  $\lambda=\infty$ & or   $\lambda=0$ &\\
\hline
$  (14,[1,2,3,7]) $      &   $2I_0$ & $I_1$ & $23 \times I_1$\\
$ (12,[1,2,3,5])$    &      $2I_0$      &  $I_2$ & $22 \times I_1$\\
$ (16,[1,2,5,7])$     &   $I\!I$  & $I_3$ & $19 \times I_1$\\
$  (22,[1,2,7,11])$   &  $I_1$ &$I_0$& $23 \times I_1$ \\
\hline
\end{tabular}
\end{center} 
 \end{proposition}
\begin{proof}  
\textbf{Step 1.} \emph{The general fiber of $\pi$ in all cases is a smooth
genus $1$ curve.}
\\
  The fiber of $\pi$ can be viewed as a curve $C$ of
degree $d$ in $\bP(1,a,b)\subset \bP(1,2,a,b)$ passing through the
singular points of $\bP(1,a,b)$ which are the same as those of
$X$. Its equation is obtained by eliminating $x_1$ from the equation
of the surface.  Note that its amplitude is $a+b+4-a-b-1=3$ so that
$\omega_C=\cO_C(3)$ which has two sections, $x_0^3,x_2$, for type 1
and 2, and one section, $x_0^3$, for type  3 and 4.
 
  For types 1 and 2 the curve $C$ has one  ordinary double point at the unique
singularity of $X$.  Resolving the singularity of $X$ separates the
branches on $\widetilde X$ so that the resulting curve is smooth and
has genus $1$.   For types 3 and 4  the curve $C$ is already a smooth
genus $1$ curve.

\textbf{Step 2.}  \emph{Determining  the Kodaira dimensions}.
\\
We already saw in \S~\ref{ssec:Invs} that quasi-smooth   surfaces of type 3 and 4  are K3 surfaces.
We next show that  for types 1 and 2   the surface has Kodaira dimension $1$.   Since the
pencil is given by $(x_0^2 : x_1)$ there is a double curve over $(0:
1)$. The weighted  plane $x_0=0$ ($=\bP(2,a,b)$) cuts the surface in this curve in which it has amplitude $2$.
It is  a quasi-smooth curve on the  basic surface (and hence on the general
quasi-smooth surface) and passes through the unique orbifold point of $X$.
So on $\widetilde X$ this curve
is a smooth elliptic curve, the reduction of a double fiber of type
$2I_0$.   The surface has positive Kodaira dimension in both cases since
$P_2(\widetilde X)=2$ (note that $2K_{\widetilde X}$ is a fiber of
the pencil which moves in a linear system of projective dimension $1$).

\textbf{Step 3.}  \emph{Determining the fiber types.}
\\
 To verify the fiber types of the following special surfaces we used  \textsc{SageMath}.
Firstly  to show their quasi-smoothness and secondly  to calculate certain
discriminants;  quasi-smoothness  also has been verified manually.
\footnote{See ~\ref{app:SAGE} for the SAGE code we used for checking quasi-smoothness.}
 
 \textbf{\emph{Type 1}}.
It  suffices to establish this for one example, for which we take the  quasi-smooth 
surface
$x_0^{14}+ x_1^7 + x_2^4x_1+ x_3^2+ x_0^{11} x_2 + x_0^5 x_2^3 =0$.  
Away from the the plane $x_0=0$  we may assume $x_0=1$, $x_1=\lambda$ which gives the
inhomogeneous equation
 \begin{equation*}  
   \lambda z_2^4 + z_2^3 +z_2+ (1+\lambda^7)+ z_3^2  =0.
 \end{equation*}
This equation describes a varying double cover of $\bP^1$ branched in $4$
points. Its singular members are found from the discriminant of the left hand 
with respect to $z_2$ which is the degree $24$ polynomial 
$-256 \lambda^{24} + 768\lambda^{17} - 192\lambda^{16} - 27\lambda^{14} - 768\lambda^{10} + 384\lambda^9 + 6\lambda^8 + 54\lambda^7 + 256\lambda^3 - 219\lambda^2 - 6\lambda - 31$ with non-zero discriminant
and so there are $24$ singular fibers, necessarily of type $I_1$, as claimed.
  
 \textbf{\emph{Type 2}}.  The resolution of the unique orbifold point produces a
$-3$ curve (which is a bisection) and a 
$-2$ curve which must be part of the  reducible fiber over  $(1:0)$.
 To find the generic fiber type,   consider the special  quasi-smooth  surface
 \[
 x_3^2x_1 +2 x_3x_0 (x_2^2 +x_0^6) + x_2^4+ x_1^3 x_2^2+x_0x_1^4x_2+x_1^6=0.
 \]
  As before, set $x_0=1$, $x_1=\lambda$, $x_j=z_j$, $j=2,3$.
 The elliptic fibration is given by
  $-w^2 = (\lambda-1) z_2^4+ (\lambda^4-2) z_2^2+\lambda^5z_2+(\lambda^6-1) $, where $w= \lambda z_3+  2(z_2^2+1)$.
  One can check   that   in this chart the surface is smooth. The verification on the remaining points of the surface is easy using that 
  for these   $x_0=0$.
  
 The discriminant of the left hand side  of the above equation  with respect to $z_2$ is the degree $24$ polynomial 
\begin{align*}
 \lambda^2(144\lambda^{22}-  388\lambda^{21} &+ 329\lambda^{20}   - 58\lambda^{19} + 357\lambda^{18}
-1160 \lambda^{17} + 1064\lambda^{16}-  816 \lambda^{15}  \\
  & + 1536\lambda^{14}- 1168 \lambda^{13}+ 160\lambda^{12}  - 896 \lambda^{11} + 1696 \lambda^{10}- 1056 \lambda^{9}   \\
  & + 896\lambda^8- 1408\lambda^7 +768 \lambda^6 + 512  \lambda^4  - 512  \lambda^3  - 256\lambda  +256),
\end{align*}   
whose second factor is without multiple factors since the discriminant is a (huge) non-zero integer,
which shows  the claim. 
  
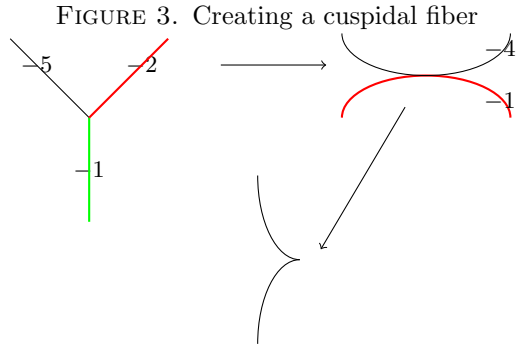
\begin{figure}[htbp]
\begin{center}

\caption{Creating a cuspidal fiber}
\label{fig:Cusp}

 \begin{tikzpicture}[scale= 0.7]
 
\begin{scope}[ shift={( -4,0)}]
\draw (0,0)--(-1.5,1.5);  \node at (-1,1) {\small $-5$} ; \node at (0,-1) {\small $-1$} ;\node at ( 1,1) {\small $-2$} ;
 \draw[thick,green,rotate=135] (0,0)--(-1.4,1.4);  
 \draw[thick,red,rotate=270] (0,0)--(-1.5,1.5); 
 \end{scope}
 
 \draw[->] (-1.5,1)--(0.5,1);
  \draw[->] (2, 0.2)--(0.4, -2.5);
  
  \draw[thick,red]   (4,0) arc [start angle=  0, end angle=  180,  x radius=16mm, y radius=8mm];
  \node at (3.8,1.3) {\small $-4$} ;
   \node at (3.8,0.3) {\small $-1$} ;
   \draw[rotate=180]  (-0.8,-1.6) arc [start angle=  0, end angle=  180,  x radius=16mm, y radius=8mm];

\begin{scope}[ shift={( 0.5,-4)}]

 \draw (-0.5,1.3) arc [start angle= -90, end angle=  -180,  x radius=8mm, y radius=16mm];
 \draw[rotate=180]   (0.5,-1.3) arc [start angle=  -90, end angle=  0,  x radius=8mm, y radius=16mm] ;
 \end{scope}
\end{tikzpicture}
\end{center}
\end{figure}

  \textbf{\emph{Type 3.}}    
Consider the special example of a quasi-smooth surface  given by  
\[
x_1 A+x_0B =0, \quad A=  x_1^7-x_3^2 ,\, B= x_0^8x_3+x_0^3x_2x_3+
 (x_0^{15} +x_0^{10} x_2+ x_0^5x_2^2+ x_2^3 )  \]
and  the genus $1$ fibration given by the rational map $\pi$.  The
line $\mathsf P_{23}$ given by $x_0=x_1=0$ lies on the surface and so
is the indeterminacy locus of this map.  It contains the two singular
points $\mathsf P_2$ and $\mathsf P_3$. The plane $x_0=0$ contains
this line as well as the curve $C_A$ given by $ A=x_0=0$; the plane
$x_1=0$ also contains the line as well as the curve $C_B$ given by
$B=x_1=0$. This curve 
  is  rational and has  a singularity at $(0:0:1)\in \bP(1,5,7)$,
corresponding to the singularity $ \mathsf P_3 $ on the surface $X$. 
On $  X'$ this gives rise to  an  $I_3$-type fiber at $\lambda=0$. 
In fact, on $X'$ the proper transform of the resolution of $ \mathsf P_3 $ consists of three
$-2$-curves, two of which together with the proper transform of $C_B$ yields the $I_3$-configuration
and the remaining $(-2)$-curve (which comes from a $( -3)$-curve on $\widetilde X$) is a section.

Similarly, one shows that $C_A\subset \bP(2,5,7)$ is a smooth rational
curve passing through $(0:1:0)$ corresponding to
$\mathsf P _2 \in X$. On $\widetilde X$ this becomes a $(-2)$-curve $F$ meeting
the total transform $E$ of the line $L_{01}$ transversally.  One has
$E\cdot E=-1$ and $K_{\widetilde X}= 2E+F$ so that
$c_1^2(\widetilde X)= -4+4-2=-2$ in agreement with $c_2(\widetilde X)=26$.
On $X'$ this gives a type $I\!I$-fiber at $\lambda=0$ as explained in
Figure~\ref{fig:Cusp}. Note that the original fibration $(x_0:x_1:x_2:
x_3)\mapsto (x_0^2:x_1)$ has a double fiber $x_0^2=0$ but it becomes
absorbed as a multiplicity $2$ component in a fiber which together
with another $(-1)$-curve contracts to the cusp at $\lambda=\infty$.

In the chart $x_0=1$, the elliptic fibration is given by $x_1=\lambda$, which gives the family
\[
\lambda^8-\lambda x_3^2+ x_3 +x_2x_3 +(1+x_2+x_2^2+x_2^3)=0.
\]
Multiplying by $\lambda^3$ and making a change of variables $x_2\lambda =x,
x_3\lambda^2=y$, this yields
 \[
y^2-xy-\lambda y=  x^3+\lambda x^2=\lambda^2 x+ \lambda^3+\lambda^{11} .
\]
The discriminant of this elliptic curve equals 
\[
 t^3(-27 t^{19}-76 t^{11}+30 t^{10}-3 t^9+ \frac 1 {16} t^8 -44 t^3-28t^2-5t+\frac 18),
\]
where the degree 19 polynomial in $t$ can be shown to have non-zero discriminant.
 This shows that away from $\lambda=0$ and $\lambda=\infty$ there are $19$ irreducible type $I_1$-fibers as claimed.
 
 \textbf{\emph{Type 4.}}  
The plane $x_0=0$ intersects the general surface   in a rational curve 
which on the desingularization $\widetilde X$   becomes a $-1$-curve (with multiplicity $2$) intersecting the $-4$-curve in two points
and thus on $X'$ this becomes an $I_1$-type fiber.
On  a general quasi-smooth type 4  surface the elliptic fibration has $23$ further
$I_1$-type fibers as one sees for instance by computing the discriminant with respect to $z_2$ of the 
left hand side of the expression
$z_2^3+\lambda^4z_2^2+(\lambda^7+\lambda^3+1)z_2+\lambda^{11}= z_3^2$,
which represents the elliptic fibration for the quasi-smooth surface 
$x_0x_2^3+x_1^4 x_2^2+(x_0 x_1^7+x_0^9 x_1^3+x_0^{15})x_2+x_1^{11}= x_3^2$.
This discriminant is the  degree $23$ polynomial 
\begin{eqnarray*}
  4\lambda^{23} - 8\lambda^{22} - 4\lambda^{21} + 20\lambda^{18} - 12\lambda^{17} + 20\lambda^{15}  - 11\lambda^{14} - 12\lambda^{13}  \\
\quad \quad  \quad + 2\lambda^{11} - 24\lambda^{10}- 4\lambda^9 + a^8 - 12\lambda^7 - 12\lambda^6 - 12\lambda^3 - 4 
\end{eqnarray*} 
and it has no  double roots which shows that there are indeed $23$ type $I_1$ fibers away from $(0:1)$.
Moreover, substituting $x_1=0$ shows that the fiber at $(1:0)$ is smooth elliptic,
confirming that $\lambda=0$ is not a root of the discriminant. 
 \end{proof}

\begin{remark}\label{rmk:bToa} If we check what happens under the birational transformation given in Remark~\ref{rmk:OnNormForms}.1 which transform a type 2  surface
in a type  1  surface,  only the fibers over $t=0$ are affected: for  the    type 2 surface we have a type  $I\!I$-fiber and the bisection meets each component in a single point,
while for the  type 1 surface the component of the type  $I\!I$-fiber coming from the quotient singularity contracts, giving a fiber  of type $I_1$ whose singularity lies on the bisection.
\end{remark}
 
 The fiber structure of an elliptic pencil allows us to calculate the so-called \textbf{\emph{trivial Picard lattice}}, that is the lattice spanned by the fibers and one (multi)section. 
 There might be more (multi)sections, enlarging the Picard lattice. This is however not the case, as we show now. Note that the argument  is different
   for all types 1--4.
   
 \begin{corollary} \label{cor:OnPic}  The  Picard lattice  $\Pic (X')$ of the minimal model $X'$ of the generic member of the four families 
 coincides with the trivial lattice. In fact, one has:
  \begin{description}
\item[for type 1:] $\Pic (X')\simeq \la 1\ra\operp\la -1\ra  $;
\item[for type 2:] $\Pic (X')\simeq  \la 1\ra\operp\la -1\ra \operp\la -2\ra $;
\item[for type 3:] $\Pic (X')\simeq U \operp A_2(-1)$;
\item[for type 4:] $\Pic (X')\simeq U  $.
\end{description}
 \end{corollary}
 \begin{proof} For  \textbf{\emph{type 1}}  the Picard lattice  
contains a half fiber  $F_0$ (the canonical curve) and the bisection  $E $ with $E.E=-3$ coming from blowing up the singularity.  Denote their classes by $f, s $. Then $f. s= 1$, $f.f=0$, $s.s= -3$.
Passing to the classes $s+2f, s+f$, one finds that the $\bZ$-span of the classes gives a lattice isometric to $S:=\la 1\ra\operp\la -1\ra  $.
This lattice is primitive.
\\
For \textbf{\emph{type 2}} the Picard lattice contains $F_0$, the bisection given by the exceptional curve $E $ with $E.E=-3$,   
and a reducible fiber $G+G'$ of type $I_2$.   Denote their classes by $f, s,  g, g' $.
Passing to the classes $s+2f, s+f,  -f+g   $,  the $\bZ$-span of the classes gives a lattice isometric to $\la 1\ra\operp\la -1\ra \operp\la -2\ra $.
   \\
\textbf{\emph{Type 3}} concerns an elliptic  K3 surface with a  section and generically one reducible fiber of type  $I_3$.
 Hence the trivial Picard lattice is isometric to $U \operp A_2(-1)$.  
 \\
  \textbf{\emph{Type 4}}  concerns an elliptic  K3 surface  with a section 
 and generically no reducible fibers so that the trivial Picard lattice is isometric to $U$.

 To show that generically the Picard lattice equals the trivial lattice,  in \textbf{\emph{type  1  and 2)}}, we bound  the Picard number. 
 We claim that  for type 1   the   Picard group does not have   rank $\ge 3$.
For this, it suffices to find a family admitting a non-symplectic automorphism $g$ of order $11$ 
having at least $1$ modulus and for which  the period map is non-constant. Indeed,  
by\footnote{The proof only uses the K3-type intersection lattice of the surface.}  \cite[Cor. 1.14 in Ch 15]{K3lects}, the transcendental lattice 
of a surface in that family has rank divisible by $10=\phi(11)$ and so is either $10$ or $20$.  If it were 
$10$ for all surfaces in the family, the Picard number would be constant  and so also
the period map would be constant.  
\par
The required family is given by $F_{a}=x_1x_2^4+ a x_1^7+x_0^{11}x_2 + x_1^4x_2^2 -x_3^2$, $a \in \bC$. 
Its members are generally quasi-smooth, the tangent
to the kernel of the period map at each quasi-smooth member is in the (fixed) direction   
$x_0^9x_1x_2$ which can be shown to be  independent  from 
  the deformation direction $x_1^7$ of the family. Finally,   $F_{a }$ admits the automorphism
$g(x_0,x_1,x_2,x_3)=(\rho_{11} x_0,x_1,x_2,x_3,x_4)$ which sends the form $\omega$ which is the residue  along $F_{a }=0$ of the
form $ \Omega_3/ F_{a }$ to $\rho_{11}\omega$ and so the action is indeed non-symplectic. 

The proof that for type 2  the Picard group has rank $3$ has been relegated to ~\ref{app:nijgh}. The proof has  an arithmetic flavor
and uses reduction mod $2$ and $3$.\footnote{After ~\ref{app:nijgh} was ready we realized the existence of 
the birational transformation of Remark~\ref{rmk:OnNormForms} which shows that this gives another proof (cf.\ Remark~\ref{rmk:bToa}).}

 \textbf{\emph{Type 3 and 4}} concern  families of K3 surfaces for which the number of moduli equals the dimension of the period  domain.
Indeed, for type 3 one has $16$ moduli, for type 4   there are $18$ moduli (see Table~\ref{tab:Main}).
We just calculated the trivial Picard lattice which has rank $4$, respectively $2$ and so the dimension of the period domain associated to
the transcendental lattice is $\le 22-4-2=16$, respectively $\le 22-2-2=18$. We shall prove that the period map for a modular family in
both cases  generally  is an immersion (see Proposition~\ref{prop:GenPerMap}) and so equality holds which implies that the Picard lattices are as stated. 
\end{proof}

\begin{remark}\label{rmk:OnLogTfs} 
\textbf{1.} The elliptic fibrations on $X'$ of types 1 and 2
 having a single smooth double fiber $2F_0$ admit an inverse
logarithmic transformation (cf. \cite[\S V.13]{4authors}) which
leaves the fibration outside the double fiber intact but replaces
$2F_0$ by a smooth fiber which is no longer a double fiber. The
resulting surface thus is a K3-surface $X''$. Note that this
procedure changes the Kodaira-dimension!
Since one can perform a logarithmic transformation on any smooth
fiber of the resulting fibration on $X''$, one can in this way
construct elliptic surfaces, say $Y_t$, $t\in \bP \setminus
\set{(0:1)}$, that are not obtained from surfaces  that like $X$ are weighted hypersurfaces of
degree $14$ in $\bP(1,2,3,7)$.  The
Picard lattice and the transcendental lattices being the same as for
$X$, the surfaces $Y_t$ are projective and their period point belongs
to the same period domain.
\\
\textbf{2.} The statements 
 for types 3 and 4   can be used to confirm the specific calculations in S.M.
Belcastro's thesis~\cite{belcas}  for the  numbers 22 and 23  in Reid's
list since these are birational to our type 3 and 4.   We find 
  $\NS{}=  \operp^2 D_4(-1)\operp U(2 )$, respectively $\NS{}= \operp D_5(-1) \operp D_4(-1)\operp U(2)$.
For details on the calculation for  the first of these, see Example 4.5.7 in the forthcoming book~\cite{quadforms}.
 
\end{remark}

 \section{Hodge theoretic aspects: the pure variation} 
\label{sec:Hn}

 \subsection{The period map}

\par The existence of a
double fiber in the elliptic fibration causes  Torelli to fail  everywhere.
This was observed already by K.~Chakiris~ \cite[Theorem 2]{ChakCounter}.
We give a simple proof which shows that in these cases infinitesimal Torelli
always fails.  We give it here  because the geometric proof in
\cite{ChakCounter} is only sketched.

\begin{proposition}\label{prop:ChaksThm} Let $X$ be an elliptic surface fibered
over $\bP^1$ with a unique multiple fiber $m F_0$, $m\ge 2$, and such that
$K_X\simeq (m-1)F_0$.   The period map for a Kuranishi 
family   of such elliptic surfaces 
has everywhere $1$-dimensional fibers. This holds in particular for the
classes  a) and  b)  from Table~\ref{tab:Main}.\footnote{The Kuranishi family is not the same as the
modular family from  Definition~\ref{dfn:ModFam}; the latter has a fixed polarization.}
\end{proposition}

\begin{proof} First note that $p_g(X)=1$ and so $(m-1)F_0$ is the unique
canonical divisor.  Also $q(X)=0$, e.g. because of Theorem~\ref{thm:CanBunForm}.
We reason as in F. Catanese \cite[p. 150]{ttag}.  The failure of infinitesimal
Torelli is caused by the non-trivial kernel of the tangent map to the period
map. The latter is the map
\begin{align}
\label{eqn:mu}
 \mu:  H^1(T_{X}) \to \hom\left(H^0(K_{X}) \to H^1(\Omega^1_{X}) \right).  
\end{align} 
Since $T_{X}\simeq \Omega^1_{X} \otimes K_{X} ^{-1}$, the morphism
$\mu$ is induced by multiplying $H^1(T_{X}) $ by a non-zero section
$\omega$ of $K_{X}$ vanishing along the canonical divisor
$K=(m-1)F_0$.  So, from the exact sequence
\[
   0= H^0(\Omega^1_{X})  \to
   H^0(\Omega^1_{X}\otimes \cO_K)\to H^1(\Omega^1_{X}(-K_{X}))
   \mapright{\cdot\omega}   H^1(\Omega^1_{X}),
\]
one sees that the kernel of $\mu $ is isomorphic to $H^0(\Omega^1_{X}
\otimes \cO_K)$. The problem now is that $K$ is not reduced as soon as
$m\ge 3$.  
\par
If   $m=2$, and  the multiple fiber is of type $2I_0$, the normal bundle
sequence for $K\subset X$ reads
\[
  0\to \cO_{K} (-{K})  \mapright{\cdot\omega}  \Omega^1_{X}\otimes \cO_K  \to \Omega^1 _{K} \to 0 
\]
and since $\cO_K(-K)$ is a torsion line bundle on the elliptic curve
$K$, the exact cohomology sequence shows that $H^0(
\Omega^1_{X}\otimes \cO_K) \simeq H^0(\Omega^1 _{K } ) \simeq
H^0(\cO_C)$, which is $1$-dimensional.   If we have a multiple fiber of
type $2I_b$,  $b\ge 1$,  the argument is essentially the same. 
\par
 If $m\ge 3$ the argument is more involved. We sketch it only for $m=3$ so that
$K_X=2F_0$. One now uses the so-called decomposition sequence for
reducible divisors $D=A+B$ which reads (cf.\ \cite[Ch. II.1]{4authors})
\[
  0\to \cO_A(-B)\to \cO_C\mapright{\rm restr} \cO_B \to 0.
\]
We apply it to $K={F_0}+{F_0}$ and tensor it with
$\Omega^1_X|{F_0}$. This introduces the two locally free sheaves
$\Omega^1_X|{F_0}$ and $\Omega^1_X(-{F_0})|{F_0}$ on the elliptic
curve ${F_0}$. The first sheaf fits into the normal bundle sequence
for ${F_0}$ in $X$,
\[
    0\to \cO_{F_0}(-{F_0})\to \Omega^1_X|{F_0}\to \Omega^1_{F_0}\to 0
\]
and since $ \cO_{F_0}(-{F_0})$ is torsion, the same argument as
before gives $H^0( \Omega^1_X|{F_0})\simeq H^0( \Omega^1_{F_0})\simeq \bC$.
The sheaf $\Omega^1_X(-{F_0})|{F_0}$ fits into the normal
bundle sequence twisted by $\cO_{F_0}(-{F_0})$ which reads
\[
 0\to \cO_{F_0}(-2{F_0})  \to \Omega^1_X(-{F_0})|{F_0}\to \Omega^1_{F_0}(-{F_0})\to 0.
\]
Hence  $H^0(\Omega^1_X(-{F_0})|{F_0})\simeq H^0( \Omega^1_{F_0}(-{F_0}))=0$ and  
$H^1(  \Omega^1_X(-{F_0})|{F_0})=0$.
Plugging all this  into the exact sequence of the twisted decomposition
sequence
 \[
 0\to \Omega^1_X(-{F_0})|{F_0} \to \Omega^1_X\otimes \cO_K \to \Omega^1_X|{F_0}\to 0
 \]
shows that $H^0(\Omega^1_X\otimes \cO_K)\simeq H^0(  \Omega^1_X|{F_0}) \simeq \bC$.
By induction one can  show the result for all $m$.
\end{proof}
 
As stated in the introduction, in our case  the  four types of modular  
families give rise to a polarized variation of Hodge structure, each with an associated period domain,
say $D_{a,b}$, and the associated Kuranishi family (with fixed
polarization) gives rise to a (local) period map  $ M_{a,b}  \to
D_{a,b}$.    In our case the
kernel at $F\in \cM_{a,b}$ of the period map   is   precisely the kernel of the
multiplication map $R^d_F \mapright{\cdot x_0} R^{d+1}_F$. This kernel
varies with $F$.  For instance  the calculations in ~\ref{ssec:Manual}  show
that at the point corresponding to the \textbf{\emph{basic type}} this
kernel has dimension $1$ in all cases.  However, at   a general point this  kernel has dimension 0, as shown in ~\ref{ssec:KernelCode},
proving  the following result.

\begin{proposition}\label{prop:GenPerMap} 
The period map for  the
Kuranishi family for a general    type  1--4 surface\footnote{ Recall  our convention on Kuranishi families from Lemma~\ref{lem:kuranishi}}.    is an immersion.
\end{proposition}

 An  interesting arithmetic consequence of this result\footnote{For background on the Tate and the Mumford--Tate conjectures we refer to \cite[\S 1]{moonen}}.  is  the validity of the Tate 
 conjecture as well as the Mumford--Tate conjecture for each  of the present surfaces. Note that the types 1 and 2, being not of general type, are not mentioned in loc.\ cit.
 These two classes of surfaces give the first properly elliptic surfaces (in characteristic $0$)
 for which these conjectures now are known to be true.
  \begin{corollary}  \label{cor:moonen} For each of the  surfaces of types 1--4  the Tate conjecture for divisor classes and the Mumford--Tate conjecture 
 for rank $2$ cohomology holds.
 \end{corollary}
  \begin{proof} This follows from \cite[Prop. 9.2]{moonen}. Indeed, since the moduli stacks of the surfaces are irreducible of positive dimension and 
 condition (P) in loc.cit. follows from the immersive property of the period map, all conditions are satisfied. 
 \end{proof}
       
\begin{remark}L. Tu~\cite{Tu}  has shown that infinitesimal Torelli for weighted
hypersurfaces hinges on the validity of Macaulay's theorem which is
only true if the degree of the hypersurface is high enough.  
Tu's result  \cite[Theorem 2]{Tu} indeed does not apply in the present situation, 
 not  even in the case where the minimal model of  the resolution is a K3 surface obtained
from the basic surfaces of types  3  and 4. 
 \end{remark}

\subsection{The generic transcendental  lattice for the four families} \label{ssec:OnLats}

\subsection*{Digression on lattices} The discriminant form  of a non-degenerate
integral lattice $L$ plays a central role if $L$ is not unimodular.
We already introduced the discriminant group $A(L)=L^*/L$ just above
Theorem~\ref{thm:weylisbig}. Extending the form on $L$ in a
$\bQ$-bilinear fashion to $L\otimes\bQ$, one obtains a well-defined
$\bQ/\bZ$-valued form $b_L$ on the discriminant group by setting
 \[ 
b_L  :   A(L)\,  \times  \, A(L)  \to \bQ/\bZ,   \quad  \bar x . \bar y  \mapsto  x. y  \bmod \bZ \,\text{(\textbf{\emph{discriminant bilinear form}})} .
 \] 
A lattice $L$ such that $x. x$ is even is called an
\textbf{\emph{even lattice}}. These come with an integral quadratic
form $q$ given by $q(x)= \half x . x$ and for these one considers a
finer invariant, the  \textbf{\emph{discriminant quadratic form}}
\[
q_L : A(L) \to  \bQ/\bZ,   \quad \bar x \mapsto   {q(x)} \bmod \bZ.
\]
The discriminant form is a so-called torsion form and such forms are
completely local in the sense that these decompose into $p$-primary
forms where $p$ is a prime dividing the discriminant.  More
precisely, $b_L$ is the orthogonal direct sum sum of the discriminant
forms of the localizations $L_p=L\otimes\bQ_p$ and so it ties in with
the \textbf{\emph{genus}} of the lattice, i.e.\ the set of isometry
classes $\set{L_p}_{p\text { prime}}$ together with
$L\otimes\bR$. The same holds for $q_L$ if $L$ is even. 

\begin{example} \label{exmpl:torsionqf} Some torsion forms play a role later on.
For calculations on the root lattices, see  for example \cite[Table 2.4]{mwlatts}.
\begin{enumerate}
\item The lattice $\la n \ra$ with $n$ even has discriminant group
$\bZ/n\bZ$ and discriminant quadratic form which on $\bar x$ takes
the value $\frac 1n \bar x\in \bQ/\bZ$. The form is denoted $\la
\frac{1}{n}\ra$.
\item The discriminant group of the root lattice $A_n$ is the cyclic
group $\bZ/(n+1)\bZ$. The discriminant quadratic form assumes the
value $-n/(n+1) \in \bQ/\bZ$ on the generator and is denoted
$\la \frac {-n}{n+1}\ra$.
 \end{enumerate}
 \end{example}
 
\par A celebrated result of V. Nikulin~\cite[Cor.\ 1.16.3]{nikulin1} emphasizes
the role of the discriminant form in determining the genus:
 
\begin{thm*} The genus of  non-degenerate lattice is completely determined by
its type (i.e., being  even or odd), its rank, index and discriminant form.
\end{thm*}
 
It is well known that the number of isometry classes in a genus is
finite. It is also called the \textbf{\emph{class number}} of the
genus.  We state a criterion for class number $1$ due to M.\ Kneser
\cite{kneser} and V.\ Nikulin \cite[1.13.3]{nikulin1}:
    
\begin{theorem} Let $L$ be a non-degenerate \textbf{indefinite} even lattice
of rank $r$. Its class number is $1$ if the discriminant group of $L$ can be
generated by $\le r-2$ elements.  Hence, in this case $L$ is uniquely
determined by its rank, index and the discriminant quadratic form.
\label{thm:kneser}
\end{theorem}

\begin{example} \label{exmpl:torsionqfBIS}  Any indefinite odd unimodular
lattice of signature $(s,t)$ is unique in its genus and represented by
a diagonal lattice of the form $\operp^s \la 1 \ra\operp \operp^t \la
-1\ra$.  Any even unimodular lattice of signature $(s,t)$ is unique in
its genus, satisfies $s-t \equiv 0 \bmod 8$ and if $t\ge s$, is
represented by $\operp^{s} U\operp \operp^{\frac{1}{8}(t-s)}
E_8(-1)$. For instance, the intersection lattice of a K3 surface is isometric
to $\operp^3U \operp\operp^2 E_8(-1)$.
\end{example}

Although the lattices we encounter are odd, the preceding result will
be applied to certain even sublattices. Here we use a topological
result which we recall now.  For any compact orientable
$4$-dimensional manifold $X$ the second Stiefel--Whitney class $w_2$
is a characteristic class for the inner product space $H^2(X,\bbF_2)$,
i.e. $w_2.x+x.x=0$ for all classes $x\in H^2(X,\bbF_2)$.  To pass to
integral cohomology one uses the reduction mod $2$ map, induced by the
natural projection $\bZ\to \bZ/2\bZ$:
\begin{align*} 
 \rho_2 :H^2(X,\bZ) \to  H^2(X,\bZ/2\bZ)  .
 \end{align*}
Any lift of $w_2$ under $\rho_2$ is an integral characteristic
element since the intersection pairing is compatible with reduction
modulo $2$.  In the special case where $X$ is a compact almost
complex manifold of complex dimension $2$, there is a canonical
choice for a lift, namely the first Chern class $c_1$. In our
situation we apply it to lattices orthogonal to the class of a fiber
of an elliptic fibration.

\subsection*{Application} 
Recall that in the projective situation the orthogonal complement of
the Picard lattice is the transcendental lattice, the smallest
integral sublattice of $H^2$ whose complexification contains $H^{2,0}$.

\begin{proposition}  \label{prop:translatt}  
The transcendental lattice $T$ for a generic member of the universal
family of quasi-smooth surfaces of degree $d=a+b+4$ in $\bP(1,2,a,b)$
is even  and 
\begin{description}
 
\item[for type 1:] $T\simeq\operp^2 U\operp \operp^2 E_8(-1)$;
\item[for type 2:] $T\simeq \la 2\ra\operp U\operp^2 E_8(-1)$;  
\item[for type 3:] $T \simeq \operp^2 E_{8}(-1)\operp A_{2}$;
\item[for type 4:] $T \simeq \operp^2 U \operp \operp^{2}E_8(-1) $.
\end{description}
  
\end{proposition}
\begin{proof}   Recall that in  the algebraic  situation the transcendental
lattice and the Picard lattice are each  other's orthogonal complement and
both are primitively embedded in $H^2$.  In order to apply
Theorem~\ref{thm:kneser}, we observe that for a non-degenerate primitive
sublattice $S$ of a unimodular lattice $L$ and its orthogonal complement
$T=S^\perp$, one has $A_S\simeq A_T$ while $q_S\simeq -q_T$.

As previously observed, the transcendental lattice is a birational
invariant and so we may and do compute it on the minimal model $X'$ if
it differs from $\widetilde X$.  Proposition~\ref{prop:InsAndSIngs}
states that $H^2(X')$ has rank $22$ and signature $(3,19)$. If $X'$ is
a K3 surface the lattice $H^2(X')$ is even. In the other two cases it
is odd, since the stated multisections have odd self-intersection
number. However, the class $c_1(X)=-K_X=- F_0$ is a characteristic
class and so $x.x$ is even for classes  $x$  orthogonal to the class of a
fiber. In particular, the transcendental lattice is even. 

 \par We
recall that Corollary~\ref{cor:OnPic} gives the Picard lattices for
the minimal model $X'$ of the generic member of each of the four
families.  The Picard lattice for type 1 surfaces is isometric to the
unimodular lattice $ \la1\ra\operp \la -1\ra  $.    Hence
the generic transcendental lattice is even, unimodular and of signature $(2,18)$.  So, using
Example~\ref{exmpl:torsionqfBIS},  it is isometric to 
$  \operp^2  U   \operp  \operp^{2}E_8(-1)$  and a similar argument
applies ifor type 2.

\par  For type 3  the Picard lattice is    isometric to $U \operp A_2(-1)$. 
 The transcendental lattice is an even lattice
 of signature $(2,16)$ and discriminant form the one of $A_2 $.  Up to isometry, 
 there is only one lattice, the   lattice 
 $A_2   \operp \operp^2 E_{8}(-1)  $.  
 For type 4  the Picard lattice is isometric to $U$.
 So by Example~\ref{exmpl:torsionqfBIS},  in this case  
 the generic  transcendental lattice is isometric to 
 $ \operp^2  U   \operp  \operp^{2}E_8(-1)$.
  \end{proof}

\subsection{Lattice polarized variations} \label{sec:LPV}
Recall that in each of the four cases the Hodge structure  on $H^2$
and on $H^2_{\rm prim}$ is  of K3-type since $h^{2,0}=1$. So  the period
domain associated to $H^2_{\rm prim}$ is
 \[
  D(H^2_{\rm prim})= \sett{[\omega]\in \bP(H^2_{\rm prim}\otimes\bC)}
  {\omega\wedge \omega=0, \, \omega\wedge \overline{\omega}>0},
 \] 
a domain of K3-type of dimension $b_2 - 3= h^{1,1}_{\rm prim}$.  Since
the modular families (cf. Definition~\ref{dfn:ModFam}) in all cases
have generic Picard lattice of rank $\ge 2$, the associated period map
is not surjective.  In such cases one uses the smaller domain 
 \begin{equation}
\label{eqn:PerDom}
  D(T)= \sett{[\omega]\in \bP(T\otimes\bC)} {\omega\wedge \omega=0, \, \omega\wedge \overline{\omega}>0},
  \end{equation} 
associated to the general transcendental lattice $T\subset
H^2(X,\bZ)_{\rm prim}$ for the modular family for $X$. One speaks then
of \textbf{\emph{lattice polarized families}} and their associated
variations of Hodge structure. More precisely, these are the
$S$-polarized variations, where $S =T^\perp$ is the generic Picard
lattice. 
 Since the period maps for the generic member of a  modular family are  immersions (cf. Proposition~\ref{prop:GenPerMap}), 
Proposition~\ref{prop:translatt} gives rise  to  the following table:

 \begin{center}
\begin{tabular}{|c|c|c|  }
\hline
cases   &  $\dim D(T)$  &  generic rank of period map   \\
 \hline
  1.   $(14,[1,2,3,7]) $  &      $18$   & $18$   \\
 2.   $(12,[1,2,3,5])$   &      $17 $  & $17$\\
3.    $(16,[1,2,5,7])$   &    $16 $  & $16$ \\
\, 4.  $(22,[1,2,7,11])$     & $ 18$  & $ 18$
 \\
\hline
\end{tabular}
\end{center}

   \section{Associated variation of mixed Hodge structure}

\subsection{Mixed Torelli} \label{ssec:MixedTorelli}

We assume now that $X$ is an elliptic surface of type 1 or 2. 
 
We first  describe   a  subvariety\footnote{There might be other subvarieties where this holds as well.}
  inside  $\cM_{ 3,7}$, respectively inside  $\cM_{ 3,5}$ 
where the infinitesimal Torelli fails  and to which the basic surfaces belong.

\par Let  $W\subset U_{1,2,3,7}$ be  the sublocus of normal forms with  $G_0=G_5=0$. It has
 dimension $13-2=11$. A quasi-smooth surface $V(g)$ with  $g\in W$ is invariant under the involution
$\iota(x_0)=-x_0$ and $\iota$ acts on the Jacobian ring $R$.  The $+1$  eigenspace $ R_{14} ^{+ }$ of $\iota$ on 
$R_{14}$  corresponds to first  order deformations which are tangent to $L$ and transverse to the $T$-orbit,
confirming  that
$\dim (W/\!/T)=\dim  R_{14}^{+ }= 11$.  

The forms $F ( \Omega_3/  g^2)$, $F\in  R_{15}^{- }$ whose residues belong to the primitive  $(1,1)$-part of 
$H^2(V(f))$,  are all $\iota$-invariant and descend to give rational forms on 
 the quotient of $V(g)$ by $\iota$.  This is a  quasi-smooth degree $14$
surface $S\subset \bP(2,2,3,7)$ and appears as number 22 on Reid's list of
 95 K3 surfaces. In the present case the primitive part   $  H^{2}_{\rm prim} (S)$ has Hodge numbers $(1,10,1)$.   
The residue map sending $F\in  R_{15}  ^{- }$ to its residue on $V(f)$   induces  a homomorphism 
$ R_{15} ^{- } \to   H^{1,1}_{\rm prim}(S)$. 
One checks that it is generically an isomorphism and so it follows that the
multiplication map $ \text{mult}_{x_0}: R_{14}^{+ } \to     R_{15}^{- }$ 
is a map from a vector space of dimension 11 to a vector space of dimension 10, and hence always has a 1-dimensional kernel.

In  $U_{1,2,3,5}$ we consider 
 the sublocus $W'$ of normal forms given by $G_4=0$ and the vanishing of two terms in $G_2$ (only
the one containing $x_0^3x_2$ is allowed) has dimension $13$ and so $\dim (W'/\!/T)=\dim  R_{12}^{+ }= 10$

The quotient by the involution $x_0\mapsto -x_0$
is a   K3 surface in $\bP(2,2,3,5)$ of degree 12 which is number 18 in Reid’s list. 
Its  primitive cohomology has Hodge numbers $(1,9,1)$. The same argument shows that 
the multiplication map $ \text{mult}_{x_0}:  R_{12}^{+ } \to    R_{13}^{- }$ in this case 
always has a 1-dimensional kernel.
\par
In these cases
$\widetilde X=X'$ is fibered over $\bP^1$ with a unique double fiber
$2 F_0$ and $K_{X'} \simeq F_0$.   
We set $ U := X' \setminus {F_0}$
and we consider the variation of mixed Hodge structure on $H^2(U)$ when
$X$ varies in  
  one of the  subfamilies of a modular family   coming from the codimension $1$ sublocus  
$W\subset U_{1,2,3,7}$, respectively $W'\subset U_{1,2,3,5}$, where Torelli fails. 
For brevity we call the resulting
variation the \textbf{\emph{subcanonical modular variation of mixed Hodge structure.}}

\begin{theorem} \label{thm:Main}    The period map of the subcanonical   modular  variation of mixed Hodge structure
  of type  1 or type 2 surfaces is an immersion.
\end{theorem} 
\begin{proof}
We first determine the mixed Hodge structure on $H^2(U)$ by means of
the exact Gysin sequence
\[
0\to H^0({F_0})(-1) \mapright{i_*} H^2( X')\mapright{j^*} H^2(U)
\mapright{\text {res}} H^1({F_0})(-1)\to 0,
\]
where $i:{F_0}\into X'$, $j:U\into X'$ are the inclusions and "res" is
the residue map. We see that $W_2 H^2(U)\simeq H^2(X')/H^2({F_0})(-1)$
and that $\gr^W_3 H^2(U)\simeq H^1({F_0})(-1)$. 

 We next consider the variation of mixed Hodge structure given
by $H^2(U_F)$ where $U_F= U\setminus F$ and $F=F_0+ G$ is a deformation of 
a quasi-smooth reference surface  $F_0$,  and $G$  varies over an open
neighborhood of  $F_0$  in a base of a Kuranishi family  as described  by
Corollary~\ref{cor:ModFam}.   So tangent directions are
identified with polynomials $G\in R^d$.  The infinitesimal variation
is described by the Higgs fields\footnote{For background on Higgs fields
  in the mixed setting, see \cite{HolBisect}.}
$\theta^2_\xi : I^{2,0}\to I ^{1,1}$,
$\theta^3_\xi: I^{2,1 }\to I ^{1 ,2}$ in the direction of $\xi$.  The map
$\theta^2_\xi $ is induced by the map $\mu$, c.f. Eqn.~ \eqref{eqn:mu} and
if $\xi$ corresponds to the polynomial $G$, is represented by the
multiplication
\[
R_F^1\mapright{\cdot G} R_F^{d+1}, \quad G\in R_F^{d}.
\]

We consider the following two particular cases (belonging to the locus $L$, respectively $L'$):
 \begin{itemize}
 \item type 1:   $F_0 = x^{14} + y^7 + yz^4 + w^2 + x^2y^3z^2$,
                $\eta=x^{12}y + (1/7)y^4z^2$
\item type 2:  $F_0 = x^{12} + y^6 + z^4 + yw^2 + x^2y^2z^2$,
                $\eta=x^{10}y + (1/6)y^3z^2$
  
 \end{itemize}

\noindent In ~\ref{ssec:Manual}, it is shown that $V(F_0)$ is quasi-smooth and in \ref{ssec:KernelCode} that
$\theta^2_{\xi}$ is injective except in the direction $\xi=\eta$.

\par To calculate $\theta^2_{\eta}$, we consider the family of canonical
curves $E_t$ attached to the surfaces $X_t=V(F_0+t\eta)$.  In case $(a)$,
$E_t$ is defined by $y^7 + yz^4 + w^2 + t(1/7)y^4z^2$.  Moreover,
$H^{1,0}(E_0)$ is generated by $y$ while $H^{0,1}(E_0)$ is generated by 
$y^5z^2$, and hence multiplication by the tangent direction $y^4z^2$      
is injective. For type  2, $E_t$ is defined by
$y^6 + z^4 + yw^2 + t(1/6)y^3z^2$.  Moreover, $H^{1,0}(E_0)$ is generated by
$y$ whereas $H^{0,1}(E_0)$ is generated by $y^4z^2$, and hence multiplication
in the tangent direction $y^3z^2$ is injective.

\par This takes care of the direction in which $\theta^2_\xi $ fails to be
injective, and shows that the period map not only is injective along $W/\!/T $, respectively $W'/\!/T$, but -
 by the lower semi-continuity of the rank function -  at the generic point of the moduli spaces $\cM_{3,7}$ and $\cM_{3,5}$ .
\end{proof}

\subsection{Rigidity: the pure case}

We first consider rigidity for the pure polarized variations of Hodge
structure.  To avoid confusion, we explain the rigidity concept we use
here.  A variation of Hodge structure comes with a period map
$f : S \to \Gamma\backslash D$, where $D$ is a period domain classifying the
kind of Hodge structures underlying the variation, and $\Gamma$ is the
monodromy group of the variation. Rigidity in this setting is a rather
restricted concept:

\begin{definition}
A deformation of a period map $f : S \to \Gamma\backslash D$ consists
of a locally liftable horizontal map $ F:S \times T\to\Gamma\backslash D$
extending $f$ in the obvious way.  If no such
deformation exists except $f\times \id$, the map $f$ is called rigid.
\end{definition}

Recall that the essential part of a K3-type  variation is given by the
variation on the generic transcendental lattice. Since we have a
weight $2$ variation,  one may apply  
\cite[Theorem 3]{MaxRank}. In our case this  implies: 

\begin{prop*} If  the period map associated to the essential part of
a K3-type variation over a quasi-projective variety has rank $\ge 2$, it is
rigid in the above sense.
\end{prop*}

\par 
Taking into account the  possible  failure  
of infinitesimal Torelli over certain subvarieties parametrizing modular families,  
 we thus find the following rigidity
results:

\begin{proposition}  
The essential part of a variation of type  1--4
is rigid if the period map has rank $\ge 2$.
In particular   the variation over a quasi-projective
subvariety $S$ of a modular family of dimension $\ge 3$ is rigid for types 1
and 2, and if $\dim S\ge 2$  for types 3  and    4.
\label{prop:PureRigid}
\end{proposition}

\subsection{Rigidity for mixed period maps} 

 Just as in the proof of \cite[Prop. 7.2.5]{HolBisect}, we deduce from
 the rigidity in the pure case:

\begin{corollary} \label{cor:Main} Consider for a family of  type 1 or type 2 surfaces 
the   complements  of the supports of the canonical
 curve. If the mixed period map of the resulting  family has rank $\ge 3$, the family is rigid. 
\end{corollary}

\begin{proof}\cite[Prop. 7.2.5]{HolBisect} states that it is sufficient to
show the following three properties of a family as described in the
assertion, say ${\widetilde X_s}$, $s\in S$, $S$ smooth and
quasi-projective.
\begin{enumerate}
\item The family of canonical curves in $ \widetilde X_s$ is rigid.
\item The essential part of the K3 variation has a non-constant
period map and is rigid. 
\item The mixed period map is an immersion.  
\end{enumerate}
 
1  follows for the weight one variation from  the moving curve $F_0$ from \cite[Theorem 3]{MaxRank} since
in the
course  of the  proof   of Theorem~\ref{thm:Main} we proved that
the period map for the canonical curves in $\widetilde X_s$ is not
constant. 

 2  is the statement of Proposition ~\ref{prop:PureRigid}
and  3  is Theorem~\ref{thm:Main}.
\end{proof}
 
 \section{An Application to KSBA Theory}
 \label{sec:KSBA}
 
Recall that the  $28$ dimensional moduli space 
${\bf M}$ of surfaces of general type with $K^2=1$, $p_g=2$ and $q=0$ admits a
KSBA compactification $\overline{\bf M}$ proposed by and named after Koll\'ar,
Shepherd--Barron and Alexeev. 
See~\cite{GPSZ,RT} for recent results about these surfaces.

In this section, we give an application of the
results of this paper to the Hodge theory of some of the boundary divisors of
the KSBA compactification.  

\subsection{Overview of the results of \cite{GPSZ}} 
The generic member of ${\bf M}$ has a canonical model which is a
quasi-smooth hypersurface in $\bP[1,1,2,5]$.  After completing
the square, such a surface can be put in the form:
\begin{equation}
     w^2 = f(x,y,z),\qquad \deg(x)=\deg(y)=1,\quad\deg(z)=2,\quad
     \deg(w)=5 \label{eq:horikawa-form}
\end{equation}
For each of Arnold's exceptional unimodal singularity of type
\begin{equation}
  \Sigma = E_{12}, E_{13}, E_{14}, Z_{11}, Z_{12}, Z_{13}, W_{12}, W_{13}
  \label{eq:Arnold-types}
\end{equation}
  a corresponding boundary divisor $D_{\Sigma}$ of the
KSBA compactification $\overline{\bf M}$ of ${\bf M}$ is constructed.

\par To construct the boundary divisors $D_{\Sigma}$, we start with
a triple $(p,q,d)$ of positive integers and let $V$
denote the homogeneous component of degree 10 of $\mathbb C[x,y,z]$
with respect to \eqref{eq:horikawa-form}. Let
$\omega$ be the weight function which is defined on the monomials
of $V$ by the rule
\begin{equation}
\label{eqn:weightrule}
\omega(x^ay^bz^c) = pb+qc-d
\end{equation} 
Then, $\omega$ determines a $\mathbb C^*$-action on $V$ by the linear
extension of the rule:
$$
    t\ast (x^a y^b z^c)
    = \left\{\aligned t^{-\omega}x^ay^bz^c,\qquad\omega\leq 0 \\
                      x^ay^bz^c,\qquad \omega>0 \endaligned\right.
$$
By construction, this action extends continuously to a holomorphic map
$\mathbb C\times V\to V$.

\par Given a polynomial $f\in V$, let
$$
   \mathcal S(f)
    = V(w^2-t\ast f)\subseteq\bP(1,1,2,5)\times\mathbb C
$$
and $\pi:\mathcal S(f)\to\mathbb C$ be the morphism
$\pi((x:y:z:w),t)=t$. Let $\mathcal S_t(f) = \pi^{-1}(t)$.   The fiber
$\mathcal S_0(f)$ is defined by the monomials of non-negative degree
with respect to $\omega$.  For a generic polynomial $f$ of degree 10
and a suitable choice of $(p,q,d)$, the branch curve $B_0$ of
$\mathcal S_0(f)$ has a unique singularity at $(1:0:0)$ of type $\Sigma$
appearing in \eqref{eq:Arnold-types} (see~\cite{GPSZ} for the details).

\par In the setting described in the previous paragraphs, the KSBA
program compactifies the germ of the family $\mathcal S(f)\to\mathbb C$ at
$t=0$ by replacing $\mathcal S_0(f)$ with a new central fiber
$\tilde Z_f\cup\tilde Y_f$ where
\begin{itemize}
\item $\tilde Z_f$ is birational to a $(p,q)$-weighted blow up of
$\mathcal S_0(f)$ at $(1:0:0)$.  The surface $\tilde Z_f$ has at worst
rational singularities and $h(\mathcal O)=1$;
\item $\tilde Y_f$ is an ADE $K3$-surface which is defined by the
monomials of non-positive weight.  More precisely, in terms of
the data $(p,q,d)$, $\tilde Y$ is a degree $d$ hypersurface in
$\bP(1,p,q,d/2)$ when $d$ is even and degree $2d$ in
$\bP(1,2p,2q,d)$ when $d$ is odd.
\item $\tilde Z_f$ and $\tilde Y_f$ are glued together along a common
  $\bP(p,q)$.  
\end{itemize}
Fixing the data $(p,q,d)$ and varying the polynomial $f$ defines a
divisor $D_{\Sigma}$ in the KSBA compactification of ${\bf M}$.
Moreover, by the results of section 4 of~\cite{GPSZ}, there exists a
Zariksi open subset $\mathscr{U}\subseteq D_{\Sigma}$ over which there
exists a flat, proper family $p:\mathscr{S}\to\mathscr{U}$ whose
fibers $p(u) = \tilde Z_u\cup\tilde Y_u$ are surfaces of the
type described above.

\begin{remark}Preliminary calculations show that the  framework described above
is generally applicable to hypersurface degenerations in weighted projective
3-space, provided that certain numerical conditions hold, i.e. a choice of
weight function $\omega$ defines a family of surfaces
$\mathcal S(f)\to\mathbb C$ whose central fiber can be modified by adjoining
a \lq\lq tail surface\rq\rq{} to obtain a KSBA stable limit.  The details
will appear in a follow up to~\cite{GPSZ}.
\end{remark}

\par By~\cite{SZ}, there exists a Zariski open subset
$\mathscr{U}_1\subseteq\mathscr{U}$ over which
$\mathscr{V} = R^2  p_* (\bQ)$ is the underlying
$\mathbb Q$-local system of an admissible variation of
graded-polarizable mixed Hodge structure.
 Therefore, $\mathscr{H} = \text{Gr}^W_2(\mathscr{V})$ is a variation
of pure Hodge structure of weight 2 over $\mathscr{U}_1$.  Given
a $\mathbb Q$-Hodge structure $A$ of weight 2 with $F^3A_{\mathbb C}=0$
let $T[A]$ denote the smallest $\mathbb Q$-Hodge substructure of $A$
which contains $F^2A_{\mathbb C}$.  By the results section 6
of~\cite{GPSZ},  there is a Zariski open subset
$\mathscr{U}_2\subseteq\mathscr{U}_1$ such that
$$
       u\in\mathscr{U}_2 \implies
       T[\text{Gr}^W_2(\mathscr{H}_u)]
       = T[H^2(\tilde Z_u)]\oplus T[H^2(\tilde Y_u)]
$$

\par For generic $f\in V$, let $\varphi_f:\Delta^*\to\Gamma\backslash D$
denote the local period map of $\pi:\mathcal S(f)\to\mathbb C$ near
$t=0$.  By the results of section 6 of~\cite{GPSZ}, $\varphi_f$ has
finite local monodromy, and hence the limit mixed Hodge structure
$H_{\rm lim}(f)$ of $\varphi_f$ is pure.  Moreover,
\begin{equation}
\label{eqn:LHS} T[H_{\rm lim}(f)] = T[H^2(\tilde Z_f)]\oplus T[H^2(\tilde Y_f)].
\end{equation}        
Thus, at the loss of the information contained in the finite monodromy
of $\varphi_f$, we are justified in calling the transcendental
part of $\mathscr{H}$ the limit variation of Hodge structure of
${\bf M}$ along $D_{\Sigma}$.

\begin{remark}The moduli count for the surfaces $\tilde Z_\Sigma$
 in terms of the Milnor number of $\Sigma$ is given by 
$29-\mu_\Sigma$ whereas the moduli count for the surfaces $\tilde Y$ is
$\mu_\Sigma -2$, adding up to $27=28-1$ which suggests that we have a divisor.
That indeed  $D_{\Sigma}$ is a divisor   corresponds to the fact that these
two components can be deformed independently (see~\cite{GPSZ} for details).

 Since $\mu_\Sigma $  is the index of $\Sigma$ in the list
\eqref{eq:Arnold-types}, the above formulae give the moduli for each of the
components. 
\end{remark}
\par 
At this point, it is natural to ask:
\begin{enumerate}
\item What is the birational type of surface $\tilde Z_\Sigma$?
\item Does the period map of the limit variation of Hodge structure
constructed above have positive dimensional fibers?
\end{enumerate}
For the unimodal singularities of types $Z_{11}$, $Z_{12}$,
$Z_{13}$, $W_{12}$ and $W_{13}$, the answer to the first question is that
they are birational to $K3$ surfaces.  As explained in the last section
of~\cite{GPSZ}, one sees this by simply multiplying the defining
equation of $\mathcal S_0(f)$ by $y^2$ and considering the birational
transformation
$$
    (x:y:z:w)\in\bP[1,1,2,5]
    \dashrightarrow (xy:y^2:z:yw)\in\bP(2,2,2,6)\cong P(1,1,1,3)
$$
which converts a limit surface of type $Z$ or $W$ into an ADE K3
surface which is a double cover of $\bP^2$ branched along a
sextic which intersects a line $L$ in a special configuration
(for example, multiplicities $1$, $2$ and $3$ for the $Z_{11}$
singularity).
\par

All we know is that the  answer to the second question is "no" for the $E_{13}$ case, as we now explain.
By Equation~\eqref{eqn:LHS} the period map is governed by
the restriction to the  transcendental part of the $Z$-surface  and of the $Y$-surface.
 The $Z$-surface 
(type $ E_{13}$ in $\bP(1,2,3,7 )$) has injective differential 
by  Table~\ref{tab:kernel}.
 The $Y$-surface is a K3  surface in weighted projective space whose period map has  injective differential.
 Since the two surfaces $Y$ and $Z$ deform independently  this shows that the total period map has injective differential.

 \subsection{Relation with the present paper}
 One of the observations which gave rise to this paper is that for the
 singularity types $E_{13}$ and $E_{14}$, the resulting surfaces
 $\widetilde Z$ are birational to a singular hypersurface of degree 14
 in $\bP[1,2,3,7]$ by simply multiplying the defining equation of
 $\cS_0$ by $z^2$ and considering the birational transformation
$$ (x:y:z:w)\in\bP[1,1,2,5] \dashrightarrow (x_0:x_1:x_2:x_3) =
   (y:z:xz:zw)\in\bP[1,2,3,7]
$$
 In fact, both the $E_{13}$ and $E_{14}$ singularities are subvarieties of
the locus $\mathscr{J}$ of degree 14 surfaces in $\bP[1,2,3,7]$
whose singular locus consists of the orbifold point $[0:0:1:0]$ of
$\bP[1,2,3,7]$ and exactly one  $A_1$-singularity which occurs at
a smooth point of $\bP[1,2,3,7]$.

\par By a result of Burns and Wahl~\cite{buwa}, $\mathscr{J}$ should
have codimension 1 in the moduli of hypersurfaces of type $(14,[1,2,3,7])$. For
completeness, we give a direct algebraic proof here:  The group of
automorphisms of $\bP[1,2,3,7]$ acts transitively on the smooth points
of $\bP[1,2,3,7]$ (consider the orbit of $\xi=[1:0:0:0]$).  Therefore,
we consider the hypersurfaces $V(f)\in\mathscr{J}$ which have an
$A_1$-singularity at $\xi$ and are given as double covers
$x_3^2=g(x_0,x_1,x_2)$.  The condition that $f(\xi)=0$ implies the
vanishing of the coefficient of $x_0^{14}$ in $f$.  The condition that
$(\nabla f)(\xi)=0$ forces the vanishing of the coefficients of $x_0^{12}x_1$
and $x_0^{11}x_2$ in $f$ as well.  Since we are considering double covers of
$\bP[1,2,3]$, the relevant automorphism group consists of invertible
transformations of the form (cf. \eqref{eq:auto-123})
$$
      (x_0:x_1:x_2) \mapsto
      (a_0 x_0:a_1 x_1 + a_2 x_0^2:a_3 x_2 + a_4 x_0 x_1 + a_5 x_0^3)
$$
The subgroup of elements which fix the point $[1:0:0]$ corresponds to
transformations for which $a_2=a_5=0$, an hence (up to scaling) this
subgroup has $5-2=3$ parameters.  The number of monomials of degree
14 in $\bP[1,2,3]$ is 24.  So, the dimension count for
$\mathscr{J}$ is $(24-3-1) - 3=17$.

\par As noted above, the generic surface of type $(14,[1,2,3,7])$ has
elliptic fiber structure $2I_0 + 24\times I_1$.  
Let $S$ be a generic point of $\mathscr{J}$, $E_{13}$ or $E_{14}$.
In ~\ref{sec:OnEn}, it will be shown that in each case the singular locus of
$S$ consists of the orbifold point $[0:1:0:0]$ of $\bP[1,2,3,7]$
and an $A_1$ singularity at a smooth point of $S$.   In
~\ref{ssec:KernelCode} we compute the rank of the period map in the
$\mathscr{J}$-case and the  $E_{13}$-case. 
This is then shown to lead to:
 
\begin{theorem} \label{thm:OnE13&E14}
 The generic $\mathscr{J}$-type surface as well as  the generic   $E_{13}$ surface and the generic $E_{14}$ surface
 is birational to a type $(14,[1,2,3,7])$ surface which has one
 singular point of type $A_1$ at the point $(1:0:0:0)\in \bP[1,2,3,7]$
 and finite quotient singularities where it intersects the singular
 locus of $\bP[1,2,3,7]$. These surfaces are properly elliptic with $p_g=1$.
 Furthermore,
\begin{enumerate}
\item
The  elliptic fiber  type   of the elliptic fibration  in the
$\mathscr{J}$-case and the  $E_{13}$-case, is given by
$2I_0 + I_2 + 22\times I_1$ and in the   $E_{14}$-case  by
$ 2I_0 +  I_3 + 21\times I_1$. 
 
\item
Let $S_{\rm triv}$ be the "trivial" lattice, spanned by the fibers and the "canonical" multisection. 
The invariants in the three cases are given in the following table:
\begin{center}
\begin{tabular}{|c|c|c|c|c}
\hline 
 & $\mathscr{J}$  & $E_{13}$  & $E_{14}$   \\
 \hline
 $\dim S_{\rm triv}$ & $3$& $3$& $4$\\
\hline
Hodge numbers of $S_{\rm triv}^\perp$& ($1,17,1)$& $(1,17,1$)& $(1,16,1)$ \\
\hline
 $\dim $ of $D ( S_{\rm triv}^\perp)$  & $17$ & $17$  & $16$ \\
 \hline
 number of moduli & $17 $  & $16$ & $15$ \\
\hline
 rank of the period map &$17$ & $16$ & $\geq 14$  \\
\hline
\end{tabular} 
\end{center}
If $S_{\rm triv}$ is the entire Picard lattice, then $S_{\rm triv}^\perp$ is the
transcendental lattice and $D ( S_{\rm triv}^\perp)$ is the associated period domain.
\end{enumerate}
\end{theorem}
 
\medskip

\begin{remark}To prove our results on (mixed) period maps also  to the family
$\mathscr J$,  it is required to extend  the residue calculus for
quasi-smooth hypersurfaces to  the situation  where supplementary ordinary
nodes are allowed, e.g. by extending  the results \cite{dimsaito} of A. Dimca
and M. Saito and so  one expects the residue calculus  to involve  working with
polynomials in the Jacobian ring which vanish at the nodes.
Assuming this to be  the case one  argues as follows: 

(a) Rigidity   for the period map of the  family $\mathscr{J}$ should follow
from Prop. \eqref{prop:ChaksThm} upon applying  \cite[Theorem 3]{MaxRank}. To
show that the family of canonical curves is rigid, we can apply the same
residue calculations of \S~\ref{ssec:MixedTorelli} since the canonical
curve $x_0=0$ does not pass through the $A_1$-singularity at $(1:0:0:0)$.

(b) To prove also that rigidity and mixed Torelli hold in the
$\mathscr{J}$-case, we need to calculate the derivative of the period map
$\varphi:\mathscr{J}\to\Gamma\backslash D$ at a suitable surface
$S\in\mathscr{J}$. If the residue calculus can indeed be applied, we find, 
as expected,  that there are  $17$ deformation parameters.
Furthermore, the code of ~\ref{ssec:KernelCode} shows that for the
particular surface $S=V(f)$ defined by
\begin{equation}
  g = x^2z^4 + yz^4 + x^5z^3 +x^8z^2 + y^7 +x^{10}y^2 - w^2 + x^2y^3z^2
  \label{eq:J-surface}
\end{equation}
where $x=x_0$, $y=x_1$, $z=x_2$ and $w=x_3$, the kernel
of the derivative of the period map has dimension 1.  Indeed (see Table
\ref{tab:kernel}), the code shows that the
kernel of the map  $\ker((S/J)_d\stackrel{x}{\longrightarrow}(S/J)_{d+1})$
is given by 
\begin{equation}
  \eta = x^8y^3 + (4/5)x^6yz^2 + (1/2)x^3yz^3 + (1/5)y^4z^2 + (1/5)yz^4
  \label{eq:kernel-J-surface}
\end{equation}

\par The canonical curve $E_t$ of the surface $X_t=V(g+t\eta)$ is therefore
defined by the equation 
$$
      yz^4 + y^7 - w^2 + t(y^4z^2 + yz^4)/5
$$
A residue calculation shows that $y$ generates $H^{1,0}(E_0)$ whereas
$y^5z^2$ generates $H^{0,1}(E_0)$.  Moreover, $yz^4$ is contained in the
Jacobian ideal of $E_0$ whereas $y^4z^2$ does not reduce to zero modulo
the Jacobian ideal.  Therefore, the derivative of the period map of the
family $E_t$ at $t=0$, which corresponds to multiplication by $y^4z^2$,
is injective.  Thus, as in section 5, the derivative
of the mixed Torelli map is injective at the generic point of 
$\mathscr{J}$.
\end{remark}

\begin{appendix} 
\section{Manual  computations and  \textsc{SageMath} code}  \label{app:SAGE}
 
\subsection{Manual verifications}
\label{ssec:Manual}
 
The Hodge number $h^{2,0}=p_g$ for any of the  type 1--4  weighted
quasi-smooth surfaces  $F=0$ equals  $\dim R^1_F=1$.  The remaining Hodge number
$\dim H^{1,1}_{\rm prim}=\dim R^{d+1}_F$ (constant in families)
can be checked  by hand for  the \textbf{\emph{basic quasi-smooth surfaces}}  from
Proposition~\ref{prop:OurExamples}. This is also  the case for   the number of
moduli  
 $$
 \mu_F= \dim( H^1( T_X)_{\rm proj} ) =\dim(R^d_F) 
 $$
thereby   giving a check for  the results of Section~ \ref{ssec:NormForms}. 
Finally this can also be done for    the dimension of the kernel of the derivative of
the period map 
$
\delta_F= \dim \ker\left( R^d_F\mapright{\cdot x_0} R^{d+1}_F\right)$.
The results are given in the following table
\begin{center}
 \begin{tabular}{|l|c|c|c|}
\hline
Basic example& $h^{1,1}_{\rm prim}$ & $\mu$ & $\delta$ \\
\hline
 $ x_0^{14}+ x_1^7 + x_2^4x_1+   x_3^2$& $18$ & $18$  & $1$ \\
$x_0^{12}+x_1^6+x_2^4+x_1x_3^2$  &$17$ & $17$  & $1$\\
$x_0^{16}+x_1^8+x_0x_2^3+x_1x_3^2$  &$17$ & $16$  & $1$\\
$x_0^{22}+x_1^{11} + x_0 x_2^3+x_3^2$&$18$ & $18$  & $1$  \\
 \hline 
 \end{tabular} 
 \end{center}
Let us carry this out  for the  example (c) where  $J_{F }=(\frac 1{16}x_0^{15}+x_2^3,
\frac 18  x_1^7+ x_3^2,  \frac 13 x_0x_2^2, x_1x_3)$.  
 \begin{center}
\begin{tabular}{|c|c|c|c|}
\hline
$R_{F }^{16}$ & $R_{F  }^{17} $ & $R_{F  }^{16}$& $R_{F  }^{17} $ \\
\hline
 $(14,1,0,0)$  & $(15,1,0)\sim (0,1,3,0)$   &   $(9,0,0,1)$   & $(10,0,0,1)  $  \\
 $(12,2,0,0)$  &  $(13,2,0,0))$ & $(11,0,1,0)$ & $ (12,0,1,0)$ \\
 $(10,3,0,0)$  &  $(11,3,0,0)$ &  $(9,1,1,0) $  & $(10,1,1,0)$ \\
 $(8,4,0,0)$  &  $(9,4,0,0)$ &     $(7,2,1,0)$ & $(8,2,1,0)$\\
$(6,5,0,0)$   & $(7 ,5,0,0)$ & $(5,3,1,0)$  & $(6,3,1,0)$\\
 $(2,7,0,0)$   &  $(3,7,0,0)$     & $(3,4,1,0)$  & $(4,4,1,0)$ \\
 $(0,8,0,0)$ & $(1,8,0,0)$    & $(1,5,1,0)$  & $ ( 2, 5,1,0)$ \\
$(0,3,2,0)$   &  $--  $    &  $(4,0,1,1)$ &  $(5,0,1,1)$ \\
&                 &        &$(0,6,1,0) $   \\
&                 &        &$(0,0,2,1) $   \\
\hline
\end{tabular}
\end{center}
This  table shows that 
$ \delta_F= \ker( R^{16}_{F } \mapright{\cdot x_0} R^{17}_{F } )
= \bC \cdot  x_1^3x_2^2 $ and thus  infinitesimal Torelli does not hold for this
particular surface.
\par As noted earlier, for types  3  and 4, the basic hypersurfaces do not belong to
the moduli spaces $\cM_{5, 7}$, respectively $\cM_{7, 11}$. 
\par
Using the sage code listed below, we will show that $\delta_F=0$ at the generic
point $V(F)$ of the  moduli spaces in cases 1--4.

\subsection{Quasi-smoothness via Groebner basis}\label{ssec:groebner} 
 
To quickly verify quasi-smoothness of the hypersurface $V(g)$  for a given
polynomial $g$ using the \textsc{SageMath} code, we compute a Groebner basis
of the Jacobian ideal of $g$.  In each case, each case, the basis contains
some power of $x_3$.  Setting $x_3=0$ one then finds that $x_2=0$ as well.
This results in a system of equations which is easily solved by hand.
\begin{footnotesize}
\begin{sageblock}
# a=3; b=7;  #Type 1
# a=3; b=5;  #Type 1
# a=5; b=7;  #Type 3
# a=7; b=11; #Type 4
d = a+b+4
R=PolynomialRing(QQ,"x,y,z,w",order=TermOrder("wdeglex",(1,2,a,b)))
x,y,z,w=R.gens() #x=x_0, y=x_1, z=x_2, w=x_3
# Examples from Prop. 3.3.1
# g = x^14 + y^7 + z^4*y + w^2 + x^11*z + x^5*z^3 # Type 1
# g = w^2*y + 2*w*x*(z^2 + x^6) + z^4 + y^3*z^2 + x*y^4*z + y^6 # Type 2
# A = y^7 -w^2;  B = x^8*w + x^3*z*w + (x^(15) + x^(10)*z + x^5*z^2 + z^3 ); 
# g = A*y + B*x; # Type 3
# g = x*z^3 + y^4*z^2 + (x*y^7+x^9*y^3+x^(15))*z + y^11 - w^2 # Type 4
J=g.jacobian_ideal()
B =J.groebner_basis()
[print(b) for b in B]
gx=g.derivative(x); gy=g.derivative(y);
gz=g.derivative(z); gw=g.derivative(w);
print(gx.subs(z=0,w=0),"|",gy.subs(z=0,w=0),"|",gz.subs(z=0,w=0),"|",gw.subs(z=0,w=0))
\end{sageblock}
\end{footnotesize}
Note: If the equation is of the form $g=w^2+wf(x,y,z) + h(x,y,z)$, one should first
eliminate $wf(x,y,z)$ via a change of variable.

\subsection{Rank of the period maps for modular families}
\label{ssec:KernelCode}

\par In this subsection, we present code to compute the kernel of the differentials
of the period maps we consider. The basis of the code is that \textsc{SageMath} has
facilities to reduce polynomials relative to a given ideal and compute the
coefficient  matrix of a sequence of polynomials with respect to the monomials which
occur  in the sequence.  In this way, the problem reduces to a straightforward
linear algebra problem. This code has also been adapted to incorporate the
singularities of types $\mathscr{J}$, $E_{13}$ and $E_{14}$.

\par More precisely, let $R$ be a graded ring and $S$ and $J$ be homogeneous
ideals of $R$ such that $J\subset S$.  Then, the multiplicative structure of $R$
induces a well defined map
\begin{equation}
       (R/J)_{\alpha}\times (S/J)_d \mapsto (S/J)_{d+\alpha}        \label{eq:alg-kernel}
\end{equation}
where $(\text{---})_k$ denotes the degree $k$-component. As above, for types 1--4,
the determination of the kernel of the derivative of the period map at the generic surface
$V(g)$ amounts to the calculation of the kernel of \eqref{eq:alg-kernel} in the special
case where $R$ is the homogeneous coordinate ring of $P(1,2,a,b)$, $S=R$, $J$ is
the Jacobian ideal of $g$, $\alpha=1$ and $d=\deg(g)$.

\par Let $\cM$ denote a moduli space of surfaces of type $\mathscr{J}$, $E_{13}$, $E_{14}$
or $\cM_{a,b}$.  Let $V(g)$ be a generic element of $\cM$, with Jacobian ideal
$J\subseteq R$. To produce code which handles all of these moduli spaces at once, we
observe that in each case there exists a homogeneous ideal $I$ of $R$ such that every
element of the tangent space $\mathscr{T}$ to $\cM$ at $V(g)$ can be obtained by a
first order deformation $t\mapsto V(g+t\zeta)$ for some $\zeta\in I$. We therefore
set $S=I+J$, where $I=R$ for the moduli spaces $\cM_{a,b}$.
 
\par For the moduli spaces $\cM$ under consideration, $I_d$ always has a monomial
basis $B$.  Let
$
      X = \{ m\in B \mid m = m.reduce(J)\}
$
using the reduce command of \textsc{SageMath}.  Let $\tau(X)$ denote the subspace
of $\mathscr{T}$ defined by first order deformation through elements of
$\text{span}(X)$. Then, $\tau(X)=\mathscr{T}$ if and only if
\begin{itemize}
\item[(i)] $|X|=\dim\,\cM$
\item[(ii)] $\dim (\text{span}(X) + J) = |X| + \dim J$
\end{itemize}
This is easily checked by computer using the code found at the end of this
appendix (\ref{ssec:KernelCode}).  In the code, {\it ex\_dim} = $\dim\,\cM$
and the basis $B$ is the complement of the monomials listed in the parameter
{\it forbidden}.  The results of these calculations are
summarized in Tables~\ref{tab:kernel} and \ref{tab:kernel2}.  For types
1-- 4, the code also verifies the previously stated dimensions of
$(R/J)_d$ and $(R/J)_{d+1}$.
 \begin{footnotesize}
\begin{table}[h]
\begin{tabular}{lll}
1  & Defining polynomial          & Generator of Kernel \\
\hline
1.  & $x^{14} + y^7 + yz^4 + w^2 + x^2y^3z^2$  & $x^{12}y + (1/7)y^4z^2$ \\
2.  & $x^{12} + y^6 + z^4 + yw^2 + x^2y^2z^2$  & $x^{10}y + (1/6)y^3z^2$ \\
$\mathscr{J}$  &  $x^2z^4 + yz^4 + x^5z^3 + x^8z^2$ 
               & $x^8y^3 + (4/5)x^6yz^2 + (1/2)x^3yz^3$ \\
               &  $+y^7 + x^{10}y^2 - w^2+ x^2y^3z^2$
               &  $+ (1/5)y^4z^2 + (1/5)yz^4 $   \\
$E_{13}$ & $y^7 +y^2x^{10} + yz^4 + x^5z^3 + z^2x^8 - w^2$ & 
$x^8y^3 + (4/5)x^6yz^2 + (1/2)x^3yz^3$ \\
$E_{14}$ & $y^7 +y^2x^{10} + yz^4 + x^3yz^3 + z^2x^8 - w^2$ &
$x^8y^3 + (4/5)x^6yz^2 + (3/10)xy^2z^3$ 
\end{tabular}    
\caption{\small Examples, Infinitesimal mixed Torelli}
\label{tab:kernel} 
\end{table}

\begin{table}[h]
\begin{tabular}{lll}
1  & Defining polynomial \\
\hline
1. & $x^{14} + x^2z^4 + 2xy^2z^3 + y^7 + y^4z^2 + yz^4 + w^2$ & \\
2. & $x^{12} + y^6 + z^4 + yw^2 + x^2y^2z^2$ & \\
$\mathscr{J}$ & $x^{10}y^2+x^8z^2+2x^6yz^2-2x^5y^3z+x^2z^4$ & \\
              & $-2xy^2z^3+y^7+y^4z^2 + yz^4 + w^2$ & \\
$E_{13}$ & $x^{10}y^2 + x^8z^2 + 2x^5z^3 + 2x^4y^2z^2 + 2xy^2z^3$ & \\
        & $+y^7 + y^4z^2 + yz^4 + w^2$ &    \\
3.  & $x^{16} + y^8 + xz^3 + yw^2 + y^3z^2$ & \\
4.  & $w^2 + xz^3 + y^4z^2+x^{20}y + y^{11}$ & \\
\end{tabular}    
\caption{\small Examples, Infinitesimal Pure Torelli}
  \label{tab:kernel2} 
\end{table}
\end{footnotesize}
 \subsubsection{Quasi-smoothness calculations}
  To check that the specific hypersurfaces used in these calculations are
  quasi-smooth, we use the Groebner basis method of section
  \ref{ssec:groebner}. In each case, we find that some power of $z$
  and $w$ are contained in the Jacobian ideal of $g$ (this can be checked
  directly using the reduce command in \textsc{SageMath}), and hence we must
  have $z=w=0$ at the singular point. In the same way, we verify that the test
  surfaces of types $\mathscr{J}$, $E_{13}$, $E_{14}$ do not have any extra
  singularities.

  \par Since the only new feature arises in the case of types $\mathscr{J}$,
  $E_{13}$ and $E_{14}$, we only treat these cases.  For the examples for which
  the derivative of the period map has a non-trivial kernel, $z^{10}$ and $w$
  belong to the Jacobian ideal.  Moreover, the condition to have a singular
  point at $p$ reduces to $g_x=10x^9y^2=0$ and  $g_y=2x^{10}y+7y^6=0$.
  The only singular point is therefore at $p=[1:0:0:0]$, as expected.

  \subsection*{Code for infinitesimal period map calculations} The dimension
  of moduli space $\cM$ is {\it ex\_dim}.  The basis $B$ of $I_d$ is the
  complement of the monomials listed in {\it forbidden}.  For the moduli
  spaces of types $\mathscr{J}$, $E_{13}$ and $E_{14}$, the code assumes
  that $g=w^2 + h(x,y,z)$. The example for which mixed (infinitesimal) Torelli holds
  were found by perturbation of the basic examples.  The examples for which the
  pure (infinitesimal) Torelli theorem hold were found by generating a random
  element of the moduli space.
  
\begin{footnotesize}
\begin{sageblock}
#a=3; b=7;   ex_dim = 18; #Type 1. Different ex_dim for J, E13, E14 below. 
#a=3; b=5;  ex_dim = 17; #Type 2
#a=5; b=7;  ex_dim = 16; #Type 3
#a=7; b=11; ex_dim = 18  #Type 4
d = a+b+4    
R=PolynomialRing(QQ,"x,y,z,w",order=TermOrder("wdeglex",(1,2,a,b)))
x,y,z,w=R.gens()
# Code assumes g = w^2 + h(x,y,z) in cases J, E_13 and E_14
# Examples with non-zero kernels and infinitesimal mixed Torelli
# g = x^(14) + y^7 + y*z^4 + w^2 + x^2*y^3*z^2 # Type 1
# g = x^(12) + y^6 + z^4 + y*w^2 + x^2*y^2*z^2 # Type  2
# Type J:
# g = x^2*z^4 + y*z^4 + x^5*z^3 +x^8*z^2 + y^7 +x^(10)*y^2 - w^2 + x^2*y^3*z^2
# g = y^7 +y^2*x^(10) + y*z^4 + x^5*z^3 + z^2*x^8 - w^2 #E13
# g = y^7 +y^2*x^(10) + y*z^4 + x^3*y*z^3 + z^2*x^8 - w^2 #E14
# Examples with trivial kernels (i.e. infinitesimal Torelli holds)
# g = x^(14) + x^2*z^4 + 2*x*y^2*z^3 + y^7 + y^4*z^2 + y*z^4 + w^2 # Type 1
# g = x^(12) + x*z^2*w + y^6 + y^2*z*w + y*w^2 + z^4 # Type 2
# g = x^(16) + y^8 + x*z^3 + y*w^2 + y^3*z^2 # Type 3
# g = w^2 + x*z^3 + y^4*z^2+x^(20)*y + y^(11) # Type 4
# Type J:
# A = x^(10)*y^2 + x^8*z^2 + 2*x^6*y*z^2 - 2*x^5*y^3*z + x^2*z^4;
# B = - 2*x*y^2*z^3 + y^7 + y^4*z^2 + y*z^4 + w^2; g = A+B;
# Type E13
# A = x^10*y^2 + x^8*z^2 + 2*x^5*z^3 + 2*x^4*y^2*z^2 + 2*x*y^2*z^3;
# B = y^7 + y^4*z^2 + y*z^4 + w^2; g = A+B;     
# The parameters forbidden, ex_dim:
# forbidden=[]; #For types 1-4
# forbidden=[x^14,x^11*z,x^12*y]; ex_dim=17 #J surface
# forbidden=[x^14,x^11*z,x^12*y,x^2*z^4]; ex_dim=16; #E13 surface
# forbidden=[x^14,x^11*z,x^12*y,x^2*z^4,x^5*z^3]; ex_dim = 15; #E14 surface
Md=[R.monomial(*e) for e in WeightedIntegerVectors(d,(1,2,a,b))]
[Md.remove(m) for m in forbidden]   
J=g.jacobian_ideal()
gx=g.derivative(x); gy=g.derivative(y);
gz=g.derivative(z); gw=g.derivative(w);
Z=Sequence([x*gx,x^2*gy,y*gy]);
Ma=[R.monomial(*e) for e in WeightedIntegerVectors(a,(1,2,a,b))]
[Z.append(m*gz) for m in Ma]
Mb=[R.monomial(*e) for e in WeightedIntegerVectors(b,(1,2,a,b))]
[Z.append(m*gw) for m in Mb]
# Z = Degree d component of J.
W,n=Z.coefficient_matrix(); jd = rank(W);
print("Dimension of J_d = ", jd)
X=Sequence([m for m in Md if m.reduce(J)==m]); L=Set(X)
rx = L.cardinality(); print("Cardinality of X = ",rx);
[Z.append(m) for m in Md if m.reduce(J)==m];
U,n=Z.coefficient_matrix(); ru=U.rank()
print("Dimension of (J_d + span(X)) = ",ru)
# X gives a basis for the tangent space to the deformation space if 
# ru = rx + jd and rx = ex_dim
if ((ru==rx+jd) and (rx==ex_dim)): #This code must be indented.
     print("X is a basis of the tangent space, calculating the kernel dimension.")
     T=Sequence([x^2*gx, y*gx])
     M3 = [R.monomial(*e) for e in WeightedIntegerVectors(3,(1,2,a,b))]
     [T.append(m*gy) for m in M3]
     Ma = [R.monomial(*e) for e in WeightedIntegerVectors(a+1,(1,2,a,b))]
     [T.append(m*gz) for m in Ma]
     Mb = [R.monomial(*e) for e in WeightedIntegerVectors(b+1,(1,2,a,b))]
     [T.append(m*gw) for m in Mb]
     # Degree d+1 component of J.
     D, l = T.coefficient_matrix()
     r2 = D.rank(); print("Dimension of J_{d+1} = ",r2)
     [T.append((x*m)) for m in X]
     D, l = T.coefficient_matrix()
     r3 = D.rank(); print("Dimension of J_{d+1} + Im(span(X)) = ",r3)
     print("Kernel dimension = ",rx+r2-r3)
     if(rx+r2-r3>0):
     # Find the kernel.
         print("Calculating kernel.");
         c = D.ncols(); P = D.submatrix(0,0,r2,c); P1=P.row_space();
         Q = D.submatrix(r2,0,rx,c); Q1=Q.row_space();
         B = P1.intersection(Q1);
         B1 = B.basis_matrix();
         for j in range(B1.nrows()):
             s=[]
             [s.append(B1[j,k]*l[k]) for k in range(c)]
             m=(sum(s))[0]
             t, r = m.quo_rem(x)
             print("t=",t,"| x*t mod J=",(x*t).reduce(J),"| t mod J=|",t.reduce(J))
     if((Set(forbidden)).is_empty()):
         Md= [R.monomial(*e) for e in WeightedIntegerVectors(d,(1,2,a,b))]
         K = Set(Md)
         Md1=[R.monomial(*e) for e in WeightedIntegerVectors(d+1,(1,2,a,b))]
         L = Set(Md1)
         print("dim (R/J)_d = ",K.cardinality()-jd)
         print("dim (R/J)_{d+1} = ",L.cardinality()-r2)
else:
     print("X is not a basis of tangent space, exiting.")
\end{sageblock}
\end{footnotesize}

\subsection{Calculations involving  the $E_{12}$, $E_{13}$ and
$E_{14}$-singularities}\label{sec:OnEn}

\subsubsection{Reduction to  type $(14,[1,2,3,7])$} Recall that the singularity types determine
  data $(p,q,d)$ from the exponents of the occurring monomials $x^ay^bz^b$ via the weight rule 
  \eqref{eqn:weightrule}. For the
   the singularity types
$E_{12}$, $E_{13}$ and $E_{14}$ in~\cite{GPSZ}  these data are 
$$
E_{12}:\hphantom{a}(3,7,21),\qquad E_{13}:\hphantom{a}(2,5,15),\qquad
E_{14}:\hphantom{a}(3,8,24)
$$
As noted in section 3 of~\cite{GPSZ}, the following 19 monomials
$x^ay^bz^c\leftrightarrow(a,b,c)$ have non-negative weight for
$E_{12}$, $E_{13}$ and $E_{14}$:
\begin{gather*}
  (0,0,5),~(0,2,4),~(0,4,3),~(0,6,2),~(0,8,1),~(0,10,0), \\
  (1,3,3),~(1,1,4),~(1,5,2),~(1,7,1),~(1,9,0),~(2,0,4),  \\
  (2,2,3),~(2,4,2),~(2,6,1),~(2,8,0),~(3,1,3),~(3,3,2),~(4,0,3),
\end{gather*}
The monomial $x^3y^5z\leftrightarrow(3,5,1)$ occurs in non-negative weight
for both $E_{12}$ and $E_{13}$.  Finally, the monomial
$x^3y^7\leftrightarrow(3,7,0)$ occurs in non-negative weight only for $E_{12}$.
After multiplying each of the monomials in the previous list  by $z^2$  and converting
to the variables $x_0=y$, $x_1=z$ and $x_2=xz$, the previous list becomes
$x^ay^bz^c \mapsto x_0^bx_1^{c+2-a}x_2^a\leftrightarrow(b,c+2-a,a)$:
\begin{gather*}
  (0,7,0),~(2,6,0),~(4,5,0),~(6,4,0),~(8,3,0),~(10,2,0), \\
  (3,4,1),~(1,5,1),~(5,3,1),~(7,2,1),~(9,1,1),~(0,4,2),  \\
  (2,3,2),~(4,2,2),~(6,1,2),~(8,0,2),~(1,2,3),~(3,1,3),(0,1,4)
\end{gather*}
The remaining $E_{13}$ monomial  $x^3y^5z$ maps to $x^3y^5z^3=x_0^5x_2^3$.  The
monomial $x^3y^7$, which occurs only in $E_{12}$, does not
transform into a degree 14 monomial in $x_0,x_1,x_2$ by this process.  

\par Direct enumeration shows that there are 24 monomials of degree 14 in
$\bP[1,2,3]$.  Thus, there are 4 monomials missing from the list for
$E_{13}$: Since this highest power of $y$ which can appear in degree 10 in
$\bP[1,1,2]$ is 10, it follows that we miss the monomials $x_0^{14}$,
$x_0^{12}x_1$ and $x_0^{11}x_2$.  We also miss the monomial $x_0^2x_2^4=x^4y^2z^4$
since this comes by multiplying $x^4y^2z^2$ by $z^2$, and  $x^4y^2z^2$
has weight $\omega = (2)(2)+(2){(5)}-15=-1$.  

\par In particular, since we don't have the monomial $x_0^{14}$, a curve
$B = V(g)$ arising from the $E_{13}$ or $E_{14}$ singularity will always
pass through the point $(x_0:x_1:x_2)=(1:0:0)$.  Moreover, since we
don't have the monomials $x_0^{12}x_1$ and $x_0^{11}x_2$ it follows that
$\nabla g=0$ at $(1:0:0)$.  

\subsubsection{The   $\mathscr{J}$-locus} 

\begin{proposition} The singular locus of the degree 14 surface 
$V(f)\subset\bP[1,2,3,7]$ defined by the equation
\begin{equation}
  f = x_0^2x_2^4 + x_1x_2^4 + x_0^5x_2^3 + x_0^8x_2^2 + x_1^7
      + x_0^{10}x_1^2 -x_3^2  \label{eq:J-surface-old}
\end{equation}
consists of an $A_1$-singularity at the point $[1:0:0:0]$ and the finite
quotient singularity at the point $[0:0:1:0]$ which $V(f)$ inherits from
$\bP[1,2,3,7]$.  Moreover, since the defining equation of this surface
contains the term $x_0^2x_2^4$, it is not contained in the $E_{13}$ or $E_{14}$
locus.  The associated smooth elliptic surface has fiber structure
$2I_0 + I_2 + 22\times I_1$.
\end{proposition}
\begin{proof} Dividing by $x_0^{14}$ and setting $\zeta = x_1/x_0^2$,
$\nu=x_2/x_0^3$ and $\omega = x_3/x_0^7$ gives
$$
        \omega^2 = (1+\zeta)\nu^4 + \nu^3 + \nu^2 + \zeta^7 + \zeta^2
$$
 and thus one has an   $I_2$ fiber over $(1:0)$.\footnote{The only 
monomials $x_0^ax_1^bx_2^c$ which occur in $\mathscr{J}$ and survive 
evaluation at $(x_0^2:x_1)=(1:0)$ are $x_2^4x_0^2$, $x_2^2x_0^8$ and 
$x_2^3x_0^5$, so the analysis presented here is the generic case.}
        
\par Taking the discriminant of the right hand side of the previous equation with
respect to $\nu$ gives
$$
\aligned
(\zeta^6 &- \zeta^5 + \zeta^4 - \zeta^3 + \zeta^2)
(256\zeta^{18} + 1024\zeta^{17} + 1536\zeta^{16} + 1024\zeta^{15}
+ 256\zeta^{14} \\ &+ 512\zeta^{13} 
+ 2048\zeta^{12} + 3072\zeta^{11}
+ 1920\zeta^{10}+ 272\zeta^9 + 133\zeta^8 + 1013\zeta^7 \\ &+ 1536\zeta^6 
+ 896\zeta^5 + 16\zeta^4 - 123\zeta^3 + 5\zeta^2 + 28\zeta + 12)
\endaligned
$$
Factoring out $\zeta^2$ and taking the resultant of the remaining two
polynomials of degree 4 and 18 gives 999680.  Thus, the discriminant
has 22 simple roots and one double root at $\zeta=0$.  The existence of
$2I_0$ fiber is the same as the generic surface of type $(14,[1,2,3,7])$.
The fiber structure is therefore $2I_0 + I_2 + 22\times  I_1 $ as claimed.

\par The analysis of the singular locus for this surface is exactly the same
as the surface presented at the end of Appendix A.   The only difference 
between the surface considered here and the surface presented there is the 
term $x^2y^3z^2$.  The Jacobian ideal contains the monomials $z^{10}$ and 
$w$, and the condition to have singular point at $p$ is $g_x=10x^9y^2$ and
$g_y= 2x^{10}y + 7y^6$.  Therefore $(1:0:0:0)$ is the only singular point of
the surface.
\end{proof}

\par Transferring equation \eqref{eq:J-surface-old} back to $\bP(1,1,2,5)$
by dividing by $z^2$ after setting $x_0=y$, $x_1=z$, $x_2=xz$ and $x_3=zw$
yields the defining equation
$$
      w^2 = f(x,y,z),\qquad 
      f(x,y,z)=x^4y^2z^2 + x^4z^3 + x^3y^5z + z^5 + x^2y^8 + y^{10}
$$
In this case the branch curve $V(f)$ has a singularity of type $J[2,2]$ at the 
 point $(1:0:0)$.  This can be confirmed by the following sage code, which
also shows that this singularity has modality 1 and Milnor number $\mu = 12$. 
(set $x=1$ to work in an affine chart)

\begin{footnotesize}
\begin{sageblock}
r = singular.ring(0,'(y,z)', 'ds')
singular.lib('classify.lib')
h = singular.new('y^2*z^2 + z^3 + y^5z + z^5 + y^8 + y^(10)')
print(singular.eval('classify({})'.format(h.name())))
\end{sageblock}
\end{footnotesize}

\begin{remark}As shown above, $\mathscr{J}$ has dimension 17, which matches
the dimension formula $29-\mu$ of the other modality 1 singularities
considered above.
\end{remark}

\subsubsection{The case $E_{13}$} A typical $E_{13}$ branch curve is defined
by the equation
$$
     f= z^5 + y^{10} + x^4z^3 + x^3y^5z + x^2y^8
$$
which becomes
\begin{equation}
  g = x_1^7 + x_1^2x_0^{10} + x_1x_2^4 + x_2^3x_0^5 + x_2^2x_0^8
  \label{eq:E13-typical}
\end{equation}
To analyze the singularity at the point $(x_0:x_1:x_2)=(1:0:0)$, we
divide by $x_0^{14}$ and introduce the new variables $\zeta = x_1/x_0^2$
and $\nu=x_2/x_0^3$ to obtain
\begin{equation}
  \zeta^7 + \zeta^2 + \nu^4\zeta + \nu^3 + \nu^2   \label{eq:E13-branch}
\end{equation}
The lowest order term here is $\zeta^2 + \nu^2$, which produces an $A_1$
surface singularity at $(1:0:0:0)$.  As shown at the end of Appendix $A$,
this surface has no other singularities.

 \par To continue the analysis of the birational models of the $E_{13}$ surfaces
as degree 14 hypersurfaces in $\bP[1,2,3,7]$, we consider the fibration
to $\bP^1$ given by $(x_0:x_1:x_2:x_3)\mapsto (x_0^2:x_1)$.  By equation
\eqref{eq:E13-branch} the fiber over the point $(x_0^2:x_1)=(1:\zeta)$ is
given by 
$ 
    \nu^4\zeta + \nu^3 + \nu^2  + \zeta^7 + \zeta^2 = \omega^2 
$ 
where $\zeta=x_1/x_0^2$, $\nu=x_2/x_0^3$ and $\omega=x_3/x_0^7$ in the affine
chart $x_0\neq 0$ of $\bP[1,2,3,7]$.  The fiber over $\zeta=0$, is the
nodal cubic\footnote{The only monomials $x_0^ax_1^bx_2^c$ which occur in $E_{13}$
and survive evaluation at $(x_0^2:x_1)=(1:0)$ are $x_2^2x_0^8$ and $x_2^3x_0^5$, so
the analysis presented here is the generic case.}
$ 
    \omega^2 = \nu^2(\nu+1)
$ 
whose singularity at the point $(0,0)$ in the $(\nu,\omega)$-plane coincides
with
the $A_1$-singularity at the point $(1:0:0:0)$ on the surface.  Resolving this
singularity, we obtain an $I_2$-fiber.  Just like the generic degree 14 surface in
$\bP[1,2,3,7]$, the fiber over $(0:1)$ is of type $2I_0$.  To finish the
analysis of the fibers of $\pi$, we calculate the discriminant $D$ of the
polynomial \eqref{eq:E13-branch} with respect to $\nu$:
$$
D = \zeta^2 (\zeta^5 + 1) (256 \zeta^{17} + 512 \zeta^{12} - 128 \zeta^9 + 144 \zeta^8 + 229 \zeta^7 -
128 \zeta^4 + 144 \zeta^3 - 27 \zeta^2 + 16 \zeta - 4)
$$
The discriminant of the degree 17 factor of $D$ with respect to the
variable $\zeta$ is an 81 digit integer.  Thus, $\pi$ also has 22
 $I_1$ fibers.   
 
\subsubsection{The case  $E_{14}$}
The analysis of this case  is similar.  A typical $E_{14}$ branch curve is 
$ 
       f = z^5 + y^{10} + x^4z^3 + x^2y^8 
$  which becomes
\begin{equation}
      g = x_1^7 + x_1^2x_0^{10} +x_1x_2^4 + x_2^2x_0^8  \label{eq:E14-branch} 
\end{equation}
Setting $x_0=1$ and letting $\zeta = x_1/x_0^2$ and $\nu=x_2/x_0^3$ as above,
this becomes
\begin{equation}
      \zeta^7 + \zeta^2+ \zeta\nu^4 + \nu^2       \label{eq:E14-affine}
\end{equation}
so again we have an $A_1$ surface singularity at $(1:0:0:0)$.
 
 \par To determine the elliptic fibration structure, we compute the
discriminant $D$ of
\begin{equation}
  \zeta^7 + \zeta^2+ \zeta\nu^4 + \zeta\nu^3 + \nu^2
  \label{eq:E14-affine-2}
\end{equation}
which yields
$$
\zeta^3(\zeta^5 + 1) (256 \zeta^{16} + 512 \zeta^{11} - 27 \zeta^{10} + 144 \zeta^9 - 128 \zeta^8
+ 256 \zeta^6 - 27 \zeta^5 + 144 \zeta^4 - 128 \zeta^3 - 4 \zeta + 16)
$$
The discriminant of the degree 16 factor of $D$ is a 77 digit integer, and
hence this factor has no multiple roots.   The resultant of the degree 5
and 16 factors is 1049600.  We also retain the $2I_0$ fiber over
$(x_0^2:x_1)=(0:1)$.  Thus, the fibration structure of this surface is 
$2I_0 + I_3 + 21\times I_1$.   
Observe that  for the generic  $E_{14}$-surface
 the affine form of the fiber at $\zeta=0$ is $\nu^2-\omega^2=0$ 
 and so gives $2$ smooth rational curves;  the $A_1$-singularity contributes another 
rational component, confirming the $I_3$-structure at $\zeta=0$ for the generic  $E_{14}$-surface.
Thus, also  the generic $E_{14}$-surface has fiber type $2I_0 + I_3 + 21\times I_1$.

\begin{remark}Let $S\subset\bP[1,1,2,5]$ be the surface of type
$E_{12}$ defined by the equation 
$ 
             x^4z^3 + x^3y^7 -xy^9 + y^{10} + z^5-w^2=0
$ 
and $\pi:\tilde S\to\bP^1$ be the elliptic surface obtained by
resolving the indeterminacies of the map $(x:y:z:w)\mapsto(y^2:z)$.  
Then, $\pi^{-1}(1:\lambda)$ is
$ 
     \lambda^3X^4 + X^3-X + (1 + \lambda^5) = W^2
$, 
where $X=x/y$ and $W=w/y^5$.  The discriminant of the left hand side of this
equation is a degree 24 polynomial without multiple roots, and the fiber over
$(1:0)$ is an irreducible elliptic curve.
\end{remark}

\subsection{Code for ~\ref{app:nijgh}} \label{a} 
 
The polynomial $F$ of Equation~\ref{equation} is defined as follows.
 
\begin{footnotesize}
 \begin{verbatim}
WP.<x0,x1,x2,x3> = PolynomialRing(ZZ)
G  = x0*x2^2 + x0^4*x2 + 3*x1^2*x2
G0 = -1
G3 = x0^6 + 2*x0^4*x1 + x0^2*x1^2 + 2*x1^3 
G4 = 4*x0^6*x1+2*x0^4*x1^2 + x0^2*x1^3 + 4*x1^4
G6 = x0^12 + 3*x0^10*x1 + 3*x0^8*x1^2 + x0^4*x1^4 + 3*x0^2*x1^5 + x1^6
F = x1*x3^2 + G*x3 + G0*x2^4 + G3*x2^2 + G4*x0*x2 + G6
\end{verbatim}
 \end{footnotesize}
 For practical reasons, we introduced here instead $G=G_{1,2}\cdot x_1^2 + G_2\cdot x_0$.
 
\subsubsection{Checking quasi-smoothness} \label{a.qs}

Next, we choose $p\in\{2,3\}$ and check if the surface $X_p$ over $\bbF_p$ is quasi-smooth. 
This is done by checking if the radical of the ideal generated by $F$ and its derivatives is equal to the irrelevant ideal.
The outcome of this check is `true', which shows that the surface $X_p$ is quasi-smooth.

\begin{footnotesize}
\begin{sageblock}
p = 2    # (or p = 3)
Wp.<xp0,xp1,xp2,xp3> = PolynomialRing(GF(p))
f = F(xp0,xp1,xp2,xp3)
f0 = f.derivative(xp0)
f1 = f.derivative(xp1)
f2 = f.derivative(xp2)
f3 = f.derivative(xp3)
I = Ideal([f,f0,f1,f2,f3])
RI = I.radical()
RI == Ideal([xp0,xp1,xp2,xp3])
\end{sageblock}
\end{footnotesize}

\subsubsection{Checks for the arithmetic surface} \label{a.surf}
Here we give the code that checks if the arithmetic surface~$\mathcal{C}$ is smooth. 
This is done on the two affine parts of the surface. 
In both cases we check that the defining equation together with its derivatives generate the whole ring.
The outcome of these checks are `true', which shows that the surface is smooth.

\begin{footnotesize}
\begin{sageblock}
CF = x1*F(x0,x1,x2,x3/x1)
R.<t,x,y> = PolynomialRing(GF(p))
CF1 = CF(1,t,x,y)
Ft1 = CF1.derivative(t)
Fx1 = CF1.derivative(x)
Fy1 = CF1.derivative(y)
I = Ideal([CF1,Ft1,Fx1,Fy1])
I == (1)
\end{sageblock}
\end{footnotesize}

\begin{footnotesize}
\begin{sageblock}
CF2 = x^4*CF(1,t,1/x,y/x^2)
Ft2 = CF2.derivative(t)
Fx2 = CF2.derivative(x)
Fy2 = CF2.derivative(y)
I = Ideal([CF2,Ft2,Fx2,Fy2])
I == (1)
\end{sageblock}
\end{footnotesize}

\subsubsection{Calculating the discriminant} \label{a.disc}

In the next part of the code, we calculate the discriminant of the defining polynomial $F'$ of the arithmetic surface $\mathcal{C}$. 
To calculate this discriminant, we first calculate it over $\bZ[t]$, which is named \verb+pDisc+ in the code.
The outcome of this part of code gives the factorization of this polynomial modulo $p$, which is the polynomial~$\Delta_p$ given in Lemma~\ref{lemdisc}. 
    
\begin{footnotesize}
\begin{sageblock}
Fd.<tt> = FunctionField(QQ)
Rd.<xd> = PolynomialRing(Fd)
fd = F(1,Fd.0,Rd.0,0)*Rd.0
hd = G(1,Fd.0,Rd.0,0)
pDisc = 4^(-4)*discriminant(hd^2-4*fd)
R.<t> = PolynomialRing(GF(p))
Disc = R(pDisc.numerator())
Disc.factor()
\end{sageblock}
\end{footnotesize}

\subsubsection{Counting the points}\label{a.count}

The following part of the code counts the points on the surface $X_p'$ following the method described in the proof of Proposition~\ref{pts}. For $1\leq n\leq 9$, we count the~$\bbF_{p^n}$-points of~$X_p'$ at once and save the number we find in the list with the name \verb+Count+. 
 
\begin{footnotesize}
\begin{sageblock}
Count = []
for i in range(1, 10):
     q = p^i
    Fq = GF(q)
    A2 = AffineSpace(2, Fq)
    R.<t> = PolynomialRing(Fq)
   # count points above (0:1)
    f = -F(0,1,R.0,0)
    h = G(0,1,R.0,0)
    C = HyperellipticCurve(f,h)
    count = C.cardinality()
     # count points above (1:a)
    for a in Fq: 
        if Disc(a) == 0:
            g = t^2+t-a
            r = Set(g.roots()).cardinality()
            f = CF(1,a,A2.0,A2.1)
            C = Curve(f,A2)
            count = count + C.count_points(1)[0] + r
        else:
            f = -F(1,a,R.0,0)*a
            h = G(1,a,R.0,0)
            C = HyperellipticCurve(f,h)
            count = count + C.cardinality()
   Count.append(count)
# print total number of Fq-points for q=p^i with 0<i<10:
Count
\end{sageblock}
\end{footnotesize}

\begin{remark}
    The code for counting the points takes a lot of time when $p=3$ (roughly 18 hours on Mac OS with Apple M1 processor and 8GB RAM). For finding the surfaces, we used other software, namely \textsc{Magma} (see \cite{magma}). This software is faster and made it easier to search through many surfaces over~$\bbF_2$ and~$\bbF_3$ respectively, to find those that satisfied the required conditions.
\end{remark}

\subsubsection{Computing the characteristic polynomial}\label{a.pol}

The last part of code is used for the proof of Proposition~\ref{frob}. First, we compute the coefficients $c_i$.

\begin{footnotesize}
\begin{sageblock}
# Calculate the values of the trace
Tr = []
for i in range(1, 10):
    Tr.append((Count[i-i] - 1 - p^(2*i)-3*p^i)/(p^i))
# Calculate the values of the coefficients
coef = [1,-Tr[0]]
for i in range(1, 9):
    sum = 0
    for j in range(1,i+1):
        sum = sum + Tr[i-j]*coef[j]
    coef.append(-(Tr[i]+sum)/(i+1))
# Print the coefficients; N.b. first value is c0, not c1:
coef
\end{sageblock}
\end{footnotesize}

In the next part of the code, we define both possible polynomials using the functional equation. 

\begin{footnotesize}
\begin{sageblock}
R.<t>=PolynomialRing(QQ)
# Defining polynomial with positive sign of func eq
coefp = [0] * 20
for i in range (0,10):
    coefp[i] = coef[i]
    coefp[19-i] = coef[i]
wp = R(coefp)
# Defining polynomial with negative sign of func eq
coefn = [0] * 20
for i in range (0,10):
    coefn[i] = -coef[i]
    coefn[19-i] = coef[i]
wn = R(coefn)
\end{sageblock}
\end{footnotesize}

To exclude the polynomial where the sign is positive, we can print the absolute values of the roots by using the following line of code.
\begin{footnotesize}
\begin{sageblock}
for root, _ in wp.roots(CC): print(abs(root))
\end{sageblock}
\end{footnotesize}

The outcome will give a list of absolute values of the roots, of which four are not equal to 1.
As a sanity check, by using the same line of code with \verb+wn+ instead of \verb+wp+, we can see that all the roots indeed have absolute value 1.

Lastly, we factored the polynomial by using the function \verb+wn.factor()+. This gives us a factor $t-1$ and the other factor is the irreducible polynomials $h_p$ of degree 18, which is given in Proposition~\ref{frob}.

\subsubsection{A similar case for degree 14} \label{a.casea}

Here we give the code from which one can deduce that for a general choice of a quasi-smooth surface~$X$ of degree~14 in~$\bP_k(1,2,3,7)$, a minimal desingularization~$X'$ has Picard rank~$2$.
In the below code quasi-smoothness is omitted, but it can be checked with the exact same code as in~\ref{a.qs}. Also the verification that the model we use is correct, is omitted.  
We highlight the differences in the code with the degree 12 case.

\begin{footnotesize}
\begin{sageblock}
# Defining polynomial
WP.<x0,x1,x2,x3> = PolynomialRing(ZZ)
G = x0*x2^2 + x0^4*x2 + x1^2*x2
G2 = x0^2*x1
G5 = x0^8*x1 + x0^2*x1^4 + x1^5
G7 = x0^14 + x0^12*x1 + x0^10*x1^2 + x0^6*x1^4 + x0^2*x1^6 + x1^7
F = x1*x2^4 + x3^2 + G*x3 + G2*x0*x2^3 + G5*x0*x2 + G7
# Calculating the discriminant
p = 2
Fd.<tt> = FunctionField(QQ)
Rd.<xd> = PolynomialRing(Fd)
fd = F(1,Fd.0,Rd.0,0)
hd = G(1,Fd.0,Rd.0,0)
pDisc = 4^(-4)*discriminant(hd^2-4*fd)
R.<t> = PolynomialRing(GF(p))
Disc = R(pDisc.numerator())
# Counting the points
Count = []
for i in range(1, 11): 
# Note that we now count one more extension, 
# also the code below is adjusted accordingly.
    q = p^i
    Fq = GF(q)
    A2 = AffineSpace(2, Fq)
    R.<t> = PolynomialRing(Fq)
    f = -F(0,1,R.0,0)
    h = G(0,1,R.0,0)
    C = HyperellipticCurve(f,h)
    count = C.cardinality()
    for a in Fq: 
        if Disc(a) == 0:
            g = t^2+t+a #sign change, although not necessary
            r = Set(g.roots()).cardinality()
            f = F(1,a,A2.0,A2.1) #changed defining polynomial
            C = Curve(f,A2)
            count = count + C.count_points(1)[0] + r
        else:
            f = -F(1,a,R.0,0) #changed defining polynomial
            h = G(1,a,R.0,0)
            C = HyperellipticCurve(f,h)
            count = count + C.cardinality()
    Count.append(count)
# Calculating the traces
Tr = []
for i in range(1, 11):
    Tr.append((Count[i-1] - 1 - p^(2*i)-2*p^i)/(p^i)) 
# Note the slight change in the formula for the trace,
# because we now know that there is a 2-dim subspace
# on which Frobenius is acting trivial and not 3-dim.
 # Calculating the coefficients
coef = [1,-Tr[0]]
for i in range(1, 10):
    sum = 0
    for j in range(1,i+1):
        sum = sum + Tr[i-j]*coef[j]
    coef.append(-(Tr[i]+sum)/(i+1))
coef[10]
# Because the coefficient c10 is non-zero, 
# the functional equation gives us that the 
# other coefficients are positive as well.
# Calculating the characteristic polynomial of Frobenius
R.<t>=PolynomialRing(QQ)
for i in range (0,10):
    coef.append(coef[9-i])
wp = R(coef)
wp.factor()
\end{sageblock}
\end{footnotesize}

The outcome of the code gives us   an irreducible polynomial $h$ with
$$h:= \tfrac{1}{2}(2t^{20} - 2t^{18} + t^{16} - t^{14} + t^{13} + t^{12} - t^{11} - t^{10} - t^9 + t^8 + t^7 - t^6 + t^4 - 2t^2 + 2),$$
which has no roots of unity as zeros. 
We deduce that the characteristic polynomial of Frobenius acting on $H^2_\text{\'et}((X_2')_{\overline{\bbF}_2},\bQ_\ell(1))$ equals $(t-1)^2\cdot h$. 
We conclude that for any minimal desingularization of a quasi-smooth surface $X$ of degree 14 in $\bP_\bQ(1, 2, 3, 7)$, 
for which the reduction at the prime $2$ is isomorphic to $X_2$, we have  $\rho(X')=\rho(X'_{\overline{\bQ}})=2.$

\section{Normal forms: proofs}

\label{sec:NormFormsBis}

We give indications of the proof of Proposition~\ref{prop:GIT} concerning normal forms
of quasi-smooth  hypersurfaces $(F=0)$  in $\bP(1,2,a,b)$ of  degree $d=a+b+4$.
 Note that in case $(a,b)=(3,7)$ and $(a,b)=(7,11)$ one has $d=2c$ which means that
 the surface is a double cover of $\bP(1,2,a)$ branched  in a degree $d$ quasi-smooth curve $C$.
 It then suffices to  write a normal form  for   the polynomial  $F_C$ defining  $C$ and then $F=F_C-x_3^2$.  
 This deals with 2 cases:

\begin{lemma} \label{lem:37AND27} (1) If $(a,b)=(3,7)$ then, via the automorphism group of
$\bP(1,2,3)$, the polynomial  $F_C$ can be put in the form
\begin{equation}
     F_C=  x_1   x_2^4 + G_0  x_0^5  x_2^3 + G_4(x_0^2,x_1) x_2^2  
      + x_0 G_5(x_0^2, x_1)x_2  + G_7( x_0^2, x_1)
      \label{eq:nf-123}
\end{equation}
where $G_j$ is an \emph{ordinary} polynomial of degree $j$ in two variables.  The
subgroup of $\text{\rm Aut}\,\bP(1,2,3)$
 preserving  a normal form of the type  \eqref{eq:nf-123} consists of transformations
of the form $x_j\mapsto c_jx_j$ with $c_j\in\bC^*$ and 
$c_2^4 c_1=1$. 
\\
(2) 
If  $(a,b)=(7,11)$  then,  provided the coefficient of $x_1^4x_2^2$ is non-zero, via the automorphism group   of $\bP[1,2,7]$, the polynomial  $F_C$ can be put in the form
\begin{equation}
      F_C=x_0  x_2^3  + G_0 x_1^4  x_2^2 + x_0 G _7(x_0^2,x_1) x_2 + G_{11}(x_0^2,x_1),\quad G_0\not=0,
      \label{eq:nf-127}
\end{equation}
where $G_j$ is an \emph{ordinary} polynomial of degree j in two variables,
and the coefficient of $x_0^{22}$ in  $G_{11}(x_0^2,x_1)$  is zero.
The subgroup of $\text{\rm Aut}\,\bP(1,2,7)$
 preserving  a normal form of the type  \eqref{eq:nf-127} consists of transformations
of the form $x_j\mapsto c_jx_j$ with $c_j\in\bC^*$ and $c_0c_3^2=1$.

In both cases  the stabilizer of $F_C$ is generically the identity.
\end{lemma}

\begin{proof} (1).
Since $3$ is not a divisor of 14, every degree 14 curve in
$\bP(1,2,3)$ will pass through the singular point $[0,0,1]$ of
$\bP(1,2,3)$.  Thus, to be quasi-smooth the coefficient of $x_2^4x_1$
in $F_C$ has  to be non-zero, otherwise $\nabla F_C=0$ at $[0,0,1]$.  Accordingly, we
can write $
F_C  = x_2^4 P_2 + x_2^3 P_5 + x_2^2 P_8 + x_2 P_{11}  + P_{14}
$, 
where $P_j=P_j(x_0,x_1)$ is homogeneous of weighted  degree $j$ and 
$P_2(x_0, x_1) = \alpha_1 x_1 + \alpha_2 x_0^2$ with $\alpha_1\neq 0$.

\par The automorphism group of $\bP(1,2,3)$ consists of invertible
transformations of the form
\begin{equation}
      [x_0,x_1,x_2] \mapsto
      [a_0 x_0,a_1 x_1 + a_2 x_0^2, a_3 x_2 + a_4 x_0 x_1 + a_5 x_0^3]
\label{eq:auto-123}      
\end{equation}      
In particular, via the transformation
$[x_0,x_1,x_2]\mapsto [x_0,P_2(x_0,x_1),x_2]$
we can reduce the defining equation of $C$ to the form:
\begin{equation}        
  x_2^4x_1 + x_2^3 P_5 + x_2^2 P_8 + x_2 P_{11}  + P_{14}
  \label{eq:123-1}
\end{equation}
Next, we observe that
$ 
    P_5(x_0,x_1) = x_0(b_0 x_1^2 + b_1 x_1x_0^2+ b_2 x_0^4)
    = x_1(b_0 x_1x_0 + b_1 x_0^3) + b_2x _0^5.
$     
Therefore, setting $G_0=b_2$ and using the transformation
$$
       [x_0,x_1,x_2]\mapsto[x_0,x_1,x_2-\frac 14(b_0x_1x_0 - b_1x_0^3) ]
$$
we can reduce the defining of $C$ to the form
\begin{equation}
     x_2^4 x_1 + G_0 x_2^3x_0^5 + x_2^2 P_8 + x_2 P_{11}  + P_{14}
     \label{eq:123-2}
\end{equation}
To obtain the normal form \eqref{eq:nf-123}, we now observe that since
$x_0$ has degree 1 while $x_2$ has degree 2, we can write $P_8=G_4(x_0^2,x_1)$,
$P_{11}=x_0 G_5(x_0^2,x_1)$ and $P_{14}= G_7(x_0^2,x_1)$ where now the $G_j$ are ordinary polynomials of degree $j$.

\par To finish the proof of (1), we note that  the   given set of automorphisms clearly
act on the normal form \eqref{eq:nf-123}.  On the other hand, to obtain the
reduction \eqref{eq:123-1} we must use a combination of transformations of
the form $x_1\mapsto a_1x_1 + a_2x_0^2$ and $x_2\mapsto a_3x_2$.  This fixes
$a_2$ and the product $a_1a_3$.  Likewise, the reduction \eqref{eq:123-2}
fixes the coefficients $a_4$ and $a_5$.
\\
(2) This is a bit more involved. As in case (1) we write
$F_C=
      x_2^3x_0 + x_2^2 P_8 + x_2 P_{15} + P_{22}
$ 
where $P_j=P_j(x_0,x_1)$ is weighted homogeneous of degree $j$ and rewrite this equation
as
$ 
    F_C  = x_2^3x_0 + x_2^2 G_4(x_0^2,x_1) + x_2x_0G_7(x_0^2,x_1) + G_{11}(x_0^2,x_1)
$ in terms of ordinary degree $j$ polynomials $G_j$ in $x_0^2$ and $x_1$.
  Since the coefficient $b_4$ in 
$ 
         G_4(x_0^2,x_1) = b_0x _0^8 + b_1x_0^6x_1 + b_2x_0^4x_1^2
                         +b_3x_0^2x_1^3 + b_4x_1^4
$  is non-zero, using using an automorphism of
$\bP(1,2,7)$ of the form  $x_1\mapsto x_1 + \beta x_0^2$, we may assume 
that the coefficient of $x_0^8$ equals  $3\lambda $, where  $\lambda^3$ is  the coefficient of $x_0^{22}$.  
 In this way, we obtain 
\[
F_C = x_2^3x_0 + x_2^2(3\lambda  x_0^8 + x_0^2x_1q_2(x_0^2,x_1) + G_0x_1^4)
              + x_2x_0q_7(x_0^2,x_1) + G_{11}(x_0^2,x_1)  ,
\]
where $G_0=b_4$.  Next, we consider the automorphism
$$
     x_2\mapsto x_2 - x_0x_1 G_2(x_0^2,x_1)/3 - \lambda x_0^7.
$$
Then, $ x_0x_2^3$ transforms into 
$ x_0x_2^3
        -x_2^2\left(3 \lambda x_0^8 + x_0^2x_1q_2(x_0^2,x_1)\right)
      +x_2(\cdots) - \lambda^3 x_0^{22} + x_1(\cdots) 
$
and $F_C$ becomes
$$
F_C = x_0 x_2^3 + G_0   x_1^4 x_2^2
              +  x_0 x_2  G_7(x_0^2,x_1) + G_{11}(x_0^2,x_1),       
$$ 
where now the coefficient of $x_0^{22}$ is zero.
\par Finally, the given transformations preserve the normal form, and
unipotent mixing of the variables destroys the given normal form.
\par
The last assertion follows by considering the relations imposed on the coefficients of $F_C$ if $(c_0,c_1,c_2)\in (\bC^*)^3$
fixes each of them.
\end{proof}

It is clear that the statement of Lemma~\ref{lem:37AND27} implies Proposition~\ref{prop:GIT}, parts (a) and (d).
 
 \par Now we consider the two cases which are not double covers. The next lemma implies 
Proposition~\ref{prop:GIT}, parts (b) and (c).
 
\begin{lemma}  (1) In case $(a,b)= (3,5)$ via the automorphism group of
$\bP(1,2,3,5)$, the defining equation of $F$ can be put in the form
\begin{align}
     F &=  x_1  x_3^2+ x_0 G_2(x_0^3,x_2) x_3  + G_0 x_2^4
   +  \nonumber \\
   & \hspace{3em} G_3(x_0^2, x_1)x_2^2 + G _4(x_0^2,x_1)x_0x _2 + G_6(x_0^2,x_1) , \quad G_0\not=0,
      \label{eq:nf-1235}
\end{align}
where $G_j$ is an ordinary polynomial of degree $j$  in two variables.
The subgroup of $\text{\rm Aut}\,\bP(1,2,3,5)$
preserving  a normal form of the type   \eqref{eq:nf-1235} consists of transformations
of the form $x_j\mapsto c_jx_j$ for $c_j\in\bC^{*}$ with
$c_1c_3^2=1$.   
\\
(2) In case $(a,b)=(5,7)$   via the automorphism group of
$\bP(1,2,5,7)$, the defining equation of $F$ can be put in the form
\begin{equation}
      F=   x_1 x_3^2 + x_0^2 G_1(x_0^5,x_2) x_3
            + r_0 x_0 x_2^3  + G_0 x_1^3  x_2^2
                         + x_0 x_2 G_5(x_0^2,x_1) + G_8(x_0^2,x_1) ,
      \label{eq:nf-1257}
\end{equation}
where $G_j$ is an ordinary polynomial of degree $j$  in two variables and
$r_0$ is a non-zero constant. The subgroup of
$\text{\rm Aut}\,\bP(1,2,5,7)$ acting on the normal form
\eqref{eq:nf-1257} consists of transformations
of the form $x_j\mapsto c_jx_j$ for $c_j\in\bC^*$ with
$c_1c_3^2=1$.   
\par
In both cases  the stabilizer of $F $ is generically the identity.
\end{lemma}
\begin{proof}  (1) The surface $X$ will pass through the singular point $[0,0,0,1]$ and so the
monomial $x_3^2x_1$ must therefore appear with non-zero coefficient in $f$,
otherwise $\nabla F=0$ at $[0,0,0,1]$.  We can therefore write
$ 
 F = x_3^2p_2(x_1,x_0) + x_3 P_7(x_0,x_1,x_2) + P_{12}(x_0,x_1,x_2)
 $,
where $P_2(x_1,x_0)=\alpha_1 x_1 + \alpha_2x_0^2$ with $\alpha_1\neq 0$.
Therefore, using the transformation $x_1\mapsto \alpha_1x_1+\alpha_2x_0^2$
we can reduce the defining equation of $X$ to
\begin{equation*}
  F= x_1  x_3^2+ x_3 P_7(x_0,x_1,x_2) + P_{12}(x_0,x_1,x_2)
 \end{equation*}
(of course, this changes $P_7$ and $P_{12}$ as well). 
\par
We next simplify $P_{12}$.
Note that if  the coefficient  of $x_2^4$ is zero, $X$     passes  through the singular point  $[0,0,1,0]$ of
$\bP[1,2,3,5]$ and $\nabla F=0$ at $[0,0,1,0]$  which violates  the assumption that $X$ be quasi-smooth.
 Thus we can write
$
P_{12}=q_0 x_2^4 + b_3(x_0,x_1) x_2^3 + b_6(x_0,x_1)x_2^2 +b_9(x_0,x_1)x_2
          + b_{12}(x_0,x_1)
$ 
which can be rewritten as   $
     P_{12}=G_0 x_2^4 + G_1(x_0^2,x_1) x_0 x_2^3
           + G_3(x_0^2,x_1)x_2^2 +G_4(x_0^2,x_1)x_0x_2
          + G_6(x_0^2,x_1)$  where each  $G_j$  is an  ordinary polynomials of degree $j$.
Finally, using the transformation $x_2\mapsto x_2 -x_0G_1(x_0^2,x_1)/4G_0$ we
can obtain the simplified form
\[ 
  P_{12} = G_0 x_2^4 + G_3(x_0^2,x_1)x_2^2 +G_4(x_0^2,x_1)x_0x_2
  + G_6(x_0^2,x_1),
\]
 possibly changing $G_3$, $G_4$ and $G_6$.   
 
 The previous transformation of $x_2$ will also have changed $P_7$ which we subsequently simplify as follows.
Using a transformation of the form $x_3\mapsto x_3+P_5(x_0,x_1,x_2)$
we can remove all of the monomials from $P_7$ which are divisible by $x_1$, i.e. $P_7$ becomes  
\[
P_7=   x_0 P_6(x_0,x_2)  ,\quad \deg P_6=6,
\]
since $x_2$ has degree $3$. This  finally brings  $F$  in  the desired form.
\\
(2)
 Using that the surface $X$ passes through
the point $[0,0,0,1]$ we deduce that $F$
 must contain the monomial $x_3^2x_1$ so that  the  defining equation has the form
\[ 
   F= x_3^2P_2(x_0,x_1) + x_3 P_9(x_0,x_1,x_2) + P_{16}(x_0,x_1,x_2),
\]
where $P_2(x_0,x_1) = \alpha_0 x_0^2 + \alpha_1 x_1$ with $\alpha_1\neq 0$.
Therefore, using the transformation $x_1\mapsto \alpha_0x_0^2 + \alpha_1x_1$ we can assume that
\[ 
P_2(x_0,x_1)= x_1.
\]
Next, we use a transformation of the form $x_3\mapsto x_3 + P_7(x_0,x_1,x_2)$
to eliminate all of the terms of $P_9(x_0,x_1,x_2)$ which are divisible by
$x_1$ so that 
\[
 P_9= P_9( x_0 ,x_2) =x_0^4 G_1(x_0^5, x_2) ,
\] 
with $G_1$   an ordinary polynomial of degree 1 in two variables. Note that  have potentially
changed $P_{12}$ which   now will be written as
 \[
 P_{12}=G_{12} (x_0,x_1,x_2).
 \]
Finally, we consider  $P_{16}(x_0,x_1,x_2)$  which  must contain $x_2^3x_0$ to avoid creating a
singularity at $[0,0,1,0]$.  Thus, we can write
$ 
     P_{16}(x_0,x_1,x_2)
     = r_0 x_2^3x_0 + x_2^2  P _6(x_0,x_1) + x_2P_{11}(x_0,x_1) + R_{16}(x_0,x_1)
$      
which  can  be rewritten as 
 $ 
     P_{16}(x_0,x_1,x_2) = r_0 x_0 x_2^3 + x_2^2 G_3(x_0^2,x_1) 
                         + x_0 x_2 G_5(x_0^2,x_1) + G_8(x_0^2,x_1)
$,
where $G_j$ is an ordinary polynomial of degree $j$ and $r_0$ is a constant.
Using the transformation $x_2\mapsto x_2 - x_0q_2(x_0^2,x_1)/(3r_0)$ we
can remove all of the terms of $x_2^2G_3(x_0,x_2)$ which are
divisible by $x_0^2$, i.e.\ all terms except $x_1^3x_2^2 $.
In other words,
\[
        P_{16}(x_0,x_1,x_2) = r_0 x_0 x_2^3 + G_0 x_1^3 x_2^2
                         + x_0 x_2G_5(x_0^2,x_1) + G_8(x_0^2,x_1)
\]
which brings $F$ in the required shape.
The group of substitutions which preserve  this form is given by
$x_j\mapsto a_j x_j$ for $a_j\in\bC^*$, where $a_1a_3^2=1$ to
keep the coefficient of $x_3^2x_1$ equal to 1.  
\par
The last assertion follows by considering the relations imposed on the coefficients of $F_C$ if
$(c_0,c_1,c_2,c_3)\in (\bC^\times)^4$
fixes each of them.
\end{proof}

\section{The Picard number of generic type  2  members.\\
By Wim  Nijgh} \label{app:nijgh}

\begin{small} \textbf{Acknowledgements}. 
The author would like to thank Ronald van Luijk for all his input and feedback on this work. 
The author would also like to thank Chris Peters for the opportunity to work on this problem and the interesting conversations we had on this topic. 

\medskip 
Let $k$ be an arbitrary field and let~$\overline{k}$ be an algebraic closure of~$k$.
For any variety~$Y$ over~$k$, we let~$Y_{\overline{k}}$ denote its base change to~$\overline{k}$.
Furthermore, if $Y$ is projective, we denote by~$\NS(Y)$ the Neron-Severi group of~$Y$ and by~$\NS(Y)_\textnormal{tor}$ its torsion subgroup. 
We denote by~$\rho(Y)$ the Picard number of~$Y$, which is the rank of~$\NS(Y)$.
\end{small}

\subsection{Overview} In the weighted projective space $\bP_k(1,2,3,5)$ with coordinates $x_0,x_1,x_2,x_3$, 
we look at the family of quasi-smooth surfaces of degree~12. 
After some linear transformation, such a surface is given by an equation $F=0$ where
\begin{align}\small
\begin{split} \label{equation} 
    F =\; & x_1 x_3^2 + G'_0 x_0 x_1^3 x_3 + G_{1,1}(x_0^3,x_2) x_0^2 x_1 x_3 + G_{1,2}(x_0^3,x_2) x_1^2 x_3 + G_2(x_0^3,x_2) x_0 x_3 \\
    &+ G_0 x_2^4 + G_{1,3}(x_0^2,x_1)x_0x_2^3 +
    G_3(x_0^2,x_1)x_2^2 + G_4(x_0^2,x_1) x_0 x_2 + G_6(x_0^2,x_1),
\end{split}
\end{align}
such that $G_0,G'_0\in k$, each~$G_{1,i}$ is homogeneous of degree~$1$ and each~$G_i$ is homogeneous of degree~$i$. 
If $\ch(k)\neq 2$, one can assume that $$G'_0=G_{1,1}=G_{1,2}=G_{1,3}=0$$ 
after some linear transformation and obtain the family described in Proposition \ref{prop:GIT}(b). 

Now let  $Y$ be a quasi-smooth surface of degree~12 in~$\bP_k(1,2,3,5)$.
Note the only singular points in~$\bP_k(1,2,3,5)$ are the points $(0:1:0:0)$, $(0:0:1:0)$ and~$(0:0:0:1)$.  
From equation~\ref{equation}, we observe that the point~$(0:0:0:1)$ is always contained in the surface~$Y$. 
If the coefficient of the monomial~$x_1^6$ in~$G_6$ is non-zero, then $(0:1:0:0)$ is not on the surface $Y$, 
and if $G_0\neq 0$, then~$(0:0:1:0)$ is not on $Y$.

From now on we assume that we are in the general case where indeed the points $(0:1:0:0)$ 
and $(0:0:1:0)$ are not on $Y$. Let~$Y'$ be a minimal desingularization of~$Y$.
The following lemma shows that we can obtain $Y'$ from a blowup in the point $(0:0:0:1)$.  

\begin{lemma}\label{DE}
    Suppose that $\ch(k)\neq 5$ and let $Y$ be as above. 
    Then the blowup of~$Y$ in $(0:0:0:1)$ gives a minimal desingularization of~$Y$. 
    The exceptional locus contains two rational curves, i.e., 
    each curve is isomorphic to $\bP^1$, which are both defined over $k$. 
    The self-intersection number of these curves equal~$-2$ and~$-3$ and they intersect each other transversally in one point.
\end{lemma}
\begin{proof}
    We can generalize the proof of Proposition~\ref{prop:InsAndSIngs}(b) to deduce that the point~$(0:0:0:1)$ is a quotient singularity of type $\tfrac{1}{5}(1,3)$. 
    The procedure of resolving this singularity generalizes to fields~$k$ with $\ch(k)\neq 5$, see \cite[Proposition 2.5]{Got96}, and the desired results all follow. 
\end{proof}

From this observation, we deduce the following result.

\begin{corollary} \label{rho3}
    Let $Y'$ be a minimal desingularization of a quasi-smooth surface of degree~12 in $\bP_k(1,2,3,5)$. Then we have $\rho(Y')\geq 3$. 
\end{corollary}
\begin{proof}
    The strict transform of the hyperplane section given by the equation $x_0=0$ and the two curves obtained from the blow-up are  blue  linearly  independent from each other in $\NS(Y')$ and are all non-torsion, cf. Corollary~\ref{cor:OnPic}(b).
\end{proof}

These notes aim to prove that for a field of characteristic 0, and for a general enough choice, the geometric Picard number $\rho(Y'_{\overline{k}})$, 
and hence also the Picard number $\rho(Y')$, equals~3. 
We will do this by showing that it holds for the surface of Definition~\ref{surf}.

\begin{definition}\label{polF}
    We define $F\in \bZ[x_0,x_1,x_2,x_3]$ to be the polynomial given as in equation~\eqref{equation}, with \begin{align*}
        G'_0=G_{1,1} &= G_{1,3} = 0, \quad  \,       G_{1,2}(x_0^3,x_2)  = 3x_2; \\
        G_2(x_0^3,x_2) &= x_2^2 + x_0^3 x_2, \quad G_0  = -1; \\
        G_3(x_0^2,x_1) &= x_0^6 + 2x_0^4x_1+ x_0^2 x_1^2+2x_1^3; \\
        G_4(x_0^2,x_1) &= 4x_0^6x_1+2x_0^4x_1^2+x_0^2 x_1^3+4x_1^4; \\
        G_6(x_0^2,x_1) &= x_0^{12} + 3x_0^{10} x_1 + 3x_0^8 x_1^2 + x_0^4  x_1^4 + 3x_0^2 x_1^5 + x_1^6.
    \end{align*}
\end{definition}

\begin{definition}\label{surf}
    We define~$X$ to be the degree~12 surface in~$\bP_{\bQ}(1,2,3,5)$ given by~$F=0$. 
    We define~$X'$ over~$\bQ$ as the surface obtained by the blowup of $X$ in the point $(0:0:0:1)$.
\end{definition}

\begin{theorem}\label{thm}
     The surface $X'$ is smooth and   $\rho(X')=\rho(X_{\overline{\bQ}}')=3 $. 
\end{theorem}

The proof of Theorem~\ref{thm} can be found in~\ref{proof}. The proof uses a similar method as described in the proof of~\cite[Theorem~3.1]{luik2} and in \cite[Section 4]{kloosterman}. 
We will look at good reductions of this surface over~$\bbF_2$ and over~$\bbF_3$, denoted~$X_2'$ and~$X_3'$, respectively, and show that
(i) $\rho((X_2')_{\overline{\bbF}_2}),\rho((X_3')_{\overline{\bbF}_3})\leq 4$ and 
(ii) the discriminants of the geometric Neron-Severi lattices of $X_2'$ and $X_3'$ do not differ by a square factor. 
We will see that this implies that $\rho(X'_{\overline{\bQ}})$ is at most~$3$. 

To calculate the discriminants (up to a square factor) of these Neron-Severi lattices, we will use the Artin-Tate formula. 
This, together with a result about finding upper bounds for the Picard number, will be discussed in~\ref{ATBg}. 

Next, we will define the surfaces of good reduction and determine the characteristic polynomial of Frobenius acting on some cohomology group. 
This characteristic polynomial will give the upper bound for the Picard number, and together with the Artin-Tate formula, it will give the necessary information we need in order to prove Theorem~\ref{thm}. 
This work will be done in~\ref{goodred}. 

Some of the proofs in~\ref{goodred} are based on computations which are done in \textsc{SageMath}. 
The code which is used can be found in~\ref{a}.

\subsection{The Neron-Severi group for varieties over finite fields}\label{ATBg}

In this section, we recall some known results for the Neron-Severi group for varieties over finite fields. These results will be used in the proof of Theorem~\ref{thm} and some of the intermediate results in~\ref{goodred}.

Assume that $k$ is a finite field. 
Set $p:=\ch (k)$ and $q:=\#k$.
Let~$Y$ denote any projective, smooth and geometrically connected surface over~$k$. Define $$\alpha(Y):= \chi(Y,\mathcal{O}_Y)-1+\dim(\Pic_{Y/k}).$$  

Let~$\ell\neq p$ be any other prime. The absolute Galois group of~$k$, which we will denote by~$\Gal(\overline{k}/k)$ and which is generated by Frobenius, acts on the geometric Neron-Severi group $\NS(Y_{\overline{k}})$ as well as on the second cohomology group $\tH^2_\text{\'et}(Y_{\overline{k}},\bQ_\ell(1))$. 
We let~$\Frob_q$ denote the linear map induced by Frobenius on $\tH^2_\text{\'et}(Y_{\overline{k}},\bQ_\ell(1))$
and let~$\varphi$ denote the characteristic polynomial of~$\Frob_q$.

\begin{proposition}\label{inclNSH}
    There is an inclusion 
    $$\NS(Y_{\overline{k}})\otimes \bQ_\ell(1)\hookrightarrow \tH^2_\text{\'et}(Y_{\overline{k}},\bQ_\ell(1))$$
    of finite-dimensional vector spaces that respects the Galois action. 
\end{proposition}
\begin{proof}
    See \cite[Proposition 6.2]{luik1}
\end{proof}

\begin{corollary}\label{NSeq}
    Identify $\NS(Y)$ as a subset of $\NS(Y_{\overline{k}})$. 
    Then the following holds.
    \begin{itemize}
        \item[(i)] Under the embedding of Proposition~\ref{inclNSH}, we have the equality 
            $$\NS(Y)\otimes \bQ_\ell(1)=\NS(Y_{\overline{k}})\otimes \bQ_\ell(1)\cap \tH^2_\text{\'et}(Y_{\overline{k}},\bQ_\ell(1))^{\Gal(\overline{k}/k)}.$$
        \item[(ii)] If $r$ denotes the multiplicity of the eigenvalue 1 of $\Frob_p$, then for the Picard number of $Y$, we have $\rho(Y)\leq r$. 
        \item[(iii)] The number of eigenvalues, counted with multiplicity, of $\Frob_p$ which are a root of unity, is an upper bound for $\rho(Y_{\overline{k}})$.
        \item[(iv)] The Tate conjecture holds for $Y$ if and only if the upper bounds in (ii) and~(iii) are exactly the Picard numbers of the surfaces~$Y$ and~$Y_{\overline{k}}$, respectively. \qed
    \end{itemize}
\end{corollary}

\begin{remark}\label{evenub}
    For the surfaces we study, we have $\dim \tH^2_\text{\'et}(Y_{\overline{k}},\bQ_\ell(1))=22$, see Proposition~\ref{prop:InsAndSIngs}(b). 
    In particular, because $22$ is even, we have as a corollary of the Weil conjectures, that in our case the upper bounds given in Corollary~\ref{NSeq} will be even. 
\end{remark}

Remark~\ref{evenub} is the reason that we compare the reduction at two different primes in the proof of Theorem~\ref{thm}. We use the following result to make this comparison.

\begin{lemma}[Artin-Tate formula]\label{AT}
    Suppose the Tate conjecture holds for~$Y$. 
    Then the group~$\Br(Y)$ is finite, and 
    $$\lim_{t\to 1} \frac{\varphi(t)}{(t-1)^{\rho(Y)}}=\frac{\# \Br(Y)\cdot \disc (\NS(Y)/\NS(Y)_\textnormal{tor})}{q^{\alpha(Y)}(\#\NS(Y)_\textnormal{tor})^2}.$$
\end{lemma}
\begin{proof}
    See \cite[Theorem 5.2]{Tat66}. 
    \end{proof}

\begin{corollary}\label{NSuptosq}
    Suppose the Tate conjecture holds for~$Y$. Then the discriminant of the Neron-Severi lattice $\NS(Y)/\NS(Y)_\textnormal{tor}$ is up to a square factor equal to $$q^{\alpha(Y)}\cdot \lim_{t\to 1} \frac{\varphi(t)}{(t-1)^{\rho(Y)}}.$$
\end{corollary}
\begin{proof}
    If the Brauer group is finite, its order $\#\Br(Y)$ is a square (see \cite{brauer1} and its corrigendum \cite{brauer2}). 
    With this observation, the result follows directly from Lemma~\ref{AT}.
\end{proof}

\subsection{Good reductions at the primes 2 and 3}\label{goodred}

In this section, fix $p\in\{2,3\}$. 
We will define two surfaces over $\bbF_p$, which will be good reductions for the surfaces~$X$ and~$X'$ of Definition~\ref{surf}, respectively. 

\begin{definition}\label{Xp}
       We define the surface $X_p$ over $\bbF_p$ as the degree~12 surface in~$\bP_{\bbF_p}(1,2,3,5)$ given by~$F=0$, where~$F$ from Definition~\ref{polF} is seen as a polynomial with coefficients in $\bbF_p$. 
    We also define the surface $X_p'$ to be the blowup of $X_p$ in the point $(0:0:0:1)$.
\end{definition}

\begin{lemma}
    The surface~$X_p$ is quasi-smooth and the surface~$X_p'$ is smooth.
\end{lemma}
\begin{proof}
    A direct verification, done in \textsc{SageMath} (see \ref{a.qs}), shows that $X_p$ is quasi-smooth. Because $(0:0:0:1)$ is the only singular point on $X_p$, it follows from Lemma~\ref{DE} that $X'_p$ is smooth.
\end{proof}

Our next aim is to count the number of points on the surface~$X_p'$, which will be used in the proof of Proposition~\ref{frob} to determine the characteristic polynomial of Frobenius.
To do this, we will use an  elliptic fibration  on  the surface $X_p'$ (see \S~\ref{ssec:GenEllFib} for this notion) whose fibers do not contain a~$-1$-curve.
 
The elliptic fibration we will use, is the morphism that is induced by the rational map $\tau\colon X_p\dashrightarrow \bP^1$ defined by  $(x_0:x_1:x_2:x_3)\mapsto (x_0^2:x_1)$. 
The following lemma shows that the map $\tau$ extends to a minimal elliptic fibration $\tau'\colon X_p'\to \bP^1$. 

\begin{lemma}\label{fib}
    The map $\tau\colon X_p\dashrightarrow \bP^1$ extends to a minimal elliptic fibration $\tau'\colon X_p'\to \bP^1$ 
    and for the curves in the exceptional locus, we have that the~$-2$-curve is in the fiber above the point $(1:0)$ and that the~$-3$-curve is a double section for this fibration.
\end{lemma}
\begin{proof}
    We can apply the proof of Proposition~\ref{prop:Main1}(b) to the surfaces~$X_p$ and $X_p'$. 
\end{proof}

Next, we define an arithmetic surface $\mathcal{C}\to \Spec\bbF_p[t]$. We refer the reader to~\cite[Section~IV.4f]{sil94} for the definition and standard results on arithmetic surfaces.

\begin{definition}\label{asC}
    The polynomial $F$  from Eqn.\eqref{equation} defines the arithmetic surface 
    $$\mathcal{C}\subset 
    \Spec \bbF_p[t]\times \bP(1,2,1)$$ 
   as the zero set of  
     $F'=tz^4\cdot F(1,t,x/z,y/tz^2)$, $F'\in \bbF_p[t][x,y,z]$.
\end{definition}

Using the birational map $\mathcal{C}\dashrightarrow X_p$ given by 
 $(t,(x:y:z))\mapsto(1:t:x/z:y/tz^2) $ (and with inverse 
given by 
 $(x_0:x_1:x_2:x_3)\mapsto (x_1/x_0^2,(x_0x_2:x_0x_1x_3:x_0^4))$) 
 induces  an isomorphism 
$$\mathcal{C}\setminus\{tz=0\}\xrightarrow{\sim} X_p\setminus (\{x_0=0\}\cup \{x_1=0\}).$$
With $E\subset X'_p$  the exceptional locus of $X'_p\to X_p$ and $O:=(0:0:0:1)$, 
 we have   $X_p'\setminus E \xrightarrow{\sim} X_p\setminus \{O\} $.

Next,   setting  $t=x_1/x_0^2$, we can identify $\Spec \bbF_p[t]\subset \bP^1$ as a subscheme. This identification makes $U:=X'_p\setminus {\tau'}^{-1}(0:1)$ an arithmetic surface over $\bbF_p[t]$.
Combined with the above observations, we get an embedding 
 $\mathcal{C}\setminus\{tz=0\}\hookrightarrow U$ 
of $\Spec \bbF_p[t]$-schemes. 
The next lemma shows that this extends to an isomorphism.

\begin{lemma}\label{proj_mod}
    The embedding $\mathcal{C}\setminus\{tz=0\}\hookrightarrow U$ above, extends to an isomorphism $\mathcal{C}\xrightarrow{\sim} U$ as $\Spec \bbF_p[t]$-schemes.
\end{lemma}
\begin{proof}
    Note that because $\tau'$ is a minimal elliptic fibration, it follows that $U$ is a minimal proper regular model for its generic fiber as defined in~\cite[Theorem~IV.4.5b]{sil94}.
    We will show that~$\mathcal{C}$ is a minimal proper regular model as well and then the result will follow from~\cite[Theorem~IV.4.5b]{sil94}.
    
    To show this, we first note that this surface is projective over~$\Spec \bbF_p[t]$, and hence proper over~$\Spec \bbF_p[t]$. 
    To check that it is smooth, we note that for each $a\in \overline{\bbF}_p$, the point $(a,(0:1:0))$ does not lie on $\mathcal{C}$. It follows that every point of this surface lies on the affine where $x$ does not vanish or where $z$ does not vanish.
    Now using \textsc{SageMath}, see~\ref{a.surf}, we can check that~$\mathcal{C}$ is smooth over~$\bbF_p$ by checking both affines.
    It follows that $\mathcal{C}$ is regular.
    So it remains to show that this surface~$\mathcal{C}$ is minimal.
    
    Recall that there is an embedding $\mathcal{C}\setminus\{tz=0\}\hookrightarrow U$. 
    Because the fibration on $U$ is minimal, we deduce that the only possible exceptional curves in the fibers of $\mathcal{C}\to \Spec \bbF_p[t]$ can be found at $z=0$ or in the fiber~$t=0$.
    Note that for every fiber $t=a$, it is easy to see that $\mathcal{C}\cap \{t=a,z=0\}$ is $0$-dimensional, cf. the proof of Lemma~\ref{pts_sing_fib}, from which we deduce that every fiber above $t=a\neq 0$ cannot contain an exceptional curve. 
    
    The fiber above~$t=0$ is given by the equation   $y(y+x^2+xz)=0 $ and so it consists of two rational curves $E_1$ and $E_2$ which intersect each other in two points, i.e., 
     as a divisor, the fiber is given by $E_1+E_2$. Because the intersection number of a fiber with every fibral divisor is zero, see~\cite[Proposition~IV.7.3(b) and Remark~IV.7.6]{sil94}, 
     it follows that 
     $E_i\cdot(E_1+E_2)=0$,   $i =1,2$.
    Hence  $E_1^2=E_2^2=-E_1\cdot E_2=-2$. 
    This shows that there are no fibral exceptional curves on~$\mathcal{C}$ and we conclude from~\cite[Remark~IV.7.5.1]{sil94} that~$\mathcal{C}$ is minimal.
\end{proof}

One of the tools we will make use of, is the discriminant related to this arithmetic surface~$\mathcal{C}$. 
For a definition of the discriminant of a weighted homogeneous polynomial, we refer the reader to~\cite[§1.1]{Ter23}.

\begin{lemma}\label{lemdisc}
    Define~$\Delta_p:=\disc F'$. 
    Then we have 
    \begin{align*}
        \Delta_2(t)&=t^2(t^{10} + t^9 + t^8 + t^7 + t^2 + t + 1)(t^{12} + t^8 + t^5 + t^4 + t^3 + t + 1);\\
        \Delta_3(t)&=2t^2(t + 2)(t^9 + 2t^8 + 2t^7 + 2t^6 + t^5 + 2t^4 + t^2 + t + 2) \; \cdot\\
        &\hspace{3cm} (t^{12} + 2t^{10} + t^8 + t^7 + 2t^6 + t^5 + t^2 + 2),
    \end{align*} 
    where the terms in between brackets are irreducible.
\end{lemma}
\begin{proof}
    Recall that the equation of~$\mathcal{C}$ is of the form $y^2+h_t(x,z)y+f_t(x,z)$, where 
    \begin{align*}
        f_t(x,z)&=F'(t,(x,0,z)); \\
        h_t(x,z)&=z^2 (G_{1,2}(1,x/z)t^2+G_2(1,x/z)).
    \end{align*}  
    The formula for the discriminant of such a polynomial is given in~\cite[Example~3.5]{Ter23} combined with~\cite[Lemma~3.3]{Ter17}.
    From this formula, we deduce that the discriminant $\Delta$ of $F'$ over $\bZ[t]$ can be given by
    $$\Delta=4^{-4}\cdot \disc (h_t(x,z)^2-4f_t(x,z))\in \bZ[t],$$
    where~$\disc$ denotes the discriminant of a polynomial of degree~$4$, see~\cite[Chapter~12.1.B (1.35)]{Gel94}.
    
    Using \textsc{SageMath}, see~\ref{a.disc}, we use the above formula to calculate this polynomial~$\Delta$. Then we 
    factor its reduction mod $p$ to obtain the above expressions.
\end{proof}

We will use the discriminant~$\Delta_p$ of Lemma~\ref{lemdisc} to deduce the type of fibers of the fibration~$\tau'\colon X_p'\to \bP^1$.

\begin{lemma}\label{singfib}
    Let $\tau'\colon X'_p\to \bP^1$ be the elliptic fibration as above. 
    Then the fiber above $P\in\bP^1(\overline{\bbF}_p)$ is singular if and only if $P=(1:t_0)$ with~$t_0$ a zero of $\Delta_p$.
    Moreover, if it is a singular fiber, then it is of type~I${}_2$ if $P=(1:0)$ and it is of  type~I${}_1$ otherwise.
\end{lemma}
\begin{proof}
    Let $t_0$ be a zero of~$\Delta_p$. 
    Recall from Lemma~\ref{proj_mod} that the fiber above~${(1:t_0)}$ of~$\tau'$ is isomorphic to the above fiber~$t_0$ of the arithmetic surface~$\mathcal{C}\to \Spec \bbF_p[t]$. 
    We can use Tate's algorithm (see~\cite[Section~IV.8 and~IV.9]{sil94}) to determine the type of fiber above $t_0$ on the arithmetic surface $\mathcal{C}$ as follows.
    
    First, we choose some separable extension of~$\bbF_p(t)$ which is unramified at $t_0$ and such that the base change of $\mathcal{C}$ to this field, gives an arithmetic surface~$\mathcal{C}'$ which has a section. 
    We can then put the defining equation of~$\mathcal{C}'$ in Weierstrass form to apply the algorithm. 
    Now choose~$t_0'$ such that the fiber above $t_0'$ on~$\mathcal{C}'$ gets mapped to the fiber above~$t_0$ on $\mathcal{C}$.
    Because the extension is unramified above~$t_0$, the base change is regular and minimal around~$t_0'$ and so the fiber above~$t_0'$ on~$\mathcal{C}'$ is isomorphic to the fiber above~$t_0$ on~$\mathcal{C}$ are isomorphic over some separable extension, and hence the fiber types are the same. By the defining property of the discriminant, see~\cite[Theorem~1.2]{Ter23}, the valuation of the discriminant of the Weierstrass form at~$t_0'$ will exactly equal the valuation of the polynomial~$\Delta_p$ at~$t_0$. 
    From this we deduce that we can use the valuation of~$\Delta_p$ at~$t_0$ to deduce the type of fiber above~$t_0$.
    
    Note that~$\Delta_p$ has a factor~$t^2$ and that all the other factors are separable. So for $t_0\neq 0$, we have multiplicity~1. Hence, above $(1:t_0)$, we deduce from~\cite[Section~IV.9, Table~4.1]{sil94}, that this is a fiber of type~I${}_1$.

    For $t_0=0$, we have~$v_{0}(\Delta_p)=2$.
    In characteristics~$2$ and~$3$ the order of vanishing of the discriminant will be bigger than~2 in case the fiber has multiplicative reduction due to wild ramification, see~\cite[Proposition~5.1]{SS13}.
    We deduce from the above that the fiber above $(1:0)$ is of type~I${}_2$. 
    
    Because~$\Delta_p$ has degree~24, which equals the Euler characteristic of the surface (see Proposition~\ref{prop:Main1}(b)), we deduce that these are all the singular fibers of this fibration $\tau'\colon X'_p\to \bP^1$ and that all other fibers are smooth.
\end{proof}

\begin{rmq}
    In the proof of Lemma~\ref{singfib}, we deduced that the fiber above $t=0$ is of type I${}_2$ by using the discriminant $\Delta_p$, but we already encountered this fiber in the proof of Lemma~\ref{proj_mod}, from which we also could have concluded that it is of type I${}_2$.
\end{rmq}

We now define the following affine curves, which we will use in Proposition~\ref{pts} to count the points on the surface $X_p'$.

\begin{definition}\label{Ca}
    Set $\Tilde{F}=x_1\cdot F$.
    For each $a \in \bbF_{p^n}$, we define the affine curve~$C_a$ over~$\bbF_{p^n}$ in~$\mathbb{A}^2_{\bbF_{p^n}}(x',y')$ by the equation $\Tilde{F}(1,a,x',y'/a)=0$. We define the curve~$C_\infty$ in~$\mathbb{A}^2_{\bbF_p}(x',y')$ to be the curve given by the equation $F(0,1,x',y')=0$. 
\end{definition}

For the curves $C_a$, we have the following result.

\begin{lemma}\label{pts_sing_fib}
    Let $a\in \bbF_{p^n}$ and let $g\in \bbF_{p^n}[s]$ be given by $g=s^2+s-a$. 
    The number of $\bbF_{p^n}$-points on the fiber above the point $(1:a)$ of~$\tau'\colon X'_p\to \bP^1$ is equal to the number of $\bbF_{p^n}$-points on~$C_a$ plus the number of roots in~$\bbF_{p^n}$ of the polynomial~$g$.
\end{lemma}
\begin{proof}
    By Lemma~\ref{proj_mod}, we have that the fiber above $(1:a)$ of the morphism~$\tau'$ is isomorphic to the fiber above $t=a$ of~$\mathcal{C}$.
    The embedding $$(x',y')\mapsto (a,(x':y':1))$$ embeds the curve $C_a$ into the fiber above $t=a$ of the arithmetic surface~$\mathcal{C}$ and is isomorphic to the affine part of this fiber where $z$ does not vanish.
    So it follows that the number of points on the fiber above $(1:a)$ equals the number of points on $C_a$ plus the number of points on this fiber intersected with $\{z=0\}$. 
    
    Recall that the defining polynomial of $\mathcal{C}$ is given by 
    $$F':=tz^4\cdot F(1,t,x/z,y/tz^2)\in \bbF_p[t][x,y,z],$$
    and that $F'(a,(0,1,0))=1$. 
    It follows that all the points on the intersection of $t=a$ with $z=0$ and the arithmetic surface $\mathcal{C}$ are on the affine where $x$ does not vanish. 
    This means that these points are of the form $(a,(1:s:0))$ such that~$F'(a,(1,s,0))=0$.
    Following the steps defining $F'$, we can deduce that
    $$F'(1,s,0)=s^2+G_2(0,1)s+a\cdot G_0=s^2+s-a,$$ 
    from which the result follows.
\end{proof}

Now we combine the above results to count the points on the surface~$X_p'$.

\begin{proposition}\label{pts}
      The number of~$\bbF_{p^i}$-points on~$X_p'$ is given by the following table.
 
 \begin{center}
    \begin{tabular}{r|r|r}
        $n$ & $\#X_2'(\bbF_{2^n})$ & $\#X_3'(\bbF_{3^n})$ \\
        \hline
        1 & 11 & 17 \\
        2 & 29 & 95 \\
        3 & 65 & 803 \\ 
        4 & 241 & 6767 \\ 
        5 & 1121 & 59477 \\
        6 & 4289 & 532883 \\
        7 & 16769 & 4798097 \\
        8 & 67329 & 43071575 \\
        9 & 264449 & 387431885 
    \end{tabular}
 \end{center}
\end{proposition}
\begin{proof} 
    Let $1\leq n\leq 9$ be given. 
    As mentioned earlier, we will count the $\bbF_{p^n}$-points of $X_p'$ by counting for each point $P\in \bP^1(\bbF_{p^n})$ the number of $\bbF_{p^n}$-points in the fiber of the map $\tau'\colon X_p'\to \bP^1$ and sum their total. 
    Using \textsc{SageMath}, we follow the steps described next and count for each fiber the number of points and add them together. 
    The code can be found in~\ref{a.count}.

    We start with the fiber above $(0:1)$.  By Lemma~\ref{singfib}, we have that it is non-singular. 
    Recall that $X_p'\setminus E \cong X_p\setminus\{(0:0:0:1)\}$, from which it follows that we can identify the affine curve~$C_\infty$ as an affine part of the fiber above $(0:1)$. 
    This curve~$C_\infty$ is given by an equation of the form $y^2+h(x)y=f(x)$ where $f=-F(0,1,x,0)$ and $h=G_{1,2}(0,1,x,0)$.
    A smooth projective closure can be defined by using the function \verb+HyperellipticCurve+ in \textsc{SageMath}. 
    Because this defines a smooth projective curve, of which an affine is isomorphic to an affine part of the fiber above $(0:1)$, it is isomorphic to this fiber. 
    In particular, the amount of $\bbF_{p^n}$-points will be the same and we can count the number of points on this hyperelliptic curve, see Remark~\ref{hypel}.

    Above all the other fibers, i.e. above $(1:a)$, we first check if the curve~$C_a$ is smooth by checking if the discriminant~$\Delta_p$ vanishes, see Lemma~\ref{singfib}. 
    If it is not smooth, we use Lemma~\ref{pts_sing_fib} to count the number of points.
    In the smooth case, we can count the points similarly as in the case for the fiber $(0:1)$ as follows. 
    The defining polynomial of $C_a$ is again of the form $y^2+h(x)y-f(x)$ where $f=-a\cdot F(1,a,x,0)$ and $h=G_{1,2}(1,x)a^2+G_2(1,x)$. 
    Then, we can again define an isomorphic hyperelliptic curve using~$f$ and~$h$ and count the points on this curve.
\end{proof}

\begin{remark}\label{hypel}
    The main reason to use the function \verb+HyperellipticCurve+ in~\textsc{SageMath} in the proof of Proposition~\ref{pts} is that it has a built-in pointing count algorithm which is faster than naive point counting on curves.
\end{remark}

Our next goal is to find the characteristic polynomial of Frobenius acting on the vector space~$H_p:=H^2_\text{\'et}((X_p)_{\overline{\bbF}_p},\bQ_\ell(1))$. We will denote by $\Frob_p$, the linear map on~$H_p$ that is induced by the Frobenius morphism on~$\overline{\bbF}_p$ and by $f_p$ the characteristic polynomial of $\Frob_p$. 
    
\begin{proposition}\label{frob}
    The characteristic polynomial  of $\Frob_p$ equals  $f_p=(t-1)^4\cdot h_p $,  where $h_p$ is irreducible and equals
    \begin{align*}
        h_2(t) &=  t^{18} + t^{17} + t^{16} + 2t^{15} + 3t^{14} + 3t^{13} 
        + \tfrac{7}{2}t^{12} + \tfrac{9}{2} t^{11} + \tfrac{9}{2} t^{10} \\
        & \hspace{1cm} + \tfrac{9}{2} t^9 + \tfrac{9}{2} t^8 + \tfrac{9}{2} t^7 
        + \tfrac{7}{2} t^6 + 3t^5 + 3t^4 + 2t^3 + t^2 + t + 1;\\
        h_3(t) &= t^{18} + \tfrac{5}{3} t^{17} + \tfrac{8}{3} t^{16} + \tfrac{10}{3} t^{15} + 4 t^{14} + \tfrac{14}{3} t^{13} 
        + \tfrac{16}{3} t^{12} + \tfrac{16}{3} t^{11} + \tfrac{16}{3} t^{10} \\
        & \hspace{1cm} + \tfrac{16}{3} t^9 + \tfrac{16}{3} t^8 + \tfrac{16}{3} t^7 
        + \tfrac{16}{3} t^6 + \tfrac{14}{3} t^5 + 4 t^4 + \tfrac{10}{3} t^3 + \tfrac{8}{3} t^2 + \tfrac{5}{3} t + 1.
    \end{align*}
\end{proposition}
\begin{proof}
    Computations in this proof are done in~\textsc{SageMath}, see~\ref{a.pol}.

    Recall from Proposition~\ref{inclNSH}, that there is an inclusion
    $$\NS((X_p')_{\overline{\bbF}_p})\otimes \bQ_\ell(1)\hookrightarrow H_p$$
    respecting the Galois action. From Corollary~\ref{rho3}, we deduce that there is a subspace~$U_p$ of~$H_p$ of dimension~3 on which Frobenius is acting trivially. 

    Let $V_p$ denote the quotient space $V_p=H_p/U_p$. Because Frobenius leaves $U_p$ invariant, we get an induced action on $V_p$, denoted by $\overline{\Frob}_p$, 
    and we have the relation 
    $ f_p(t)=(t-1)^3\cdot g_p(t)$,     where~$g_p$ denotes the characteristic polynomial of~$\overline{\Frob}_p$. \
    
    From Proposition~\ref{prop:InsAndSIngs}(b), we know that $\dim H_p=22$. It follows that
    $$\dim V_p=\dim H_p-\dim U_p=22-3=19,$$
    and so the polynomial $g_p$ has degree 19. 
    Moreover, it is equal to
    $$g_{p}(t)=c_0t^{19}+c_1t^{18}+\dots + c_{18}t + c_{19},$$
    where $c_0=1$, where $c_1=-\Tr(\overline{\Frob}_p)$, and where the other~$c_i$ are given recursively by Newton's identity (see \cite[§26, Exercise~3]{waerden})
        \begin{equation*}
            c_i=-\frac{\Tr(\overline{\Frob}_p^i)+\sum_{j=1}^{i-1}c_j\Tr(\overline{\Frob}_p^{i-j})}{i}.
        \end{equation*}
    The functional equation gives us that $t^{19} g(1/t)=\pm g(t)$ and so we either have $c_i=c_{19-i}$ for all $0\leq i\leq 9$ or $c_i=-c_{19-i}$ for all $0\leq i\leq 9$. 
    
    From the Lefschetz Trace formula, the Weil conjectures, the relation 
    $$\Tr(\overline{\Frob}_p^i)=\Tr(\Frob_p^i)-3$$
    and the fact that the eigenvalues of $\Frob_p^i$ on $H_p$ differ by a factor $p^i$ from the eigenvalues of Frobenius acting on $\tH^2_\text{\'et}((X_p')_{\overline{\bbF}_p},\bQ_\ell)$, we deduce the equality
    \begin{equation*}\label{eq:trace}
        \Tr(\overline{\Frob}_p^i)=\frac{\#X_{p}'(\bbF_{p^i})-1-p^{2i}-3p^i}{p^i}.
    \end{equation*}

    Using the above formula, Proposition~\ref{pts} and Newton's identity mentioned above, we get the following values of~$c_i$ for $1\leq i\leq 9$ for both $p=2,3$.
    \begin{center} 
    \begin{tabular}{r|lllllllll}
        & $c_1$ & $c_2$ & $c_3$ & $c_4$ & $c_5$ & $c_6$ & $c_7$ & $c_8$ & $c_9$ \\ \hline
        $p=2$ & 0 & 0 & 1 & 1 & 0 & $\frac{1}{2}$ & 1 & 0 & 0 \\
        $p=3$ & $\frac{2}{3}$ & 1 & $\frac{2}{3}$ & $\frac{2}{3}$ & $\frac{2}{3}$ & $\frac{2}{3}$ & 0 & 0 & 0 
    \end{tabular}
    \vspace{0.2cm}
    \end{center}
    
    The above table gives us two options for the polynomials $g_p$, which depend on the sign of the functional equation. 
    It follows from the Weil conjectures that~$g_p$ should have all roots on the unit circle.
    By using \textsc{SageMath}, we can exclude, for both $p=2$ and $p=3$, the one for which the sign of the functional equation is positive because this polynomial has roots outside the unit circle. 
    Moreover, using a factorization algorithm in \textsc{SageMath}, we can calculate that~$g_p$ contains a factor $t-1$ and an irreducible factor of degree 18 equal to $h_p$. From this, the statement follows.
\end{proof}

\begin{corollary}\label{rho4}
    For the Picard number of $X_p'$, we have  $\rho(X_p')=\rho((X_p')_{\overline{\bbF}_p})\leq 4 $.
\end{corollary}
\begin{proof}    
    Because the minimal polynomial of every root of unity has integral coefficients, it follows that the characteristic polynomial of Frobenius~$f_p$, which we found in Proposition~\ref{frob}, has no other roots that are a root of unity except for~$1$. This means that $\rho(X_p')=\rho((X_p)_{\overline{\bbF}_p})$ and the upper bound follows from Corollary~\ref{NSeq}.
\end{proof}

\subsection{Proof of the main result, Theorem~\ref{thm}}\label{proof}

Let $\bZ_{(6)}$ denote the localization of $\bZ$ away from the ideal~$(6)$. Define the scheme $\mathcal{X}$ over~$\bZ_{(6)}$ to be the blow-up of the scheme $\Proj(\bZ_{(6)}[x_0,x_1,x_2,x_3]/(F))$ at the ideal $I=(x_0,x_1,x_2)$. 
Because blow-ups commute under flat morphisms (see \cite[\href{https://stacks.math.columbia.edu/tag/0805}{Lemma 0805}]{stacks-project}), 
 the reduction of~$\mathcal{X}$ at  the prime~$2$ is isomorphic to~$X_2'$, the reduction at  the prime~$3$ is isomorphic to~$X_3'$, 
 and the generic fiber is isomorphic to~$X'$. From this observation, we conclude that the surface $X'$ is smooth as well, because it has a smooth reduction.

By the proof of~\cite[Proposition~6.2]{luik1}, 
 $ N:=\NS(X_{\overline{\bQ}}')/\NS(X_{\overline{\bQ}}')_\text{tor}$ 
embeds for both $p=2$ and $p=3$ into the lattice
 $$
 N_p:=\NS ((X_p')_{\overline{\bbF}_p})/\NS((X_p')_{\overline{\bbF}_p})_\text{tor}.
 $$
By Corollary~\ref{rho4}, we have that  $\rho(X_{\overline{\bQ}}')\leq \rho((X_p')_{\overline{\bbF}_p})\leq 4$. 
If the Tate conjecture does not hold for $X_2'$ or $X_3'$, then it follows from Corollary~\ref{NSeq} that $\rho(X_{\overline{\bQ}}')\leq 3$, and hence with Corollary~\ref{rho3} that 
 $\rho(X')=\rho(X'_{\overline{\bQ}})= 3, $  
and we would be done.

So assume for the remainder that the Tate conjecture holds for both surfaces~$X_2'$ and~$X_3'$. 
Combining Corollary~\ref{NSeq} and Corollary~\ref{rho4}, we find for both $p=2$ and $p=3$ the equality  $\rho(X_p')=\rho((X_p')_{\overline{\bbF}_p})=4$. 
Because $\Br(\bbF_p)=0$, it follows that the Neron-Severi lattice of $X_p'$ equals the lattice~$N_p$ for both $p=2$ and $p=3$.

Using Corollary~\ref{NSuptosq}, we have that
\begin{align*}
        \disc N_2&=s_2^2\cdot 2^{\alpha(X_2')}\cdot h_2(1) = s_2^2\cdot 2^{\alpha(X_2')} \cdot \frac{103}{2};\\
        \disc N_3&=s_3^2\cdot 3^{\alpha(X_3')}\cdot h_3(1) = s_3^2\cdot 3^{\alpha(X_3')} \cdot 72,
\end{align*}
for some $s_2,s_3\in \bQ$. 
We deduce that the discriminants of $N_2$ and $N_3$ do not differ by a square factor. 

From the theory of lattices, we know that the discriminant of a full-rank sublattice always differs by a square factor from the discriminant of the full lattice.
It follows that $N$ cannot be embedded in both the lattices~$N_p$ as a full-rank sublattice. We deduce  $\rho(X_{\overline{\bQ}}')< \rho((X_p')_{\overline{\bbF}_p})=4  $.
Combining this with Corollary~\ref{rho3}, gives the result of Theorem~\ref{thm}. 
\qed

\begin{remark}\label{rmk:OnTypeA}
  (1)  In the proof of Theorem~\ref{thm}, one can use the valuation at the prime 103 to conclude that the discriminants do not differ a square factor. So instead of using Corollary~\ref{NSuptosq}, we could also have used Lemma~\ref{AT} with the original, slightly weaker, result of Tate, \cite[Theorem~5.1]{Tat66}, which states that the Brauer group is a square or two times a square.
\\
  (2) One can apply the above methods to find that for a minimal desingularization of a general member of the family of quasi-smooth degree 14 surfaces in the weighted projective space $\bP_k (1, 2, 3, 7)$ with $\ch k = 0$, the Picard number equals 2, cf. Corollary~\ref{cor:OnPic}(a).
    We give a sketch of the proof here.
    The minimal desingularization $X'$ of a quasi-smooth surface $X$ of degree 14 in $\bP_k (1, 2, 3, 7)$ is given by blowing up in $(0:0:1:0)$ if $\ch k\neq 3$.
    The exceptional locus of the blowup consists of only one curve, a $-3$-curve, which will be a double section for the elliptic fibration.
    From this we deduce the lower bound $\rho(X')\geq 2$.

    Now it suffices to find a surface over $\bbF_2$ with Picard rank at most $2$. 
    Then as in Proposition~\ref{pts} and Proposition~\ref{frob}, we can count the $\bbF_{2^i}$ points and find the characteristic polynomial of Frobenius, with only minor adjustments to the proofs.
    
    If we apply the method to the surface $X_2$ over $\bbF_2$ given by the equation~$F=0$, where
    \begin{align*}
        F &= x_3^2 + G x_3 + x_1 x_2^4 + G_2 x_0 x_2^3 + G_5 x_0 x_2 + G_7 , \text{ and}\\
        G &= x_0 x_2^2 + x_0^4 x_2 + x_1^2 x_2; \quad G_2  = x_0^2 x_1;\\
        G_5 &= x_0^8 x_1 + x_0^2 x_1^4 + x_1^5; \qquad G_7  = x_0^{14} + x_0^{12} x_1 + x_0^{10} x_1^2 + x_0^6 x_1^4 + x_0^2 x_1^6 + x_1^7,
    \end{align*}
    one can find that the surface $X_2'$ has Picard rank $2$, see~\ref{a.casea}. 
    From this one can deduce the desired result.
\end{remark}
\end{appendix}

\end{document}